\documentclass[11pt,reqno]{amsart}

\usepackage{color}
\usepackage[colorlinks=true,allcolors=blue,backref=page]{hyperref}

\usepackage{amsmath, amssymb, amsthm}
\usepackage{mathrsfs}
\usepackage{mathtools}
\usepackage[noabbrev,capitalize,nameinlink]{cleveref}
\crefname{equation}{}{}
\usepackage{fullpage}
\usepackage[noadjust]{cite}
\usepackage{graphics}
\usepackage{pifont}
\usepackage{tikz}
\usepackage{bbm}
\usepackage[T1]{fontenc}

\usetikzlibrary{arrows.meta}

\usepackage{environ}
\usepackage{framed}
\usepackage{url}
\usepackage[linesnumbered,ruled,vlined]{algorithm2e}
\usepackage[noend]{algpseudocode}
\usepackage[labelfont=bf]{caption}
\usepackage{cite}
\usepackage{framed}
\usepackage[framemethod=tikz]{mdframed}
\usepackage{appendix}
\usepackage{graphicx}
\usepackage[textsize=tiny]{todonotes}
\usepackage{tcolorbox}
\usepackage{enumerate}
\allowdisplaybreaks[1]
\usepackage{enumerate}

\crefname{algocf}{Algorithm}{Algorithms}

\crefname{equation}{}{} 
\AtBeginEnvironment{appendices}{\crefalias{section}{appendix}} 

\usepackage[color,final]{showkeys} 

\colorlet{refkey}{orange!20}
\colorlet{labelkey}{blue!30}

\crefname{algocf}{Algorithm}{Algorithms}

\numberwithin{equation}{section}
\newtheorem{theorem}{Theorem}[section]
\newtheorem{proposition}[theorem]{Proposition}
\newtheorem{lemma}[theorem]{Lemma}

\crefname{claim}{Claim}{Claims}

\newtheorem{corollary}[theorem]{Corollary}

\newtheorem*{question*}{Question}

\theoremstyle{definition}
\newtheorem{definition}[theorem]{Definition}

\newtheorem*{definition*}{Definition}

\theoremstyle{remark}
\newtheorem*{remark}{Remark}

\newcommand{\snorm}[1]{\lVert#1\rVert}

\newcommand{\imod}[1]{(\mathrm{mod}~#1)}
\newcommand{\qbinom}{\genfrac{[}{]}{0pt}{}}

\newcommand{\mb}{\mathbb}
\newcommand{\mbf}{\mathbf}
\newcommand{\mbm}{\mathbbm}
\newcommand{\mc}{\mathcal}
\newcommand{\mf}{\mathfrak}
\newcommand{\mr}{\mathrm}
\newcommand{\ol}{\overline}

\newcommand{\Ups}{\Upsilon}
\newcommand{\eps}{\varepsilon}

\allowdisplaybreaks

\newcommand{\lL}{\lambda}

\title{The existence of subspace designs}

\author[A1]{Peter Keevash}

\address{Mathematical Institute, University of Oxford, UK.}

\email{keevash@maths.ox.ac.uk}

\author[A2]{Ashwin Sah}
\author[A3]{Mehtaab Sawhney}
\address{Department of Mathematics, Massachusetts Institute of Technology, Cambridge, MA 02139, USA}
\email{\{asah,msawhney\}@mit.edu}

\thanks{Keevash was supported by ERC Advanced Grant 883810. Sah was supported by the PD Soros Fellowship. Sawhney was supported by the Churchill foundation. Sah and Sawhney were supported by NSF Graduate Research Fellowship Program DGE-2141064. }

\begin{document}

\begin{abstract}
We prove the existence of subspace designs with any given parameters, provided that the dimension of the underlying space is sufficiently large in terms of the other parameters of the design and satisfies the obvious necessary divisibility conditions. This settles an open problem from the 1970s. Moreover, we also obtain an approximate formula for the number of such designs.
\end{abstract}
\maketitle

\section{Introduction}\label{sec:introduction}

A widely circulated problem in the 1970s asked for vector space analogues of combinatorial designs, whereby combinatorial designs could be considered as designs in vector spaces over the `field with one element'. This problem arose during an exciting time in the history of combinatorial designs, when Wilson \cite{Wil75} proved the graph case of the Existence Conjecture
(a problem posed by Steiner in the 19th century, eventually resolved by Keevash \cite{Kee14}). In an early article on the general algebraic problem, Cameron \cite{Cam74} gave his `commentary' on the combined efforts of many researchers, including Petrenjuk, Wilson, Ray-Chaudhuri \cite{BR06}, Noda, Bannai, Delsarte \cite{Del76}, Goethals, and Seidel. Cameron remarked that subspace 1-designs (spreads) are `common', but there were no known non-trivial subspace $t$-designs with $t>1$.

This problem has recently seen considerable progress, following a renewed interest due to its connections with Network Coding (see \cite{GPSV18, Etz13}) and advances in techniques, including computational methods for finding explicit examples and probabilistic methods for obtaining general results. To discuss progress on the problem to date we require the following definitions. Let $\mb{F}_q$ be a finite field of order $q$. Let $\mr{Gr}_q(n,r)$ denote the set of $r$-dimensional subspaces (`$r$-spaces') of the $n$-dimensional vector space $\mb{F}_q^n$. An $(n,s,r,\lL)_q$-design consists of a subset of $\mr{Gr}_q(n,s)$, called blocks, such that each $r$-space is contained in exactly $\lL$ blocks. This definition captures the established meaning of `subspace design' in Combinatorics and in Network Coding, although we remark that there is also a large literature in Theoretical Computer Science on a similar but weaker notion of `subspace design' (replace `exactly' by `at most') introduced by Guruswami and Xing \cite{GX13}.

There are some parallels between the histories of subspace designs and combinatorial designs. Indeed, for combinatorial designs it was a longstanding open problem, resolved by Teirlinck \cite{Tei87}, to show the existence of non-trivial $(n,s,r,\lL)$-designs for all $r$ and some $\lL$ (where `non-trivial' means that $s>r$ and not all $s$-sets are blocks). Similarly, the existence of non-trivial $(n,s,r,\lL)_q$-designs for all $r$ and some $\lL$ was a longstanding open problem, resolved much more recently by Fazeli, Lovett, and Vardy \cite{FLV14}. This general result was preceded by various explicit constructions; for details of these we refer to the survey by Braun, Kiermaier, and Wassermann \cite{BKW18}. While Teirlinck used an explicit construction, the construction in \cite{FLV14} is probabilistic (adapting a method of Kuperberg, Lovett, and Peled \cite{KLP17}), and requires $\lL \ge q^{Crn}$. 

The parallels continue for Steiner systems, where for many years after Teirlinck's result the existence of $(n,s,r)$-designs with $s>r \ge 3$ was only known in sporadic cases, and the existence of any examples for $r \ge 6$ was unknown until the general result of \cite{Kee14}. The situation for $(n,s,r)_q$-designs was even more dire, and was highlighted by Kalai \cite{Kal16} as one of the most important open problems remaining in Design Theory. It was conjectured by Metsch \cite{Met99} that no such designs with $s>r>1$ exist. This was recently disproved by Braun, Etzion, \"{O}sterg\aa rd, Vardy, and Wassermann \cite{BEOPV16}, who developed improved computational methods to find $(13,3,2)_2$-designs. However, there were no known examples for any other parameters, let alone any general results. 

In this paper we remedy this situation by completely answering the question: we show the existence of $(n,s,r)_q$-designs for any prime power $q$ and $s>r\ge 1$. Moreover, our result is analogous to Keevash's, in that we show the existence of $(n,s,r)_q$-designs for all sufficiently large $n$ satisfying the necessary `divisibility conditions'. Here recall the 
Gaussian $q$-binomial $\qbinom{n}{k}_q$, the number of $\mb{F}_q$-subspaces of $\mb{F}_q^n$ of dimension $k$, also given by the formula
\[\qbinom{n}{k}_q=\frac{(q^n-1)\cdots(q^1-1)}{(q^k-1)\cdots(q^1-1)(q^{n-k}-1)\cdots(q^1-1)}.\]

\begin{definition}\label{def:q-steiner}
Let $q$ be a prime power and let $\mb{F}_q^n$ be the $n$-dimensional vector space over $\mb{F}_q$. For $s > r$ and $\lambda\ge 1$, an $(n,s,r,\lambda)_q$-design is a multicollection $\mc{S}$ of $s$-dimensional subspaces such that every $r$-dimensional subspace is contained in exactly $1$ space in $\mc{S}$. We say it is \emph{simple} if there are no repeated $s$-spaces.
\end{definition}

\begin{theorem}\label{thm:main}
Fix $q,s,r$. For $n\ge n_{\ref{thm:main}}(q,s)$ such that $\qbinom{s-i}{r-i}_q\mid\qbinom{n-i}{r-i}_q$ for all $0\le i\le r-1$ there is an $(n,s,r)_q$-design.
\end{theorem}
\begin{remark}
Additionally, one can prove an analogue for ``sufficiently pseudorandom'' collections of $r$-dimensional subspaces, similar in spirit to \cite[Theorem~1.10]{Kee14} (with certain $q$-analogues of pseudorandomness conditions, as we will see in \cref{sub:q-extensions}). However, we do not pursue this extension here.
\end{remark}

We also prove a counting version as a simple corollary of the proof.

\begin{corollary}\label{cor:counting}
Under the assumptions of \cref{thm:main}, for $n\ge n_{\ref{cor:counting}}(q,s)$ the number of $(n,s,r)_q$-designs is
\[\bigg((1\pm q^{-c_{\ref{cor:counting}}(r,s)n})\frac{\qbinom{n-r}{s-r}_q}{\exp(\qbinom{s}{r}_q-1)}\bigg)^{\qbinom{n}{r}_q/\qbinom{s}{r}_q}.\]
\end{corollary}

The situation when $\lambda > 1$ is very similar, with a few added considerations regarding simplicity and the approximate covering step. We briefly sketch the necessary changes in \cref{sec:final-counting}, but the majority of focus and discussion everywhere else will be regarding \cref{thm:main}.
\begin{theorem}\label{thm:lambda}
Fix $q,s,r,\lambda$. For $n\ge n_{\ref{thm:lambda}}(q,s,\lambda)$ such that $\qbinom{s-i}{r-i}_q\mid\lambda\qbinom{n-i}{r-i}_q$ for all $0\le i\le r-1$ there is a simple $(n,s,r,\lambda)_q$-design.
\end{theorem}

We now briefly discuss at a high level some of the new techniques involved in this result. A more detailed proof outline and guide to the structure of the paper is given in \cref{sec:strategy,sub:organization}.

\subsection{New techniques: absorption in rigid algebraic scenarios}\label{sub:introduction-techniques}
Classic methods such as the R\"odl nibble for hypergraph matchings or more recent results can easily be seen to give an ``approximate'' version, i.e., a collection of $s$-dimensional spaces which cover $1-o(1)$ fraction of the $r$-spaces exactly once, and the remainder is uncovered. Therefore, as is typical, the key issue is dealing with the remainder. The most general form of this is the idea of \emph{absorption}, often credited to fundamental work of Erd\H{o}s, Gy\'arf\'as, and Pyber~\cite{EGP91}, and extended by R\"odl, Ruci\'nski, and Szemer\'edi~\cite{RRS06}. One sets aside some structure before attempting to solve a decomposition problem. Then, after approximately decomposing everything else, the remainder is small enough so that it can be handled in conjunction with the absorbing structure (akin to a sponge absorbing water).

Several traditional methods of absorption involved setting aside essentially randomly found structures to work with, which often are robust enough for the desired situation. However, for the problem of constructing Steiner systems, such absorbers are not sufficient for the task due to the sparsity of usable ``local switches'' to work with in such structures (and a similar phenomenon holds in our setting, to a worse degree). The work of Keevash \cite{Kee14} introduced a powerful idea of \emph{randomized algebraic constructions} to use as templates to construct Steiner systems in general. The method of \emph{iterative absorption} (introduced by K\"uhn and Osthus~\cite{KO13} and Knox, K\"uhn, and Osthus~\cite{KKO15}) was also adapted by Glock, K\"uhn, Lo, and Osthus \cite{GKLO16} to construct Steiner systems; however, as we will briefly discuss later, this technique appears to be less suitable in our setting.

In our situation, algebraic structure is already inherently present. In fact, the rigidity of subspaces of a ground vector space compared to subsets of a ground set means that we are more restricted in various ways. One can still perform approximate decompositions via the R\"odl nibble (or the more modern technique of random removal processes seen in \cref{sec:approximate}), since this can be seen purely from a hypergraph matching perspective, and we develop a framework for working with notions of embedding, pseudorandomness, and ``typicality'' in $q$-analogues of hypergraphs, which we call $q$-systems (\cref{sub:q-extensions}). We note that even the interactions between pseudorandomness conditions and counts of various ``$q$-embeddings'' already highlight the inherently subtle linear-algebraic nature of the problem; see e.g.~\cref{lem:one-step-typical} which must account for certain ``twists'' when iteratively embedding $q$-systems. Furthermore, when it comes to the template and the absorption process, the fact that we are already forced to work with $\mb{F}_q$-subspaces of $\mb{F}_q^n$ poses substantial challenges.

The template, in this setting a special set of $r$-dimensional subspaces coming from a collection of $s$-spaces, is formed via a randomized algebraic process as in \cite{Kee14}, but we are focused on making entire vector spaces play nice with respect to each other. Furthermore, one must ensure the template is sufficiently generic (in the sense of \cref{def:generic}) to work with and the necessary algebraic constructions may not exist over $\mb{F}_q$ directly. Thus, we must pass to a field extension $L/\mb{F}_q$ and put an $L$-space structure on $\mb{F}_q^n$. In cases where $n$ is not divisible by any small number, say $n$ is prime, this is not directly possible and we ultimately embed multiple incompatible $L$-structures on vector spaces of finite codimension (\cref{sub:setup}).

For the absorption process itself, after creating an approximate decomposition, the remainder (or \emph{leave}) is covered in a way such that some $r$-dimensional subspaces in the template are covered exactly twice. Then we attempt to remove certain template $s$-spaces and reconfigure the rest in a way that removes the extra multiplicity from this \emph{spillover}. To do this, we find a ``signed integral decomposition'' of the spillover by understanding certain associated lattices, and we furthermore guarantee that it is appropriately \emph{bounded} (\cref{def:bounded}). Then we use a ``subspace exchange process'' to massage this integral decomposition into a form amenable to absorption using the template structure. The latter bears similarity to the ``clique exchange'' of \cite{Kee14}, although the $q$-analogue and multiple $L/\mb{F}_q$-structures pose various new technical difficulties. 

However, the integral decomposition is significantly hampered by the rigidity of the subspace setting. For Steiner systems, the key associated lattice is defined by relatively simple divisibility conditions (due to work of Graver and Jurkat \cite{GJ73} and of Wilson \cite{Wil73}) and it in fact has a particularly natural ``bounded'' generating set to work with, formed by certain ``octahedral'' structures (see e.g.~\cite[Section~5]{Kee14}). However, work of Ray-Chaudhuri and Singhi \cite{RCS89} shows that lattices associated to $(n,s,r)_q$-designs are not nearly so nice. As a result, we work with a greedily designed bounded approximate generating set, and introduce a way to boost this approximate behavior by using multiple copies to ``cover gaps''. Furthermore, many arguments using the symmetry of all vertices in $K_n$ are hampered by the less robust symmetries available in our setting, so here and elsewhere we often resort to delicate moment computations instead of more standard concentration of measure arguments which do not apply. See \cref{sub:outline-integral} for further discussion of the bounded integral decomposition and its role in the proof, and see \cref{sub:bounded-sparse} for further detail on this boosting.

Finally, we remark that despite the challenges posed by using a randomized algebraic template, the iterative absorption method seemed less amenable in this setting. For instance, at a high level, the ``cover-down'' procedure in \cite{GKLO16} is not clearly compatible with the algebraic setting, since an $r$-space will intersect a subspace of $\mb{F}_q^n$ (that we may be trying to ``zoom in towards'') in another subspace, which constrains the intersection types that may appear when treating this as a hypergraph problem; these intersection types do not appear to play well together as nicely as in the Steiner system scenario. Of course, there may be ways to use the vector space and dimensional structure to proceed, but the aforementioned complicated structure of the underlying associated lattices \cite{RCS89} in our setting suggest that at the very least, more work must be done if it is possible to accommodate such an approach. 

\subsection*{Acknowledgements}
We thank Zach Hunter for pointing out an error in the previous version of the proof of \cref{lem:rainbow-decomposition}.

\section{Proof strategy}\label{sec:strategy}
We now outline the proof strategy in more detail. At a basic level we wish to run a random subspace removal process to cover almost all $r$-spaces, and then cover the rest using an absorber. Of course, as in the case of Steiner systems, the key issue is precisely what sort of absorption strategy will suffice. As in the work of Keevash \cite{Kee14}, we will take an algebraic approach: we plant a well-structured template first, cover most of the remainder, and then absorb the rest into this template using a robust quantity of local switches. Throughout, we let $V=\mb{F}_q^n$.

\subsection{Algebraic template}\label{sub:outline-template}
The $s$-template $\mc{S}_\mr{tem}$, precisely constructed in \cref{def:template}, is a set of $s$-spaces coming from a linear algebraic construction. This yields an underlying set of $r$-spaces, $G_\mr{tem}$, which are covered by this collection.

This construction is obtained from the following observation: if $N\in\mb{F}_q^{s\times r}$ has the property that $\Pi N$ is invertible for every $\Pi\in\mb{F}_q^{r\times s}$ of rank $r$, then the multiset
\[\{\{\mr{span}_{\mb{F}_q}(Nx)\colon x\in V^r\text{ and }\dim\mr{span}_{\mb{F}_q}(Nx)=s\}\}\]
provides an $(n,s,r,\lambda)_q$-design (not necessarily simple) for some appropriate $\lambda$. Here the span of a column vector with coordinates in $V$ refers to taking the $\mb{F}_q$-span of its $s$ coordinates treated as a set. Indeed, for every $r$-space $R$, there are a fixed number of possible $\mb{F}_q$-bases $b\in R^r$ and for each $\Pi\in\mb{F}_q^{r\times s}$ of rank $r$ each possible such $b$ shows up as the value $\Pi Nx=b$ for precisely one $x$ by invertibility of $\Pi N$.

Therefore, if we sample the construction in a way that forces no $r$-space to be repeated, then we will have a partial design appearing as a dense subset of a rigid algebraic structure. In reality, we will further subsample this collection (at a rate depending on a parameter $\tau$ which ultimately will be chosen to be $q^{-cn}$) in order to ensure removing the template does not significantly limit our options. We can enforce that no $r$-space is repeated by having each space $R\in\mr{Gr}(V,r)$ choose ``how it wants to be included in the template'' so that only $1$ out of the $\lambda$ possible ``configurations'' can possibly appear. A similar concept of sampling the template (called \emph{activation} and used in a more general situation) and using configurations to reduce multiplicity (called \emph{compatibility}) appears in \cite[Definition~3.2]{Kee14}.

However, note that if $s,r$ are large and $q$ is small then the desired matrix $N$ may not exist. Thus, we must actually consider a field extension $L=\mb{F}_{q^\ell}$ of $\mb{F}_q$ for appropriate $\ell$ large enough that such $N\in L^{s\times r}$ will exist (only with respect to $\Pi$ with coefficients in $\mb{F}_q$). In particular, an \emph{algebraically generic} choice of $N$ works. If $\ell|n$ then it is possible to give $V=\mb{F}_q^n$ the structure of an $(n/\ell)$-dimensional $L$-space and make sense of the matrix product $Nx$. When $\ell\nmid n$, though, we can only do this on a subspace of $V$ of bounded codimension. For this reason, and in order to ensure the template is fully spread throughout $V$ and not concentrated in one location, we actually plant $z=n^2$ different copies of the template within $z$ different $L$-structures of randomly chosen bounded codimension subspaces of $V$. For clarity of notation, we fix a single $L$-vector space $K$ of $L$-dimension $\lfloor n/\ell\rfloor$ (and $\mb{F}_q$-dimension $\ell\lfloor n/\ell\rfloor$) and consider uniformly random injective linear maps $\iota_1,\ldots,\iota_z\colon K\to V$. The indices $i\in[z]$ for the different copies of the template will often be referred to as \emph{colors}, so that we can have a notion of \emph{monochromatic} and \emph{rainbow} $s$-spaces.

Finally, we explain in \cref{sub:outline-absorption} exactly why this rigid algebraic structure leads to robust switches which allow for absorption.

\subsection{Approximate design}\label{sub:outline-approximate}
Having set aside the $r$-spaces $G_\mr{tem}$, the next step is to construct an approximate $(n,s,r)_q$-design (avoiding the template). We do this by running a random process to create an approximate matching in the $\qbinom{s}{r}_q$-uniform hypergraph defined as follows: vertices are $r$-spaces other than $G_\mr{tem}$, and edges are labeled by $s$-spaces with the edge containing precisely vertices corresponding to the $r$-dimensional subspaces (so we must restrict to $s$-spaces whose $r$-dimensional subspaces are all not in $G_\mr{tem}$).

Naively, one might run the following process: uniformly at random select a hyperedge covering a set of yet-uncovered vertices and iterate until this is no longer possible. However, it is convenient to terminate the process the moment the remaining induced hypergraph is sufficiently irregular in an appropriate sense. Additionally, more importantly, we care greatly about the time for which we can control this process (and specifically the number of $r$-spaces in the remainder). With this basic process, the size of the remainder will be small, but dependent on the initial irregularity coming from removing $G_\mr{tem}$. Thus, the remainder will not be smaller than the size of the template, which poses a problem for our absorption strategy.

Hence, there is an additional step where we find a subset of hyperedges which is regularized to account for these minor irregularities, allowing us to control the process for longer. This is achieved in \cref{lem:template-boosted}, and the random process is run in \cref{prop:approximate-covering}. The content of \cref{lem:template-boosted} can be thought of as a kind of \emph{regularity boosting}; procedures similar to this occur in \cite[Lemma~4.1]{Kee14} and in \cite[Lemma~4.2]{BGKLMO20} in the setting of Steiner systems, though the general idea dates back further to approximate hypergraph matching results of Alon, Kim, and Spencer \cite{AKS97}. We note here that the required regularity boost is obtained by ``local rebalancing'', which is implemented using local decodability of the lattice associated to $(n,s,r)_q$-designs. In general the necessary local ``gadgets'' used to implement regularity boosting correspond to short kernel vectors of the associated boundary operator $\partial_{s,r}$ (see \cref{sub:lattices} for a precise definition of these lattices and operators).

Finally, we remark that the remainder is not only small with respect to the number of $r$-spaces left over, but it also not too concentrated in any location. Specifically, every $(r-1)$-space does not have too many extensions to an $r$-space in the remainder, which we encode via a condition called \emph{boundedness} (\cref{def:bounded}).

\subsection{Covering the leave, and the spillover}\label{sub:outline-spillover}
After removing the template and this approximate cover, the remainder is a \emph{leave} which is significantly smaller than the template. The next step is to take a collection of $s$-spaces, one covering each $r$-space in the leave, such that each $s$-space has all but $1$ of its $r$-spaces in $G_\mr{tem}$. If we add in these $s$-spaces, then the result will have almost all $r$-spaces covered exactly once, but some are covered twice. The ones covered twice form the \emph{spillover}. This is accomplished in \cref{lem:final-covering}.

For technical reasons that will become apparent in \cref{sub:outline-absorption}, we will further run the above process so that the spillover satisfies a certain disjointness condition: for every $r$-space in the spillover, it is contained in an $s$-space of $\mc{S}_\mr{tem}$, and we wish for said $s$-spaces to be distinct. In fact, we guarantee the slightly stronger property of field disjointness (\cref{def:field-disjoint}). Furthermore, we wish for boundedness of the spillover and in fact $r$-dimensional \emph{field boundedness} (\cref{def:field-bounded}) to ensure that the results are not overly concentrated with respect to the underlying $L$-structures as well.

\subsection{Bounded integral decomposition}\label{sub:outline-integral}
At this stage, we need to find a way to ``remove'' the extra copy of the spillover from our design. Ideally, we wish to find $\mc{S}_1\subseteq\mc{S}_\mr{tem}$, a set of template $s$-spaces, and a set $\mc{S}_2$ of $s$-spaces with the following properties: (a) all of the spillover is contained in $r$-spaces coming from $\mc{S}_1$, so that removing $\mc{S}_1$ yields a partial design with all multiplicities $1$ and $0$, instead of $2$ on the spillover, and (b) $\mc{S}_2$ precisely covers the ``hole'' of multiplicity $0$ $r$-spaces left by the removal of $\mc{S}_1$. We will ultimately find such a decomposition, but this requires multiple steps.

The first step is finding a reasonable integral decomposition of the spill. At this point, it is useful to consider signed multicollections of $s$-spaces and $r$-spaces, and to represent these as vectors in $\mb{Z}^{\mr{Gr}(V,s)}$ and $\mb{Z}^{\mr{Gr}(V,r)}$ with a natural boundary map $\partial_{s,r}$ mapping the first space to the second (see \cref{sub:lattices}). We show that the spillover, $J$, can be represented as the result of considering the (signed) $r$-spaces within a signed sum of $s$-spaces $\Phi$. Furthermore, we show that boundedness of $J$ guarantees we can obtain $\Phi$ which is bounded. This gets us a bit closer to the above, which is equivalent to $J=\partial_{s,r}(\mc{S}_1-\mc{S}_2)$ for $\mc{S}_1\subseteq\mc{S}_\mr{tem}$ and $\mc{S}_2\subseteq\mr{Gr}(V,s)$ (abusively identifying a set with its indicator vector).

This is a highly nontrivial argument, and we prove the abstract \cref{thm:bounded-inverse} that shows that in general bounded integral elements of an appropriate lattice have correspondingly bounded inverses. We save detailed discussion of the techniques for the proofs in \cref{sec:bounded-integral-decomposition}, but a crucial step involves robustly showing in a sense that appropriate collections of $r$-spaces within a random host $H\subseteq\mr{Gr}(V,r)$ can be decomposed by $s$-spaces whose $r$-subspaces are all in $H$. Beyond this, the fact that the key lattices defined in \cref{sub:lattices} have \emph{robust local decodability} plays an important role. See \cite[Section~1.4]{Kee14} for a discussion of the interplay of robust local decodability and ``bounded integral designs'' in the case of Steiner systems. We note however that the proof techniques in \cite{Kee14} are largely unavailable to use at this stage as the proof relies on an essentially explicit description of the corresponding lattice in terms of ``octahedra''. No such nice characterization appears to exist for subspace designs due to nontrivial conditions on the lattice (see \cite{RCS89}). Instead our proof essentially only uses local decodability and the existence of local ``subspace exchanges'' (\cref{prop:subspace-exchange}).

\subsection{Subspace exchange process}\label{sub:outline-exchange}
At this stage, the spillover is expressed as a signed integral decomposition of $s$-spaces that is appropriately bounded. We now massage its form to get closer (but not all the way) to the $\mc{S}_1,\mc{S}_2$ mentioned in \cref{sub:outline-integral}.

First, we rewrite the spillover $J$ as a signed integral decomposition $\partial_{s,r}\Phi$ where every $r$-space appears with multiplicity at most $1$ in the positive and negative portions (it could appear once in both and cancel out). We furthermore require that the $r$-spaces that appear are contained within $G_\mr{tem}$ and are not just bounded but field bounded, and that they continue to be field disjoint as in \cref{sub:outline-spillover}. Additionally, for technical reasons we require the $s$-spaces that appear to be \emph{rainbow} in the sense that all $\qbinom{s}{r}_q$ different $r$-subspaces are in different template copies. This is the content of \cref{lem:final-rainbow}, and it is accomplished in two steps of \emph{splitting} and \emph{elimination}. In \emph{splitting}, we randomly replace each $s$-space with a signed ``flipped configuration'' of $s$-spaces (with the same signed sum of $r$-spaces) disjointly randomly from the current support in sequence, so that all the high multiplicity cancellation occurs in a controllable way (being ``split'' into groups which do not interact). In \emph{elimination}, we take the high multiplicity cancellations and break them into pairs, and then randomly replace these pairs with a signed ``flipped configuration'' that reduces the multiplicity of the cancellation at a specified $r$-space (allowing us to ``eliminate'' all cancellations except of the form $(+1)+(-1)=0$).

Second, for a technical reason we slightly massage the resulting $s$-spaces to maintain the earlier properties, but also to ensure that every $s$-space $S$ coheres with all its $r$-subspaces $R$: if $R$ comes from template index $i$, then $S\leqslant K_i$, i.e., $S$ is within the vector space upon which the field structure for the $i$th template copy is placed. This is done in \cref{lem:final-color-consistent}.

Finally, we turn the positive $s$-spaces in our decomposition \emph{monochromatic}, which is a necessary precondition for having the form described in \cref{sub:outline-integral}. In fact, there is another less obvious precondition, which we call \emph{configuration compatibility} of the $s$-space (\cref{def:configuration-compatible}): essentially, notice that every $s$-space in $\mc{S}_\mr{tem}$ is composed of $r$-spaces that each have a different configuration of ``how it wants to be included in the template'' (recall the discussion in \cref{sub:outline-template}). Transforming the positive $s$-spaces in such a way is done in \cref{prop:final-disjoint-monochromatic}, again by running a random process of ``flips''. The key issue at this stage is guaranteeing that we can indeed find flips that have the resulting positive $s$-spaces be monochromatic (in fact, the negative $s$-spaces will not be monochromatic). Additionally, beyond guaranteeing $r$-dimensional field boundedness and field disjointness, for technical reasons at this stage we need to guarantee some slightly stronger properties defined over $s$-spaces.

\subsection{Counting extensions (with accessories) and disjoint random processes}\label{sub:outline-extension}
To prove \cref{lem:final-rainbow,lem:final-color-consistent,prop:final-disjoint-monochromatic}, and in fact also \cref{lem:final-covering}, we need two ingredients: (a) an understanding of how many ways there are to extend a fixed structure of $r$-spaces into a larger pattern of $r$-spaces such that the new $r$-spaces are all in certain prescribed copies of the template with certain configuration properties (etc.), and (b) an ability to run a disjoint random process, i.e., one where at each stage the next choices are taken not randomly from all extensions, but only those that are appropriately disjoint from the previous choices (etc.). It is not hard to see that (b) requires some form of boundedness as an input, but we also need the heuristics for (a,b) to work out that we can continue to maintain this boundedness throughout.

We achieve (a) in \cref{sec:extension-template}, in particular \cref{prop:template-extendable}. It is written in a high degree of generality, but basically shows that extensions into the template (with various conditions, as long as they are not over-constraining) act as one might expect from a dense random set of $r$-spaces, up to factors of the density of the template. The proof follows by various concentration techniques over the randomness of the template, including Azuma--Hoeffding (\cref{lem:azuma}) and the method of moments.

We encapsulate a general framework for (b) in \cref{lem:master-disjointness}; however, we note that it is adapted to processes that mainly consider properties defined via the underlying $r$-spaces, so it is not directly suitable for \cref{prop:final-disjoint-monochromatic} (in which we maintain some $s$-space related properties as well).

\subsection{Absorbing the spillover}\label{sub:outline-absorption}
The output of \cref{prop:final-disjoint-monochromatic} is a decomposition of the spillover $J$ as $\partial_{s,r}\Phi$ where $\Phi\in\{0,\pm1\}^{\mr{Gr}(V,s)}$ and the positive $s$-spaces are monochromatic, configuration compatible, and disjoint and bounded in various senses. Finally, in \cref{prop:final-absorber} we show that the positive $s$-spaces can be transformed into a collection of positive and negative $s$-spaces with the same image under $\partial_{s,r}$, such that the new positive $s$-spaces are all in $\mc{S}_\mr{tem}$. This provides us the form as claimed in \cref{sub:outline-integral}, and so will finish.

This is the stage where we strongly use the algebraic structure of the template. Up until this point we have mainly used extension counts coming from \cref{prop:template-extendable} which merely ensure that $G_\mr{tem}$ has many substructures as if random (with some basic accoutrements); then we performed ``flips'' replacing $s$-spaces with other $s$-spaces, but these did not need to be in $\mc{S}_\mr{tem}$ (we only needed that the $r$-subspaces were in $G_\mr{tem}$, plus whatever other conditions). On the other hand, in this step we specifically want to use ``flips'' such that in the result, all positive $s$-spaces are in $\mc{S}_\mr{tem}$, which were the highly structured algebraic $s$-spaces we set aside at the beginning. To highlight this difference, up to factors of the template density there are roughly order $q^{sn}$ many $s$-spaces whose $r$-subspaces are in $G_\mr{tem}$, whereas $\#\mc{S}_\mr{tem}\le q^{rn}$ is much smaller.

To prove \cref{prop:final-absorber}, we use the Lov\'asz Local Lemma (\cref{lem:lll}) to show we can simultaneously find disjoint ``algebraic flips'' that involve $\mc{S}_\mr{tem}$ in this special way. There are much fewer choices here than in \cref{sub:outline-exchange}. These special algebraic flips are given in \cref{def:absorber} (defined over $L$; we then translate to template copies via the various maps $\iota_1,\ldots,\iota_z$), and we analyze their basic properties such as their count in \cref{sec:absorber}, coming from the underlying structure of the template. These properties allow us to establish \cref{prop:final-absorber}, completing the final piece of the argument.

\subsection{Possible future directions}\label{sub:future-directions}
In light of our resolution of the existence problem for $(n,s,r)_q$-designs, a number of natural questions arise. The most natural relate to so-called resolvability-type variants. For instance, can one find $(n,3,2)_2$-designs that can be partitioned into a collection of $(n,3,1)_2$-designs (for all large appropriately divisible $n$)? More generally, one might ask for a Baranyai-type design: for given $s\ge 1$, we want to partition $\mr{Gr}_q(n,s)$ into $(n,s,s-1)_q$-designs, each of we further partition into $(n,s,s-2)_q$-designs, etc., down to $(n,s,1)_q$-designs. More broadly, we make the informal conjecture (the correct precise statement is not clear) that various extensions in the style of \cite{Kee18c} are all possible under the necessary divisibility conditions coming from associated lattices (which are not always fully obvious).

\subsection{Organization}\label{sub:organization}
We now briefly discuss the organization of the paper. In \cref{sec:preliminaries} we collect a number of basic notions and results which will be used throughout the paper: $q$-analogues of hypergraphs called $q$-systems (\cref{sub:q-extensions}), lattices and boundary operators for integral designs (\cref{sub:lattices}), concentration inequalities (\cref{sub:concentration}), a notion of a subset of field elements being generic with respect to algebraic equations (\cref{sub:genericity}), and finally the formal basic setup (\cref{sub:setup}). We then formally construct the subspace exchanges used which will be the ultimate source of all manner of ``local switches'' and ``flips'' in \cref{sec:exchanges}. Using this we then prove our bounded integral decomposition result in \cref{sec:bounded-integral-decomposition}. In \cref{sec:template} we formally construct the template. In \cref{sec:absorber,sec:absorption} we first prove that the template robustly contains many subspace exchanges of a special form (involving template $s$-spaces) and use these to prove \cref{prop:final-absorber} (the main absorption statement). In \cref{sec:approximate} we prove the necessary approximate covering results. In \cref{sec:extension-template} we count various extensions into the template (\cref{prop:template-extendable}) and these are then exploited in \cref{sec:cover} (alongside a disjoint random process framework \cref{lem:master-disjointness}) to both create the spillover and then massage the spillover into a form where \cref{prop:final-absorber} is applicable. Finally in \cref{sec:final-counting} we put all the pieces together to prove \cref{thm:main}, and also discuss \cref{cor:counting,thm:lambda}.

\section{Preliminaries and setup}\label{sec:preliminaries}
\subsection{\texorpdfstring{$q$}{q}-systems and \texorpdfstring{$q$}{q}-extensions}\label{sub:q-extensions}
We will often need to understand the number of ways we can embed certain configurations of subspaces in a (potentially random) host. Thus we define various related notions for future use, starting with a natural $q$-analogue of hypergraphs.
\begin{definition}\label{def:q-system}
An $r$-dimensional \emph{$q$-system} $H$ on $\mb{F}_q^n$ is a subset of $\mr{Gr}_q(n,r)$, or equivalently an element of $\{0,1\}^{\mr{Gr}_q(n,r)}$; multi-$q$-systems are elements of $\mb{Z}_{\ge 0}^{\mr{Gr}_q(n,r)}$ and signed multi-$q$-systems are elements of $\mb{Z}^{\mr{Gr}_q(n,r)}$. We will refer to the elements of these as \emph{$r$-spaces} and write $d(H):=|H|/\qbinom{n}{r}_q$ for the density of $H$. We write $\mr{Vec}(H)$ for the underlying vector space upon which $H$ lies. Given a subspace $F\leqslant\mr{Vec}(H)$, we write $H[F]$ for the $q$-system on $F$ composed of those $r$-spaces in $H$ fully contained within $F$.
\end{definition}
\begin{remark}
Similar to hypergraphs, the underlying vector space of $H$ may be larger than the span of all its $r$-spaces.
\end{remark}
Now we define $q$-embeddings and $q$-extensions (we may omit the parameter $q$ as it is generally obvious).
\begin{definition}\label{def:q-embedding}
A \emph{$q$-embedding} of a $q$-system $H$ in a multi-$q$-system $G$ is an injective linear map $\phi\colon\mr{Vec}(H)\to\mr{Vec}(G)$ such that $G_{\phi(R)} > 0$ for all $R\in H$. Given an $r$-dimensional $q$-system $H$, a subspace $F\leqslant\mr{Vec}(H)$, and an injective linear map $\phi\colon F\to\mb{F}_q^n$, we call $E = (\phi,F,H)$ a \emph{$q$-extension}. We write $e_E:=|H\setminus H[F]|$ and $v_E:=\dim\mr{Vec}(H)-\dim F$. Now if $G$ is an $r$-dimensional $q$-system on $\mb{F}_q^n$ we write $\mc{X}_E(G)$ for the set of embeddings of $H$ in $G+\phi(H[F])$ (not distinguishing potential multiple edges) that agree with $\phi$ on $F$. Let $X_E(G):=\#\mc{X}_E(G)$.
\end{definition}
\begin{definition}\label{def:q-typical}
We say that an $r$-dimensional $q$-system $G$ on $\mb{F}_q^n$ is \emph{(c,h)-typical} if for every $q$-extension $E=(\phi,F,H)$ with $v_E > 0$ and $\dim\mr{Vec}(H)\le h$ we have $X_E(G)=(1\pm c)d(G)^{e_E}q^{v_En}$.
\end{definition}
We finally define a notion of boundedness of a vector $J$ with respect to a potentially sparse $q$-system $L$.
\begin{definition}\label{def:embedding-boundedness}
Given a $q$-extension $E = (\phi,F,H\setminus\{R\})$, $L\subseteq\mr{Gr}_q(n,r)$, and $J\in\mb{Z}^{\mr{Gr}_q(n,r)}$, let
\[X_E^R(L,J):=\sum_{\phi^\ast\in\mc{X}_E(L)}|J_{\phi^\ast(R)}|.\]
We say that $J$ is \emph{$(\theta,h)$-bounded wrt $L$} if $X_E^R(L,J)\le\theta d(L)^{e_E}q^{v_En}$ for all $q$-extensions $E = (\phi,F,H\setminus\{R\})$ with $\dim\mr{Vec}(H)\le h$ and $R\in H\setminus H[F]$.
\end{definition}

Finally, it will be useful for us to have a slightly easier condition for verifying typicality. For hypergraphs, one can reduce the notion of typicality for $G$ coming from extension counts to a question of how many vertices simultaneously extend a collection of $(r-1)$-sets into a collection of edges of $G$ (see e.g.~\cite[Definition~1.3,~Lemma~2.16]{Kee14}). However, in $q$-systems there is a possibility that vector spaces are in some sense ``transverse'' but still dependent, so the required notion is subtler.
\begin{lemma}\label{lem:one-step-typical}
Suppose $c\in(0,1)$, $n$ is large with respect to $h$, and that the $r$-dimensional $q$-system $G$ on $\mb{F}_q^n$ satisfies the following property:
\begin{itemize}
    \item Consider any choice of $(r-1)$-spaces $Q_1,\ldots,Q_a\leqslant\mb{F}_q^n$ and vectors $v_1,\ldots,v_a\in\mb{F}_q^n$ where $a\le q^{rh}$ such that if $i\neq j$ and $Q_i=Q_j$ then $\mr{span}_{\mb{F}_q}(Q_i\cup\{v_i\})\neq\mr{span}_{\mb{F}_q}(Q_j\cup\{v_j\})$. Then there are $(1\pm c)d(G)^aq^n$ vectors $v\in\mb{F}_q^n$ with $\mr{span}_{\mb{F}_q}(Q_i\cup\{v+v_i\})\in G$ for all $i\in[a]$.
\end{itemize}
Then $G$ is $(2^hc,h)$-typical.
\end{lemma}
\begin{proof}
Consider some extension $E=(\phi,F,H)$ with $v_E > 0$ and $t = \dim\mr{Vec}(H)\le h$. Thus $v_E=t-\dim F$ and $t\ge\dim F+1$. Let $\{u_1,\ldots,u_t\}$ be a basis of $\mr{Vec}(H)$ where $\{u_1,\ldots,u_{\dim F}\}$ is a basis for $F$ and let $V_i = \mr{span}_{\mb{F}_q}\{u_1,\ldots,u_i\}$ for $0\le i\le t$. Let $e_i=|H[V_i]|-|H[V_{i-1}]|$ and note $e_E = e_{\dim F+1}+\cdots+e_t$.

We now count $\psi\in\mc{X}_E(G)$. Note that $\psi(u_i) = \phi(u_i)$ is fixed for $i\in[\dim F]$. For $i\in\{\dim F+1,\ldots,t\}$ in order, we count the number of ways to choose $\psi(u_i)$ given the prior choices. Specifically, a choice of $\psi(u_i)$ will determine the locations of the $r$-spaces in $H[V_i]\setminus H[V_{i-1}]$: each such space can be written as $\mr{span}_{\mb{F}_q}(Q\cup\{u_i+w\})$ for some $(r-1)$-space $Q\leqslant V_{i-1}$ and vector $w\in V_{i-1}$, hence its image under $\psi$ is $\mr{span}_{\mb{F}_q}(\psi(Q)\cup\{\psi(u_i)+\psi(w)\})$ where $\psi(Q)\leqslant\psi(V_{i-1})$ and $\psi(w)\in\psi(V_{i-1})$. By the given condition, there are $(1\pm c)d(G)^{e_i}q^n$ choices of $\psi(u_i)$ which ensure that simultaneously all these spaces land in $G$. We can apply the condition since clearly $|H[V_i]\setminus H[V_{i-1}]|\le q^{rh}$ (being $r$-spaces within a space of dimension at most $h$) and since the $r$-spaces of $H[V_i]\setminus H[V_{i-1}]$ with the same restriction to $V_{i-1}$ (which is called $Q$ here) will have associated vectors $w$ that are different $\imod{Q}$, otherwise they would represent the same $r$-space in $H$.

The total count can be read off by multiplying these numbers of choices, which gives $(1\pm c)^{v_E}d(G)^{e_E}q^{v_En} = (1\pm 2^hc)d(G)^{e_E}q^{v_En}$, as desired.
\end{proof}

\subsection{Lattices and boundary operators}\label{sub:lattices}
We can view a (signed multi-)$q$-system of dimension $s$ as an element of $\mb{Z}^{\mr{Gr}_q(n,s)}$. Define the boundary map $\partial_{s,r}\colon\mb{Z}^{\mr{Gr}_q(n,s)}\to\mb{Z}^{\mr{Gr}_q(n,r)}$ via
\[e_S\mapsto\sum_{\substack{R\leqslant S\\\dim R = r}}e_R\]
and linearity. The problem of constructing an $(n,s,r)_q$-design is the same as finding $\Phi\in\{0,1\}^{\mr{Gr}_q(n,s)}$, or equivalently just $\Phi\in\mb{Z}_{\ge 0}^{\mr{Gr}_q(n,s)}$, with
\[\partial_{s,r}\Phi = \sum_{R\in\mr{Gr}_q(n,r)}e_R.\]
(We will often abusively treat sets as their indicator vectors in the appropriate space, which in particular allows application of $\partial_{s,r}$.)

It is therefore useful to define the image lattice $\mc{L}=\mc{L}_{s,r}:=\partial_{s,r}\mb{Z}^{\mr{Gr}_q(n,s)}$, which defines the conditions for the image. In order for the desired element to be just in the image of $\partial_{s,r}$, this imposes natural divisibility constraints on $n$ which we will find to be identical to those in \cref{thm:main}.

This lattice $\mc{L}$ was characterized by Ray-Chaudhuri and Singhi \cite{RCS89}. We note that they showed the ``obvious'' necessary divisibility conditions are not sufficient for an element to be in the lattice (unlike the set system case studied by Wilson \cite{Wil73} and Graver and Jurkat \cite{GJ73}); however, for a vector in $\mb{Z}^{\mr{Gr}_q(n,r)}$ with all coefficients equal to some $\lambda\in\mb{Z}$, the conditions degenerate to precisely those natural conditions. In particular, it suffices to require:
\begin{equation}\label{eq:divisibility}
\qbinom{s-i}{r-i}_q\mid\lambda\qbinom{n-i}{r-i}_q\text{ for all }0\le i\le r-1.
\end{equation}
Note that a positive density of such values exist: the values $n$ satisfying $A!|n-r+1$ for $A$ large in terms of $q,r,s$ all satisfy \cref{eq:divisibility}. Alternatively, the values $n$ satisfying $s-i|n-i$ for all $0\le i\le r-1$ will satisfy \cref{eq:divisibility}: this is a compatible system of modular constraints by the Chinese remainder theorem, so this proves that the density of valid $n$ is in fact lower-bounded independent of $q$. Additionally, \cref{eq:divisibility} is equivalent to constraining $n$ to live in one of a finite list of bounded parameter modular congruences (depending on $q,s,r$).

We state the precise result for future reference.
\begin{theorem}[{\cite[Theorem~1.1]{RCS89}}]\label{thm:integral-lattice}
For $n\ge s > r\ge 1$ we have
\[\lambda\sum_{\substack{R\leqslant\mb{F}_q^n\\\dim R=r}}e_R\in\partial_{s,r}\mb{Z}^{\mr{Gr}_q(n,s)}\]
if and only if \cref{eq:divisibility} holds.
\end{theorem}

\subsection{Concentration inequalities}\label{sub:concentration}
We will often need the Chernoff bound for binomial and hypergeometric distributions (see e.g.~\cite[Theorems~2.1~and~2.10]{JLR00}). Recall that the hypergeometric distribution $\mr{Hyp}(n,n_1,n_2)$ for $n\ge\max\{n_1,n_2\}$ is the size of the intersection of two independent uniformly random subsets of $[n]$ of sizes $n_1,n_2$.
\begin{lemma}[Chernoff bound]\label{lem:chernoff}
Let $X$ be either:
\begin{itemize}
    \item a sum of independent random variables, each $[0,1]$-valued, or
    \item hypergeometrically distributed (with any parameters).
\end{itemize}
Then for any $\delta>0$
\[\Pr[X\le(1-\delta)\mb{E}X]\le\exp(-\delta^2\mb{E}X/2),\qquad\Pr[X\ge (1+\delta)\mb{E}X]\le\exp(-\delta^2\mb{E}X/(2+\delta)).\]
\end{lemma}

We will also frequently require the Azuma-Hoeffding inequality (see \cite[Theorem~2.25]{JLR00}).
\begin{lemma}[Azuma--Hoeffding inequality]\label{lem:azuma}
Let $X_0, \ldots, X_n$ form a martingale sequence such that $|X_k-X_{k-1}|\le c_k$ almost surely. Then 
\[\mb{P}[|X_0-X_n|\ge t]\le 2\exp\bigg(-\frac{t^2}{2\sum_{k=1}^nc_k^2}\bigg)\]
\end{lemma}
\begin{remark}
We will refer to $\sum_{k=1}^nc_k^2$ as the \emph{variance proxy} in such a situation.
\end{remark}

We will also require the following useful binomial domination lemma. 
\begin{lemma}\label{lem:bernoulli-domination}
Let $X_i\in\{0,1\}$, $p_i\in[0,1]$, and $a_i\in\mb{R}^+$ for $i\in[n]$. Suppose that 
\[\mb{P}[X_i=1|X_1,\ldots,X_{i-1}]\le p_i\]
for all $i\in[n]$ and let $Y_i$ be independent random variables distributed as $\mr{Ber}(p_i)$ (i.e., it is $1$ with probability $p_i$ and $0$ otherwise). Then for any $t\ge 0$ we have 
\[\mb{P}\big[\sum_{i\in[n]}a_iY_i\ge t\big]\ge\mb{P}\big[\sum_{i\in[n]}a_iX_i\ge t\big].\]
\end{lemma}

Finally we will use the Lov\'asz Local Lemma \cite{EL75}. 
\begin{lemma}[{\cite[Lemma~5.1.1]{AS16}}]\label{lem:lll}
Let $\mc{B}_1,\ldots,\mc{B}_n$ be events in a probability space and let $D = ([n], E)$ be a directed graph which is a dependency graph for $(\mc{B}_i)_{i\in[n]}$, i.e., for each $i\in[n]$, $\mc{B}_i$ is mutually independent of all events $\{\mc{B}_j\colon(i,j)\notin E\}$. If $x_i\in[0,1]$ and $\mb{P}[\mc{B}_i]\le x_i\prod_{(i,j)\in E}(1-x_j)$ for all $i\in[n]$ then
\[\mb{P}\Big[\bigcap\ol{\mc{B}_i}\Big]\ge\prod_{i=1}^n(1-x_i).\]
\end{lemma}

\subsection{Algebraic genericity}\label{sub:genericity}
It will be useful to introduce the following notion of algebraic genericity.
\begin{definition}\label{def:generic}
For any field $K$ and subsets $S,T\subseteq K$ we say that $S$ is \emph{$T$-generic of degree $d$} if there is no nonzero polynomial of degree at most $d$ in variables $\{x_s\colon s\in S\}$ with coefficients in $T$ that vanishes when we substitute $x_s=s$ for all $s\in S$. We also say that a vector or matrix is $T$-generic (and a collection of such is jointly $T$-generic) if the set of entries is $T$-generic.
\end{definition}

One useful fact is that if a $d\times d$ matrix is $\{0,\pm1\}$-generic of degree $d$ then it is invertible. We additionally record a master lemma encapsulating the fact that certain linear-algebraic conditions are algebraically generic outside of explicit degeneracies. This will be used to show various properties of the algebraic template, e.g.~it is well-defined and has good extendability properties.

\begin{lemma}\label{lem:generic}
Given $s>r\ge 1$ and prime power $q$, suppose that $d\ge d_{\ref{lem:generic}}(r)$ and $\ell\ge\ell_{\ref{lem:generic}}(d,s)$. Then any matrix $N\in\mb{F}_{q^\ell}^{s\times r}$ which is $\mb{F}_q$-generic of degree $d$ satisfies the following:
\begin{itemize}
    \item For every $\Pi\in\mb{F}_q^{r\times s}$ of rank $r$, we have that $\Pi N$ is invertible.
    \item For every $\Pi\in\mb{F}_q^{r\times s}$ of rank $r$ and $y\in\mb{F}_q^{1\times s}$ with $y\notin\mr{row}_{\mb{F}_q}(\Pi)$, the row space, we have that $yN(\Pi N)^{-1}z\in\mb{F}_{q^\ell}^{1\times r}$ has no coordinates in $\mb{F}_q$ for all $z\in\mb{F}_q^r\setminus\{0\}$.
    \item For every choice of $\Pi_i,\Pi_i^\ast\in\mb{F}_q^{r\times s}$ of rank $r$ for $i\in\{1,2\}$ such that (a) there do not simultaneously exist $M_i\in\mr{GL}(\mb{F}_q^r)$ with $\Pi_i=M_i\Pi_i^\ast$ for $i\in\{1,2\}$, and (b) there is not $M\in\mr{GL}(\mb{F}_q^r)$ with $M\Pi_1^\ast=\Pi_2^\ast$, we have that the vector $(\Pi_1N)(\Pi_1^\ast N)^{-1}z_1-(\Pi_2N)(\Pi_2^\ast N)^{-1}z_2$ is nonzero for every choice of $z_1,z_2\in\mb{F}_q^r\setminus\{0\}$.
    \end{itemize}
Furthermore, such a matrix exists.
\end{lemma}
\begin{remark}
In the second bullet, if $y\in\mr{row}_{\mb{F}_q}(\Pi)$ then we can write $y=y'\Pi$ for some $y'\in\mb{F}_q^{1\times r}$ and thus $yN(\Pi N)^{-1}=y'\in\mb{F}_q^{1\times r}$. In the third bullet, if the condition (a) is not satisfied then $(\Pi_iN)(\Pi_i^\ast N)^{-1}=M_i$ and the vector of interest equals $M_1v_1-M_2v_1\in\mb{F}_q^r$, which can easily have zero entries.
\end{remark}
\begin{proof}
Let us treat $N=(x_{ij})_{i\in[s],j\in[r]}$ as a matrix of variables. We will show that failure of at least one of the bullet points implies at least one nontrivial polynomial relation over $\mb{F}_q$ of bounded degree (in terms of only $q,s$). This will immediately imply that any evaluation $N\in\mb{F}_{q^\ell}^{s\times r}$ which is $\mb{F}_q$-generic of degree $d$ satisfies all these properties. Furthermore, for $\ell>(d+1)^{rs}$ we can find $rs$ elements of $\mb{F}_{q^\ell}$ which together are $\mb{F}_q$-generic of degree $d$: simply consider $\{\alpha,\alpha^{(d+1)},\alpha^{(d+1)^2},\ldots,\alpha^{(d+1)^{rs-1}}\}$ where $\alpha$ generates $\mb{F}_{q^\ell}$ over $\mb{F}_q$ (which exists by the primitive element theorem). Indeed, when we plug in these values in any order into an $\mb{F}_q$-polynomial of degree at most $d$ in $rs$ variables, every monomial gives a different degree of $\alpha$ due to uniqueness of base $d+1$ expansion.

In this proof we will denote by $e_i\in\mb{F}_q^{1\times s}$ the row vector with $1$ in coordinate $i\in[s]$ and $0$ elsewhere (and sometimes abuse notation to refer to the restriction to $\mb{F}_q^{1\times r}$ when $i\in[r]$). We let $I_{r,s}\in\mb{F}_q^{r\times s}$ be the matrix with $e_i$ in its $i$th row for $i\in[r]$. Additionally, note that we may reduce to the cases $z=e_1$ by replacing $\Pi$ by $Q\Pi$ where $Q\in\mr{GL}(\mb{F}_q^r)$ is such that $Q^{-1}e_1=z$ in the second bullet point (this does not change the row space of $\Pi$), and similarly to $z_1=z_2=e_1$ for the third bullet point (the conditions (a) and (b) are unaffected). Now we study each bullet point in turn.

For the first bullet point, fix some $\Pi\in\mb{F}_q^{r\times s}$ of rank $r$ and note that we can write $\Pi=I_{r,s}M$ for some $M\in\mr{GL}(\mb{F}_q^s)$. Since $M$ is invertible, the elements of $N'=MN$ are linearly independent linear forms in the original variables composing $N$. Thus by shifting basis of the polynomial ring in question we may assume that the elements of $N'$ are our base variables. The condition fails if $\Pi N=I_{r,s}(MN)$ is noninvertible, i.e., $\det(I_{r,s}N')=0$. This is just the determinant of the first $r$ rows of $N'$, which gives a nontrivial polynomial relation of degree $r$.

For the second, again fix some appropriate $\Pi$ and now $y\notin\mr{row}_{\mb{F}_q}(\Pi)$. This implies that we can simultaneously write $\Pi=I_{r,s}M$ and $y=e_{r+1}M$ for some $M\in\mr{GL}(\mb{F}_q^s)$. We have $yN(\Pi N)^{-1}=e_{r+1}N'(I_{r,s}N')^{-1}$ where again $N'=MN$ has elements which are an invertible linear transformation of those of $N$ so can be treated as containing the base variables. Now the elements of $(I_{r,s}N')^{-1}$ can be obtained from the adjugate matrix via Cramer's rule while the elements of $e_{r+1}N'$ are completely disjoint from the defining variables of this. We thus easily see that for any $\alpha\in\mb{F}_q^{1\times r}$ each of the $r$ elements of $\det(I_{r,s}N')\cdot(e_{r+1}N'(I_{r,s}N')^{-1}-\alpha)$ is a nontrivial polynomial of degree $r$ (with degree $1$ in each of the variables of $e_{r+1}N'$).

The third is the most complex. Recall that we may focus on the situation that the first column is nonzero. Let $V_i=\mr{row}_{\mb{F}_q}(\Pi_i^\ast)$ for $i\in\{1,2\}$ and let $V=V_1\cap V_2$ and $t=\dim V$. Let $\{v_1,\ldots,v_t\}$ be an $\mb{F}_q$-basis for $V$ and extend this via $\{a_1,\ldots,a_{r-t}\}$ and $\{b_1,\ldots,b_{r-t}\}$ to respectively obtain bases for $V_1$ and $V_2$. Put together these vectors clearly generate the space $V_1+V_2$, and there are $2r-t=\dim V_1+\dim V_2-\dim(V_1\cap V_2)=\dim(V_1+V_2)$ of them, so they in fact form a basis of $V_1+V_2$ (this also implies $\max(0,2r-s)\le t\le r$). Now consider $M^\intercal\in\mr{GL}(\mb{F}_q^s)$ which maps $e_i$ to $v_i$ for $i\in[t]$, $e_{t+i}$ to $a_i$ for $i\in[r-t]$, and $e_{r+i}$ to $b_i$ for $i\in[r-t]$. We have
\[\Pi_1^\ast N = A_1J_1^\ast(MN),\quad\Pi_2^\ast N = A_2J_2^\ast(MN)\]
for some appropriate $A_i\in\mr{GL}(\mb{F}_q^r)$, where $J_1^\ast=I_{r,s}$ and $J_2^\ast$ has its $i$th row equal to $e_i$ for $i\in[t]$ and $e_{i+r-t}$ for $i\in\{t+1,\ldots,r\}$. This is since we can trivially deduce $\mr{row}_{\mb{F}_q}(J_i^\ast M)=\mr{row}_{\mb{F}_q}(\Pi_i^\ast)$ for $i\in\{1,2\}$, which implies these two rank $r$ matrices are related by left-multiplication of an invertible matrix. Furthermore write $\Pi_i=J_iM$ for $i\in\{1,2\}$, and again let $N'=MN$ and treat the elements of $N'$ as base variables for our polynomials. Note that
\begin{equation}\label{eq:inverse-difference}
(\Pi_1N)(\Pi_1^\ast N)^{-1}-(\Pi_2N)(\Pi_2^\ast N)^{-1}=(J_1N')(A_1J_1^\ast N')^{-1}-(J_2N')(A_2J_2^\ast N')^{-1}.
\end{equation}

By the first given condition and without loss of generality, we may assume that $\Pi_1$ and $\Pi_1^\ast$ are not related by left-multiplication of an invertible matrix. Equivalently, the same can be said of $J_1$ and $J_1^\ast=I_{r,s}$. This implies that $J_1\in\mb{F}_q^{r\times s}$ has a nonzero element outside the first $r$ columns. Suppose it is in the $(i^\ast,j^\ast)$ position (so $j^\ast\in\{r+1,\ldots,s\}$) and has value $\alpha\in\mb{F}_q^\times$ and let $z$ be the $j^\ast$th row of $N'$.

Now we split into two cases. In the first, we suppose $j^\ast\in\{r+1,\ldots,2r-t\}$. Let us focus on the element of \cref{eq:inverse-difference} which is in the $(i^\ast,1)$ position and specifically the influence of the variables from $z$. The first term in the difference \cref{eq:inverse-difference} contributes precisely the dot product of the $i^\ast$th row of $J_1N'$ and the first column of $(A_1J_1^\ast N')^{-1}$. Said row is equal to $\alpha z$ plus a vector composed only of linear combinations of variables from other rows of $N'$. Said column is a vector of nonzero rational functions only of the first $r$ rows of $N'$ with an explicit representation due to Cramer's rule. The second term contributes something analogous but with $J_2N'$ and $(A_2J_2^\ast N')^{-1}$. However, since $j^\ast\in\{r+1,\ldots,2r-t\}$, the $z$ influence looks qualitatively different. To be precise, the $i^\ast$th row of $J_2N'$ can be written as some (potentially zero) multiple of $z$ plus a linear combination $z'$ of other variables of $N'$. But we can find $y\in\mb{F}_q^{1\times r}$ with $yA_2=e_{j^\ast-(r-t)}\in\mb{F}_q^{1\times r}$. and hence from the definition of $J_2^\ast$ we deduce $z=yA_2J_2^\ast N'$. Then $z(A_2J_2^\ast N')^{-1}=y$. On the other hand, $z'(A_2J_2^\ast N')^{-1}$ only involves the variables of $z$ insofar as they are involved in the inverse matrix via Cramer's rule.

Ultimately, the condition that the $(i^\ast,1)$ element of \cref{eq:inverse-difference} equals $0$ leads to an equation of the form $(\alpha L(z)+\beta)-I(z)=0$ where (a) $L(z)$ is linear in $z$ with nonzero rational functions in the other variables as coefficients and $\beta$ is a potentially zero rational function serving as a constant term, and (b) $I(z)$ is a potentially zero rational function in $z$ with $z$-degree at most $0$ with denominator the nonzero polynomial $\det(A_2J_2^\ast N')=\det(A_2)\det(J_2^\ast N')$ of total degree $r$ and $z$-degree $1$. Multiplying over by $\det(A_1J_1^\ast N')\det(A_2J_2^\ast N')$ can be seen to therefore give a nontrivial polynomial relation due to our analysis of the $z$ variables, and the total degree is at most $2r$. We conclude consideration of this case.

The second case is when $j^\ast\in\{2r-t+1,\ldots,s\}$. We again consider the $(i^\ast,1)$ element of \cref{eq:inverse-difference} and its dependence on $z$. The first term contributes the same $L(z)$ as before, and now the second term is instead analogous to the first term and contributes something of the form $\alpha'L'(z)+\beta'$ with $L'(z),\beta'$ satisfying conditions analogous to $L(z),\beta$ and $\alpha'\in\mb{F}_q$ (the only difference is that unlike $\alpha$, the value $\alpha'$ can potentially be $0$). We thus derive an equation $\alpha L(z)+\beta=\alpha'L'(z)+\beta'$ of the described form. When we multiply over the determinants, this will become a polynomial relation of degree at most $2r$ and it remains to show nontriviality.

To do this, note that $\alpha L(z) = \alpha z(A_1J_1^\ast N)^{-1}e_1^\ast$ where $e_1^\ast\in\mb{F}_q^r$ is the obvious indicator vector and $\alpha'L'(z) = \alpha'z(A_2J_2^\ast N)^{-1}e_1^\ast$. If $\alpha'=0$ then clearly we have a nontrivial relation. Otherwise, in order to have a trivial relation we will need $\alpha(A_1J_1^\ast N)^{-1}e_1^\ast=\alpha'(A_2J_2^\ast N)^{-1}e_1^\ast$ so
\[e_1^\ast=\alpha^{-1}\alpha'(A_1J_1^\ast N)(A_2J_2^\ast N)^{-1}e_1^\ast.\]
But by the second given condition, $\Pi_1'=\alpha^{-1}A_1J_1^\ast$ and $\Pi_2'=A_2J_2^\ast$ are not equal up to left-multiplication by an invertible matrix. Therefore there is a row $y$ of $\Pi_1'$ with $y\notin\mr{row}_{\mb{F}_q}(\Pi_2')$, and we have
\[yN(\Pi_2'N)^{-1}e_1^\ast\in\{0,1\}.\]
This leads to a nontrivial polynomial relation of degree $r$ by our proof of the second bullet point.
\end{proof}

\subsection{Setup}\label{sub:setup}
Let $s > r\ge 1$ and let $V=\mb{F}_q^n$ be a vector space, with $q$ fixed and $n$ satisfying divisibility conditions as in \cref{thm:main}. Let $z = n^2$. Fix a positive integer $\ell$ to be chosen later, let $n-\ell m=\ell'\in\{0,\ldots,\ell-1\}$, and let $L=\mb{F}_{q^\ell}\hookrightarrow\mb{F}_{q^{\ell m}}=K$ be finite fields. Let $n$ be large with respect to $\ell,q,r,s$. Consider $z$ uniformly random injective $\mb{F}_q$-linear maps $\iota_1,\ldots,\iota_z\colon K\to V$. We define the embedding $\iota_i\colon K\hookrightarrow V$ by identifying $\mb{F}_q^{\ell m}$ arbitrarily as the additive structure of $K$ and then choosing a uniformly random full rank matrix $W_i\in\mb{F}_q^{n\times\ell m}$ to embed this into $V$. Explicitly, let $\alpha$ be a generator of $K^\times$ and define
\[\iota_i\bigg(\sum_{j=0}^{\ell m-1}v_j\alpha^j\bigg) = W_i\begin{bmatrix}v_0\\\vdots\\v_{\ell m-1}\end{bmatrix}\in\mb{F}_q^n=V.\]
Now let $K_i=\iota_i(K)$ and $L_i=\iota_i(L)$. (Note that there are $z=n^2$ different field structures present at once on various subsets of $V$, or equivalently a single field structure linearly embedded in various ways; it will always be obvious when we are, say, multiplying an element of $L_i$ and $K_i$ according to the field structure on $K_i$ given by $\iota_i$.) Finally, consider $d\ge 1$ and let $N_\mr{tem}\in L^{s\times r}$ and $x_\mr{abs}^\ast\in L^{s\times s}$ be jointly $\mb{F}_q$-generic of degree $d$.

We wish to find a subset of the $q$-system $\mr{Gr}(V,s)$ such that every $r$-space in the $q$-system $G:=\mr{Gr}(V,r)$ is contained in exactly one $s$-space. This can also be thought of as a hypergraph matching with vertex set $\mr{Gr}(V,r)$, where each $S\in\mr{Gr}(V,s)$ is thought of as a $\qbinom{s}{r}_q$-uniform edge consisting of its constituent $r$-subspaces.

\begin{remark}
There is a relation to ``special'' hypergraph decompositions as well, which will not be directly fruitful here but may lead to more general types of decomposition problems that may be of interest. The $q$-system $G=\mr{Gr}(V,r)$ can be thought of as the $(q^r-1)$-graph on $V\setminus\{0\}$ where the edges are precisely the sets of nonzero elements of $r$-dimensional subspaces. Let $M\in\mb{F}_q^{(q^s-1)\times s}$ be the matrix with every possible distinct nonzero row in $\mb{F}_q^s$ and let $H$ be a $(q^r-1)$-graph on $[q^s-1]$ whose edges are the $(q^r-1)$-tuples of rows corresponding to the nonzero elements of an $r$-dimensional subspace of $\mb{F}_q^s$.

An $M$-copy of $H$ in $G$ is a copy realized by some vector $Mx\in V^{q^s-1}$ with $x\in V^s$ (treating $M$ abusively as the base-changed operator $M\colon V^s\to V^{q^s-1}$). That is, for each $I\in E(H)$ the set of coordinates of $M_Ix\in V^{q^r-1}$ is an edge of $G$. An $(M,H)$-decomposition of $G$ is an $H$-decomposition of $G$ using $M$-copies of $H$. Given this setup, we see that this precisely corresponds to an $(n,s,r)_q$-design.

Beyond the hypergraph and lattice divisibility conditions which naturally arise for an $(M,H)$-decomposition in a more general setting, one also has conditions from linear algebra. For example, the $\mb{F}_q$-rank of any edge in $G$ must be at most the maximum $\mb{F}_q$-rank of any $M_I$ with $I\in E(H)$. Given our specific choices of $M,H$, and $G$, we see that the $M$-copies of $H$ in $G$ are those induced by $x\in V^s$ which are composed of $s$ linearly independent vectors.
\end{remark}

\subsection{Notation}\label{sub:notation}
We write $\mr{Gr}_q(n,m)$ for the Grassmannian, the set of $m$-dimensional $\mb{F}_q$-subspaces of $\mb{F}_q^n$. If $V$ is a vector space (the underlying field structure implicit) we alternatively write $\mr{Gr}(V,m)$ for the set of $m$-dimensional subspaces of $V$, and we also write $\mr{GL}(V)$ for the set of invertible linear maps $V\to V$. We write $V\leqslant W$ if $V$ is a vector subspace of $W$, typically over $\mb{F}_q$, and we write $V\leqslant_LW$ to signify that the subspace relation is treated with respect to an $L$-structure for the field $L$. For a field $\mf{K}$ we write $\dim_\mf{K}$ to give the $\mf{K}$-dimension of a vector space; if $\mf{K}=\mb{F}_q$ we typically drop the subscript. We write $\mr{span}_{\mf{K}}(S)$ where $S$ is a subset of a $\mf{K}$-vector space to be the $\mf{K}$-span, treated as a $\mf{K}$-vector space. In certain cases, $S$ may be a single vector or matrix with its elements coming from said $\mf{K}$-vector space, in which case we abusively treat this as taking the span of the elements (\emph{not} the column or row-span of the matrix). Let $\mr{row}_\mf{K}$ of a matrix be its row-space over $\mf{K}$.

Given a vector $\Phi$, we write $\Phi^+,\Phi^-$ for the portion of the vector with positive and negative coefficients, respectively, so that $\Phi=\Phi^++\Phi^-$. For example, $(1,-1)^+=(1,0)$ and $(1,-1)^-=(0,-1)$. Given an index $i$ in some index set $\mc{I}$, we write $e_i$ for the identity vector $e_i\in\mb{Z}^\mc{I}$ (mostly used for $r$-spaces and $s$-spaces in a Grassmannian).

We write $\mr{range},\mr{dom},\mr{supp}$ for the range of a function, its domain, and for the support of a vector, respectively. We write $|X|,\#X$ for the size of a set $X$, and $\mr{Sam}(X,p)$ for $p\in[0,1]$ denotes the distribution on subsets of $X$ created by independently including each element of $X$ with probability $p$. Given a hypergraph or $q$-system $H$ of uniformity $d$ on vertex set or vector space $V$, and given hypergraph or $q$-system $S$ of uniformity $d'\le d$, we write $H[S]$ for the resulting induced system: we keep a $d$-edge or $d$-space of $H$ precisely when all $d'$-subsets or $d'$-subspaces are in $S$. In the case $d'=1$, this corresponds to the traditional notion of an induced subgraph, and in this case as long as $S$ is a set or subspace, we further restrict the vertex set or underlying vector space to $S$.

Finally, we use asymptotic notation as follows: we write $f=O(g)$ or $f\lesssim g$ if $|f|\le Cg$ for some absolute constant $C>0$, and we write $f=\Omega(g)$ if $f\ge c|g|$ for some absolute constant $c>0$. We put subscripts such as $O_{q,s}$ to indicate that the absolute constant chosen depends on $q,s$ but nothing else. Furthermore, for parameters $\alpha,\beta$ we write $\alpha\ll\beta$ to mean $\beta$ is chosen to be at least some sufficiently large function of $\alpha$; this is read left-to-right.

\section{Subspace exchanges}\label{sec:exchanges}
We record the following \emph{subspace exchange} statement.
\begin{proposition}\label{prop:subspace-exchange}
Given $q$ and $s > r$, there exists $k = k_{\ref{prop:subspace-exchange}}(s)$ and two nonempty $s$-dimensional $q$-systems on $\mb{F}_q^k$ (i.e.~subsets of $\mr{Gr}_q(k,s)$), call them $\Ups$ and $\Ups'$, such that
\begin{itemize}
    \item For distinct $P,P'\in\Ups$ we have $\dim(P\cap P') < r$, and same for $P,P'\in\Ups'$;
    \item For $P\in\Ups$ and $P'\in\Ups'$ we have $\dim(P\cap P')\le r$;
    \item $\partial_{s,r}\Ups = \partial_{s,r}\Ups'$.
\end{itemize}
\end{proposition}
\begin{remark}
The first and third bullets imply these give $s$-space decompositions of some set of $r$-spaces. The second implies that the two decompositions do not intersect more than is strictly necessary, and in particular are disjoint as sets of $s$-spaces.
\end{remark}
\begin{proof}
Fix $k_1|k_2=k$ to be chosen later so that $\mb{F}_q\hookrightarrow X=\mb{F}_{q^{k_1}}\hookrightarrow Y = \mb{F}_{q^{k_2}}$ as finite fields. Let $d = d(r)$ be large to be chosen later larger than the degree of various determinants that arise in rank considerations later, e.g.~via \cref{lem:generic}. Let $u\ge 1$. Fix distinct $x^{(1)},x^{(2)}\in X^{s\times u}$ and $N\in X^{s\times r}$ such that $x^{(2)}-x^{(1)},N$ are jointly $\mb{F}_q$-generic of degree $d$. Consider any $w'\in Y^r$ and any $w\in Y^u$ which are together $X$-generic of degree $1$. These can be found as long as $k_1$ is large in terms of $d,s$ and $k_2/k_1$ is large in terms of $d,s$ (see e.g.~the first paragraph of the proof of \cref{lem:generic} and consider the field extensions $\mb{F}_{q^{k_1}}/\mb{F}_q$ and then $\mb{F}_{q^{k_2}}/\mb{F}_{q^{k_1}}$ in succession).

For $j\in\{1,2\}$ let $\mc{P}_j^{(w)} = \{\mr{span}_{\mb{F}_q}(Nw'+(Nx+x^{(j)})w)\colon x\in X^{r\times u}\}$. For one of these to not correspond to an element of $\mr{Gr}_q(k,s)$, we need that for some nonzero $y\in\mb{F}_q^{1\times s}$,
\[y(Nw' + (Nx+x^{(j)})w) = 0.\]
That is, $(yN)w'+y(Nx+x^{(j)})w=0$. Since $yN\in X^{1\times r}$ and $y(Nx+x^{(j)})\in X^{1\times u}$ and $w',w$ are jointly $X$-generic, we have $yN=0$ and $y(Nx+x^{(j)})=0$. But $y\in\mb{F}_q^{1\times s}$ is nonzero and $N$ is $\mb{F}_q$-generic so we have a contradiction.

We now let $\Ups = \mc{P}_1^{(w)}$ and $\Ups' = \mc{P}_2^{(w)}$.

We claim that as $x\in X^{r\times u}$ varies within a single $\mc{P}_j^{(w)}$, we have distinct $s$-spaces and in fact their intersection has dimension less than $r$, which verifies the first bullet point. If not, then $x\neq x'$ produce $s$-spaces which share an $r$-dimensional subspace. Thus there is an $r$-space which is within both the span of the coordinates of $Nw' + (Nx+x^{(j)})w$ and of $Nw'+(Nx'+x^{(j)})w$. Thus there are $\Pi_1,\Pi_2\in\mb{F}_q^{r\times s}$ with rank $r$ such that
\[\Pi_1(Nw'+(Nx+x^{(j)})w) = \Pi_2(Nw'+(Nx'+x^{(j)})w).\]
Thus $(\Pi_1N)w'+\Pi_1(Nx+x^{(j)})w = (\Pi_2N)w'+\Pi_2(Nx'+x^{(j)})w$. We have $\Pi_1N,\Pi_2N\in X^{r\times r}$ and $\Pi_1(Nx+x^{(j)}),\Pi_2(Nx'+x^{(j)})\in X^{r\times u}$. Since $w',w$ are $X$-generic we deduce $\Pi_1N=\Pi_2N$ and $\Pi_1(Nx+x^{(j)}) = \Pi_2(Nx'+x^{(j)})$. Now $\Pi_1,\Pi_2\in\mb{F}_q^{r\times s}$ have rank $r$ and $N$ is $\mb{F}_q$-generic so $\Pi_1 = \Pi_2$ and additionally $\Pi_1N\in X^{r\times r}$ has rank $r$ (from the first bullet of \cref{lem:generic}). Thus we first find $\Pi_1N(x-x') = 0$ and then deduce $x-x'=0$, a contradiction!

We now check the second bullet point. If there is an $(r+1)$-space which is shared between some elements of $\Ups$ and $\Ups'$ then similarly we can find $\Pi_1,\Pi_2\in\mb{F}_q^{(r+1)\times s}$ with rank $r+1$ and $x,x'\in X^{r\times u}$ (not necessarily distinct) such that
\[\Pi_1(Nw'+(Nx+x^{(1)})w) = \Pi_2(Nw'+(Nx'+x^{(2)})w).\]
Thus $(\Pi_1N)w'+\Pi_1(Nx+x^{(1)})w = (\Pi_2N)w'+\Pi_2(Nx'+x^{(2)})w$. Similarly, we deduce $\Pi_1N=\Pi_2N$ and $\Pi_1(Nx+x^{(1)})=\Pi_2(Nx'+x^{(2)})$ and then $\Pi_1=\Pi_2$. Thus $(\Pi_1N)(x-x') = \Pi_1(x^{(2)}-x^{(1)})$. Since $N$ is $\mb{F}_q$-generic, $\Pi_1N\in X^{(r+1)\times r} = X^{(r+1)\times r}$ has rank $r$ (again use \cref{lem:generic}, this time on an appropriate subset of rows). Let $a\in X^{1\times(r+1)}$ be nonzero and in its kernel, expressed as a polynomial in the coefficients of $N$ (by Cramer's rule). We thus have $a\Pi_1(x^{(2)}-x^{(1)}) = 0$. Since $N,x^{(2)}-x^{(1)}$ are jointly $\mb{F}_q$-generic, this is a contradiction.

Now we check the third bullet point, which guarantees that $\Ups,\Ups'$ provide the same subspace decomposition. Due to what we know so far, it suffices to show that every $r$-subspace of an $s$-space of $\Ups'$ can be found in some $s$-space of $\Ups$, and vice versa. Without loss of generality let us start with an $r$-space coming from $\Ups'$. This can be identified via a choice of $x'\in X^{r\times u}$ and $\Pi'\in\mb{F}_q^{r\times s}$ of full rank, and then considering the span of the elements of $\Pi'(Nw'+(Nx'+x^{(2)})w)$. We seek to identify $\Pi\in\mb{F}_q^{r\times s}$ of rank $r$ and $x\in X^{r\times u}$ with
\[\Pi(Nw'+(Nx+x^{(1)})w) = \Pi'(Nw'+(Nx'+x^{(2)})w),\]
which will finish. To this end we let $\Pi = \Pi'$ and let
\[x = x' + (\Pi'N)^{-1}\Pi'(x^{(2)}-x^{(1)}).\]
Again recall that $\Pi'N\in X^{r\times r} = X^{r\times r}$ is invertible since $N$ is $\mb{F}_q$-generic and $\Pi'\in\mb{F}_q^{r\times s}$ has rank $r$. This immediately rearranges to the desired condition.
\end{proof}

\section{Bounded integral decomposition}\label{sec:bounded-integral-decomposition}
\subsection{Preliminaries}\label{sub:bounded-decomposition-preliminaries}
We now define the notion of a bounded complex. In this section we will not directly need the setup in \cref{sub:setup}, and in particular we will (unambiguously) use $L$ to denote certain sparse random $q$-systems as opposed to a finite field.
\begin{definition}\label{def:bounded}
Define an $r$-dimensional signed multi-$q$-system $\mc{R}\subseteq\mr{Gr}(V,r)$ with $\dim V=n$ to be \emph{$\theta$-bounded} if $\snorm{\partial_{r,r-1}\mc{R}^\pm}_{\infty}\le\theta q^n$, where $\mc{R}^+$ denotes the positive part and $\mc{R}^-$ the negative part when $\mc{R}$ is written as a vector, so that $\mc{R}=\mc{R}^++\mc{R}^-$.
\end{definition}

The goal of this section is to show that any $\theta$-bounded vector $\mc{S}\in\mc{L} = \partial_{s,r}\mb{Z}^{\mr{Gr}_q(n,s)}$ has a preimage under $\partial_{s,r}$ whose positive and negative parts have $O(\theta)$-bounded images. (We can think of this as roughly meaning just that the preimage itself is bounded in some sense.)

\begin{theorem}\label{thm:bounded-inverse}
There are $C = C_{\ref{thm:bounded-inverse}}(q,s)$ and $c = c_{\ref{thm:bounded-inverse}}(q,s)$ so that for $n$ large, if $\theta\ge q^{-cn}$ and $J\in\partial_{s,r}\mb{Z}^{\mr{Gr}_q(n,s)}$ is $\theta$-bounded then there is $\Phi\in\mb{Z}^{\mr{Gr}_q(n,s)}$ with $\partial_{s,r}\Phi = J$ such that $\partial_{s,r}\Phi^\pm$ are $C\theta$-bounded.
\end{theorem}

The first ingredient in our proof is robust local decodability of the lattice $\mc{L}$; this is analogous to \cite[Lemma~5.13]{Kee14}.

\begin{proposition}\label{prop:robust-local-decodability}
There are $C = C_{\ref{prop:robust-local-decodability}}(q,s)$ and $\Delta = \Delta_{\ref{prop:robust-local-decodability}}(q,s,r)$ so that for $n$ large, if $\theta\ge q^{-n/2}$ and $J\in\Delta\mb{Z}^{\mr{Gr}_q(n,r)}$ is $\theta$-bounded there is $\Phi\in\mb{Z}^{\mr{Gr}_q(n,s)}$ such that $\partial_{s,r}\Phi = J$ and $\partial_{s,r}\Phi^\pm$ are $C\theta$-bounded.
\end{proposition}
\begin{proof}
It is enough to prove the result when $J=J^+$ has nonnegative coefficients. Let $R_0 = \mr{span}_{\mb{F}_q}(e_1,\ldots,e_r)$ and $T_0 = \mr{span}_{\mb{F}_q}(e_1,\ldots,e_{r+s})$. Let $\mc{R}_0 = \mr{Gr}_q(r+s,r)$ be the set of $r$-dimensional spaces $R\leqslant T_0$ and $\mc{S}_0 = \mr{Gr}_q(r+s,s)$ be the set of $s$-dimensional spaces $S\leqslant T_0$.

Consider the matrix $A$ where rows are indexed by elements of $\mc{R}_0$ and columns are indexed by $\mc{S}_0$ with an entry being $1$ if the column index (as a subspace) contains the row index. This is a square matrix of dimension $\qbinom{r+s}{r}_q = \qbinom{r+s}{s}_q$. By a theorem of Kantor \cite{Kan72}, we have that $\Delta := |\det A|\neq 0$, i.e., $A$ is invertible. Thus by Cramer's rule, we can write $\Delta e_{R_0}$ as an integer combination of the columns. Extending these vectors to $\mb{Z}^{\mr{Gr}_q(n,r)}$ shows that there are explicit $a_S\in\mb{Z}$ (which are up to sign determinants of some submatrices of $A$) such that
\begin{equation}\label{eq:local-decodability}
\Delta e_{R_0} = \sum_{S\in\mc{S}_0}a_S\partial_{r,s}e_S.
\end{equation}

Now write
\[J = \Delta\sum_{R\in\mc{J}}e_R\]
where $\mc{J}$ is an appropriately defined multiset. For each $R\in\mc{J}$ we choose a uniformly random $B_R\in\mr{GL}(\mb{F}_q^n)$ mapping $R_0\mapsto R$ (here we abusively treat multiple copies of an $r$-space $R$ as different). Then by \cref{eq:local-decodability},
\[J = \sum_{R\in\mc{J}}\Delta e_R = \sum_{R\in\mc{J}}\sum_{S\in B_R\mc{S}_0}a_{B_R^{-1}S}\partial_{r,s}e_S = \partial_{r,s}\sum_{S\in\mr{Gr}_q(n,s)}\sum_{R\in\mc{J}}\bigg(\sum_{S'\in\mc{S}_0}\mbm{1}_{S=B_RS'}a_{S'}\bigg)e_S =: \partial_{r,s}\Phi,\]
where $\Phi$ is an appropriately defined random integer vector. Now
\begin{align*}
\snorm{\partial_{r,r-1}\partial_{s,r}\Phi^\pm}_\infty\lesssim_{q,s}\sup_{Q\in\mr{Gr}_q(n,r-1)}\sum_{R\in\mc{J}}\mbm{1}_{Q\leqslant B_RT_0} =: \sup_{Q\in\mr{Gr}_q(n,r-1)}Y_Q
\end{align*}
since the weight on an $(r-1)$-space $Q$ is bounded (up to a factor of $\max_{S\in\mc{S}_0}a_S = O_{q,s}(1)$) by the number of times it occurs within one of the random spaces $B_RT_0$. For any $Q$, $Y_Q$ is a sum of independent indicator random variables. When $\dim(Q\cap R) = t\in\{0,\ldots,r-1\}$, the indicator has probability bounded by $O_{q,s}(q^{-(r-1-t)n})$ and there are at most $O_{q,s}(\theta q^{(r-t)n})$ such $r$-spaces $R$ since $\mc{J}$ is $\theta$-bounded. Thus $\mb{E}Y_Q = O_{q,s}(\theta q^n)$, and Chernoff (applied to each value of $t$ separately) thus shows $\mb{P}[Y_Q\ge C\theta q^n]\le\exp(-\Omega(\theta q^n))$. Taking a union bound over at most $q^{(r-1)n}$ many $(r-1)$-spaces $Q$ and using $\theta\ge q^{-n/2}$ finishes.
\end{proof}

We next record a technical ``flattening lemma'' which will let us assume that our lattice element $J$ is not too focused, which will be useful later. It follows from the previous lemma (in particular ``flattening'' is in general a trivial consequence of robust local decodability). 
\begin{lemma}\label{lem:J-flattening}
Let $\Delta = \Delta_{\ref{prop:robust-local-decodability}}(q,s,r)$ as in \cref{prop:robust-local-decodability}. There is $C = C_{\ref{lem:J-flattening}}(q,s)$ so that for $n$ large, if $\theta\ge q^{-n/2}$ and $J\in\mb{Z}^{\mr{Gr}_q(n,r)}$ is $\theta$-bounded then there is $\Phi\in\mb{Z}^{\mr{Gr}_q(n,s)}$ with $\partial_{s,r}\Phi = J-J'$ such that $J',\partial_{s,r}\Phi^\pm$ are $C\theta$-bounded and $|J_R'|\le\Delta$ for all $R\in\mr{Gr}_q(n,r)$.
\end{lemma}
\begin{proof}
Write $J = J^++J^-$ and let $J_0 = \Delta\lfloor J^+/\Delta\rfloor+\Delta\lceil J^-/\Delta\rceil$, where the floor and ceiling are applied coordinate-wise in the standard basis. That is, $J_0$ is the result of finding a multiple of $\Delta$ by ``rounding towards zero''. Clearly $J_0$ is also $\theta$-bounded, and by \cref{prop:robust-local-decodability} we can write $J_0 = \partial_{s,r}\Phi$ with $\partial_{s,r}\Phi^\pm$ being $O_{q,s}(\theta)$-bounded. We have that $J'=J-J_0$ is also $\theta$-bounded and $|J_R'|\le\Delta$ for all $R\in\mr{Gr}_q(n,r)$, as desired.
\end{proof}
We also record a technical ``boundedness lemma'' which shows that a reasonable lattice element $J$ has at most the expected number of extensions (of a specific type) from $J$ into a random host.
\begin{lemma}\label{lem:J-boundedness}
Let $h\ge 1$. There is $c = c_{\ref{lem:J-boundedness}}(h,q,r)$ so that for $n$ large and $q^{-cn}\le p\le 1/2$, the following holds. First, whp $L\sim\mr{Sam}(\mr{Gr}_q(n,r),p)$ is $(q^{-n/20},h)$-typical. Second, if $\theta\ge q^{-n/20}$ and $J\in\mb{Z}^{\mr{Gr}_q(n,r)}$ is $\theta$-bounded with $|J_R|\le q^{n/10}$ for all $R\in\mr{Gr}_q(n,r)$ then whp $J$ is $(2\theta, h)$-bounded wrt $L$.
\end{lemma}
\begin{proof}
Fix an $r$-dimensional $q$-extension $E = (\phi,F,H\setminus\{R\})$ to $\mb{F}_q^n$ with $\dim\mr{Vec}(H)\le h$ and $R\in H\setminus H[F]$; there are $O_{h,q,r}(q^{(\dim F)n})$ possible choices of $E$. We will apply \cref{lem:azuma} to $X_E^R(L,J)$. The first part of the lemma concerns $J=\mbf{1}$, i.e., it is $1$ in all coordinates.

Consider a given $r$-space $R'$ and whether it is included in $L$. If $R'\leqslant F$, it cannot affect $X_E^R(L,J)$. Now for any $k\in\{0,\ldots,\min(\dim F,r-1)\}$, there are at most $O_{q,r,h}(q^{(r-k)n})$ $r$-spaces $R'$ which intersect $F$ in dimension $k$. Changing whether or not $R'$ is included in $L$ can change the count $X_E^R(L,J)$ by at most $q^{n/10}\cdot O_{q,r,h}(q^{(v_E-(r-k))n}) = O_{q,r,h}(q^{(v_E-r+k+1/10)n})$. Therefore the variance proxy for $X_E^R(L,J)$ when applying \cref{lem:azuma} is, up to constants, bounded by $\sum_{0\le k\le\min(\dim F,r-1)}(q^{(v_E-r+k+1/10)n})^2q^{(r-k)n}$. Thus it is at most $q^{(2v_E-3/4)n}$ for $n$ sufficiently large. Furthermore $d(L) = (1\pm q^{-n/10})p$ with high probability by Chernoff.

Finally, we have
\[\mb{E}X_E^{R}(L,J) = \mb{E}\bigg[\sum_{\phi^\ast\in\mc{X}_E(L)}|J_{\phi^\ast(R)}|\bigg] = p^{e_{E}}\sum_{\phi^\ast\in\mc{X}_E(\mr{Gr}_q(n,r))}|J_{\phi^\ast(R)}|\le\theta(1+2q^{-n/10})^{e_E}d(L)^{e_E}q^{v_E n},\]
since $\dim(R\cap F)\le r-1$. Therefore by \cref{lem:azuma} and taking a union bound of size $O_{h,q,r}(q^{(\dim F)n})$, the second part of the result follows. The first part follows by noting that $\mb{E}X_E^R(L,J)$ for $J=\mbf{1}$ evaluates to $(1\pm q^{-n/3})p^{e_E}q^{v_En}$.
\end{proof}

\subsection{Lattices and sparse subsets}\label{sub:bounded-sparse}
We now construct the crucial object for our proof: we take a random sparse subset of $\mr{Gr}_q(n,r)$ and prove that the associated lattice in some sense is ``mostly generated'' by a sparse basis set. The strategy for \cref{thm:bounded-inverse} at a high level would be to find a representation of $J\pmod{\Delta}$ (where $\Delta = \Delta_{\ref{prop:robust-local-decodability}}(q,s,r)$) by starting with an arbitrary representation, covering the parts outside our sparse subset in a ``spread out'' way, and then decomposing the parts inside with the sparse basis set. Since we are working $\imod{\Delta}$, the coefficients are small and the representation is bounded by design. If we take the least residues $\imod{\Delta}$ and then consider the image over $\mb{Z}$, we obtain some bounded $J^\ast$ such that $J-J^\ast\in\Delta\mb{Z}^{\mr{Gr}_q(n,r)}$, at which point we can use \cref{prop:robust-local-decodability}. However, the fact that we can only obtain a sparse ``almost basis'' leads to additional difficulties; at a high level we take multiple independent copies of this sparse object and show that together they can ``cover the gaps'' with high probability.

\begin{proposition}\label{prop:greedy-sparse-basis}
Let $C = C_{\ref{prop:greedy-sparse-basis}}(q,s)$, $c = c_{\ref{prop:greedy-sparse-basis}}(q,s)$, $q^{-cn}\le p\le c$, and $n$ be sufficiently large. Given $\Delta\ge 1$, let $L\sim\mr{Sam}(\mr{Gr}_q(n,r),p)$. Whp, there is a set $\mc{S}\subseteq\mr{Gr}_q(n,s)[L]$ (i.e., $s$-spaces with all $r$-subspaces in $L$) which is such that (a) each $Q\in\mr{Gr}_q(n,r-1)$ has at most $p^{1/2}q^n$ extensions to an $s$-space in $\mc{S}$ and (b) there is $B\subseteq\mr{Gr}_q(n,r-1)$ with $|\mr{Gr}_q(n,r-1)\setminus B|\le C\Delta p^{1/2}q^{(r-1)n}$ and
\[\partial_{s,r}(\mb{Z}/\Delta\mb{Z})^\mc{S}\supseteq\partial_{s,r}(\mb{Z}/\Delta\mb{Z})^{\mr{Gr}_q(n,s)[L[B]]}.\]
\end{proposition}
\begin{proof}
Consider the following greedy process: start with $\mc{S}$ empty, and continually add in an arbitrary $S\in\mr{Gr}_q(n,s)[L]$ such that (i) $\partial_{s,r}e_S\notin\partial_{s,r}(\mb{Z}/\Delta\mb{Z})^\mc{S}$ and (ii) adding it to $\mc{S}$ does not make any $(r-1)$-space $Q$ have more than $p^{1/2}q^n$ extensions to an element in $\mc{S}$.

At the termination of this process, either (i) fails for all remaining subspaces so we have an appropriate (taking $B = \mr{Gr}_q(n,r-1)$) subset with $\partial_{s,r}(\mb{Z}/\Delta\mb{Z})^\mc{S}\supseteq\partial_{s,r}(\mb{Z}/\Delta\mb{Z})^{\mr{Gr}_q(n,s)[L]}$ and we are done, or (ii) fails. In the latter case, there is a subset $B'\subseteq\mr{Gr}_q(n,r-1)$ of ``saturated'' $(r-1)$-spaces $Q$ such that every $S\in\mr{Gr}_q(n,s)[L]$ with $\partial_{s,r}e_S\notin\partial(\mb{Z}/\Delta\mb{Z})^\mc{S}$ contains an element of $B'$, and every space in $B'$ is contained within $\lfloor p^{1/2}q^n\rfloor$ many $s$-spaces in $\mc{S}$. Let $B = \mr{Gr}_q(n,r-1)\setminus B'$.

Now we have the bound $\lfloor p^{1/2}q^n\rfloor|B'|\le\qbinom{s}{r-1}_q|\mc{S}|$ by counting pairs of $s$-spaces of $\mc{S}$ and $(r-1)$-spaces, one containing the other. We also have $|\mc{S}|\le\Delta\cdot|L|$ since each newly added element of $\mc{S}$ strictly extends the subgroup spanned by the $\partial_{s,r}e_S\in(\mb{Z}/\Delta\mb{Z})^L$, and the longest chain of such subgroups is bounded in size by $\Delta|L|$. Since $|L|\le 2pq^{rn}$ whp, the result follows.
\end{proof}

We now create a ``regularized'' subset $L^\ast\subseteq L[B]$ such that every $r$-space in $L^\ast$ has around the same number of extensions to an $s$-space all of whose $r$-subspaces are in $L[B]$.

\begin{lemma}\label{lem:regularized-linear}
Let $c = c_{\ref{lem:regularized-linear}}(q,s)$, $q^{-cn}\le p\le c$, and $n$ be sufficiently large. Given $\Delta\ge 1$, sample $L$ and create $\mc{S}$ and $B$ as in \cref{prop:greedy-sparse-basis} (existing whp). Whp there exists a collection $L^\ast\subseteq L[B]$ with $|L^\ast| = (1\pm\Delta p^{1/5})p\qbinom{n}{r}_q$ such that every $R\in L^\ast$ has $(1\pm p^{1/5})p^{\qbinom{s}{r}_q-1}\qbinom{n-r}{s-r}_q$ extensions to an $s$-space in $\mr{Gr}_q(n,s)[L[B]]$.
\end{lemma}
\begin{proof}
First, by \cref{lem:J-boundedness} whp $|L| = (1\pm q^{-n/20})p\qbinom{n}{r}_q$, every $Q\in\mr{Gr}_q(n,r-1)$ has $(1\pm q^{-n/20})p^{\qbinom{s}{r}_q}\qbinom{n-(r-1)}{s-(r-1)}_q$ extensions to some $S\in\mr{Gr}_q(n,s)[L]$, every $r$-space $R\in L$ has $(1\pm q^{-n/20})p^{\qbinom{s}{r}_q-1}\qbinom{n-r}{s-r}_q$ extensions to some $S\in\mr{Gr}_q(n,s)[L]$, and finally $|\mr{Gr}_q(n,s)[L]|=(1\pm q^{-n/20})p^{\qbinom{s}{r}_q}\qbinom{n}{s}_q$. In particular every $Q\in\mr{Gr}_q(n,r-1)$ has at most $2p^{\qbinom{s}{r}_q}\qbinom{n-(r-1)}{s-(r-1)}_q$ extensions into an $s$-space in $\mr{Gr}_q(n,s)[L]$. Let $L^\ast\subseteq L$ be the $r$-spaces with $(1\pm p^{1/5})p^{\qbinom{s}{r}_q-1}\qbinom{n-r}{s-r}_q$ extensions to an $s$-space in $\mr{Gr}_q(n,s)[L[B]]$. Clearly $L^\ast\subseteq L[B]$.

We now count $X$, the number of pairs $(Q,S)$ of $(r-1)$-spaces $Q\in\mr{Gr}_q(n,r-1)\setminus B$ and $s$-spaces $S\geqslant Q$ with $S\in\mr{Gr}_q(n,s)[L]$. We have by \cref{prop:greedy-sparse-basis} that
\[|L\setminus L^\ast|\cdot(p^{1/5}-q^{-n/20})p^{\qbinom{s}{r}_q-1}\qbinom{n-r}{s-r}_q\le\qbinom{s}{r}_qX\le\qbinom{s}{r}_q\cdot C_{\ref{prop:greedy-sparse-basis}}\Delta p^{1/2}q^{(r-1)n}\cdot 2p^{\qbinom{s}{r}_q}\qbinom{n-(r-1)}{s-(r-1)}_q.\]
The first inequality is deduced by considering any element $R\in L\setminus L^\ast$ and using the definition of $L^\ast$ and initial bounds on extensions within $L$ to find an extension to an $s$-space containing an $(r-1)$-space within $\mr{Gr}_q(n,r-1)\setminus B$ (this process can overcount by at most a factor of $\qbinom{s}{r}_q$ corresponding to the choice of $R$ given $(Q,S)$). The second inequality follows by starting with the bound on $\mr{Gr}_q(n,r-1)\setminus B$ from \cref{prop:greedy-sparse-basis} and then using initial bounds on extensions within $L$.

Dividing both sides and using the bounds on $p$ shows that
\[|L\setminus L^\ast|\lesssim_{q,s}\Delta p^{3/10}\cdot p\qbinom{n}{r}_q.\]
Combining with the initial bound $|L| = (1\pm q^{-n/20})p\qbinom{n}{r}_q$ finishes.
\end{proof}

We now fix $p = q^{-\xi n}$, $\xi > 0$ to be chosen small later, $\Delta = \Delta_{\ref{prop:robust-local-decodability}}(q,s,r)$ and fix a given realization of $\mc{S}$, $B$, $L[B]$, and $L^{\ast}[B]$ satisfying the conditions of \cref{prop:greedy-sparse-basis,lem:regularized-linear} with these values of $\Delta$, $p$, and $n$. We will consider a set of ``rotated'' collections defined by applying the action of $A_1,\ldots,A_u\in\mr{GL}(\mb{F}_q^n)$. Thinking of the indices $[u]$ as colors, we first show that certain colored configurations can be generated by the images of monochromatic $s$-spaces.

\begin{proposition}\label{prop:color-decomposition}
Let $0 < \xi\le\xi_{\ref{prop:color-decomposition}}(q,s)$, fix some positive integers $u\ge u_{\ref{prop:color-decomposition}}(\xi,q,s)$ and $\Delta\ge 1$, and suppose $n$ is large with respect to these parameters. Let $p = q^{-\xi n}$, and then let $L,\mc{S},B$ be as in \cref{prop:greedy-sparse-basis} and $L^\ast$ be as in \cref{lem:regularized-linear} (all existing whp). Sample $A_1,\ldots,A_u\in\mr{GL}(\mb{F}_q^n)$ uniformly at random. Whp, the following two properties hold.
\begin{itemize}
    \item Suppose $S\in\mr{Gr}_q(n,s)$ and there exists $\psi\colon\mr{Gr}(S,r)\to[u]$ so that $R\in A_{\psi(R)}\cdot L^\ast$ for all $R\in\mr{Gr}(S,r)$. Then there exists a representation
    \[\partial_{s,r}e_S = \sum_{S'\in\mc{I}}a_{S'}\partial_{s,r}e_{S'}\]
    such that for all $S'\in\mc{I}$, $a_{S'}\in\mb{Z}$ and $\mr{Gr}(S',r)\subseteq A_i\cdot L[B]$ for some $i\in[u]$.
    \item Suppose $S,S^\ast\in\mr{Gr}_q(n,s)$ with $\dim(S\cap S^\ast) = r$ and there exists $\psi\colon(\mr{Gr}(S,r)\cup\mr{Gr}(S^\ast,r))\setminus\{S\cap S^\ast\}\to[u]$ so that $R\in A_{\psi(R)}\cdot L^\ast$ for all $R$ in the domain. Then there exists a representation
    \[\partial_{s,r}(e_S-e_{S^\ast}) = \sum_{S'\in\mc{I}}a_{S'}\partial_{s,r}e_{S'}\]
    such that for all $S'\in\mc{I}$, $a_{S'}\in\mb{Z}$ and $\mr{Gr}(S',r)\subseteq A_i\cdot L[B]$ for some $i\in[u]$.
\end{itemize}
\end{proposition}
\begin{remark}
A matrix in $\mr{GL}(\mb{F}_q^n)$ is applied to a set of subspaces of $\mb{F}_q^n$ by extending the action and applying the operation element-wise in the natural way.
\end{remark}

To prove this, we first show a ``rainbow'' version of the first bullet point.
\begin{lemma}\label{lem:rainbow-decomposition}
Let $0 < \xi\le\xi_{\ref{lem:rainbow-decomposition}}(q,s)$, fix some positive integers $u\ge u_{\ref{lem:rainbow-decomposition}}(\xi,q,s)$ and $\Delta\ge 1$, and suppose $n$ is large with respect to these parameters. Let $p = q^{-\xi n}$, and then let $L,\mc{S},B$ be as in \cref{prop:greedy-sparse-basis} and $L^\ast$ be as in \cref{lem:regularized-linear} (all existing whp). Sample $A_1,\ldots,A_u\in\mr{GL}(\mb{F}_q^n)$ uniformly at random. Whp, the following property holds:
\begin{itemize}
    \item Suppose $S\in\mr{Gr}_q(n,s)$ and there exists an injective $\psi\colon\mr{Gr}(S,r)\to[u]$ so that $R\in A_{\psi(R)}\cdot L^\ast$ for all $R\in\mr{Gr}(S,r)$. Then there exists a representation
    \[\partial_{s,r}e_S = \sum_{S'\in\mc{I}}a_{S'}\partial_{s,r}e_{S'}\]
    such that for all $S'\in\mc{I}$, $a_{S'}\in\mb{Z}$ and $\mr{Gr}(S',r)\subseteq A_i\cdot L[B]$ for some $i\in[u]$.
\end{itemize}
\end{lemma}
\begin{proof}
Let $H$ be a $s$-dimensional $q$-system defined as follows. Let $b=\qbinom{s}{r}_q$. Take one copy of the construction in \cref{prop:subspace-exchange}, with underlying $k$-dimensional vector space $W$ and with collections of $s$-spaces $\Ups_0,\Ups_0'$. Choose some $F\in\Ups_0$, and let $F_j$ for $1\le j\le b$ be the $s$-spaces in $\Ups_0'$ intersecting $F$ in $r$ dimensions. Then glue $b$ additional copies of the construction from \cref{prop:subspace-exchange}, call them $W_j$ for $j\in[b]$, linearly disjointly along each $F_j$, with collections of $s$-spaces $\Ups_j,\Ups_j'$ labeled such that $F_j\in\Ups_j$. Let $F_j'$ be the unique $s$-space in $\Ups_j'$ which contains $F\cap F_j$. Thus if $k=k_{\ref{prop:subspace-exchange}}(s)$, we see that $\dim\mr{Vec}(H)=k+b(k-s)$. Finally, let the $s$-spaces of $H$ be
\[\Ups_0\cup(\Ups_0'\setminus\{F_1,\ldots,F_b\})\cup\bigcup_{j=1}^b\Big((\Upsilon_j\setminus\{F_j\})\cup\Upsilon_j'\Big).\]
The key point here is that applying \cref{prop:subspace-exchange} shows $\partial_{s,r}\Ups_0=\partial_{s,r}\Ups_0'$ and $\partial_{s,r}\Ups_j=\partial_{s,r}\Ups_j'$ for all $j\in[b]$. Therefore, we can first express $\partial_{s,r}e_F$ in terms of $\partial_{s,r}e_S$ for the $s$-spaces in $\Ups_0'$ and $\Ups_0\setminus\{F\}$. Then we can express $\partial_{s,r}e_{F_j}$ in terms of $\partial_{s,r}e_S$ for the $s$-spaces in $\Ups_j'$ and $\Ups_j\setminus\{F_j\}$ for each $j\in[b]$. Overall this provides an expression for $\partial_{s,r}e_F$ in terms of $\partial_{s,r}e_S$ for $S\in H\setminus H[F]=H\setminus\{F\}$.

Finally, let $S_0=F$ and fix a basis $\{v_1,\ldots,v_s\}$ for $S_0$ (within $\mr{Vec}(H)$). For $R\in\mr{Gr}(S_0,r)$ if $j\in[b]$ is the unique index so that $F_j\cap F=R$, define $S_R=F_j'$. We can see this is the unique $s$-space of $H$ satisfying $S_0\cap S_R=R$. Let the remaining $s$-spaces in $H\setminus(\{S_0\}\cup\{S_R\colon R\in\mr{Gr}(S_0,r)\})$ be labeled as $S_1,\ldots,S_t$, where clearly $t\le\qbinom{k+b(k-s)}{s}_q$.

Consider any outcomes of $L,\mc{S},B,L^\ast$ satisfying the conditions of \cref{prop:greedy-sparse-basis,lem:regularized-linear} for the given parameters as well as \cref{lem:J-boundedness} for $J=\mbf{1}$, and now only consider the randomness of $A_1,\ldots,A_u$. For any $S\in\mr{Gr}_q(n,s)$ and injective $\psi\colon\mr{Gr}(S,r)\to[u]$ let $\mc{E}_{S,\psi}$ be the event that (a) $R\in A_{\psi(R)}\cdot L^\ast$ for all $R\in\mr{Gr}(S,r)$ and (b) a desired representation as in the lemma statement does not exist. It suffices to give a sufficiently strong bound on $\mb{P}[\mc{E}_{S,\psi}]$ so that we can take a union bound.

To this end, let us further condition on each $A_{\psi(R)}^{-1}$ for all $R\in\mr{Gr}(S,r)$. In order to contribute to $\mc{E}_{S,\psi}$ we may assume that $R\in A_{\psi(R)}\cdot L^\ast$ occurs for all $R\in\mr{Gr}(S,r)$ (and this property can be deduced given this information). From now on, everything other than the remaining conditional randomness will be taken to be fixed. Thus $A_i$ for $i\notin\psi(\mr{Gr}(S,r))$ are uniformly random elements of $\mr{GL}(\mb{F}_q^n)$ while $A_{\psi(R)}^{-1}$ are fixed.

Now fix an $\mb{F}_q$-basis $\{v_1',\ldots,v_s'\}$ for $S$ (within $\mb{F}_q^n$). Consider some sequence $T = (u_1,\ldots,u_t)\in[u]^t$ of distinct colors that are not in $\psi(\mr{Gr}(S,r))$, which exists as long as $u$ is large enough. Consider $Y_T$, the number of $q$-embeddings $\phi\colon\mr{Vec}(H)\to\mb{F}_q^n$ of $H$ in $\mr{Gr}_q(n,s)$ such that:
\begin{itemize}
    \item $\phi(S_0) = S$ and in fact $\phi(v_i) = v_i'$ for all $i\in[s]$.
    \item For each $R\in\mr{Gr}(S_0,r)$, we have $\phi(S_R)\in\mr{Gr}_q(n,s)[A_{\psi(\phi(R))}\cdot L[B]]$.
    \item For each $i\in[t]$ we have $\phi(S_i)\in\mr{Gr}_q(n,s)[A_{u_i}\cdot L[B]]$.
\end{itemize}
We now compute the mean and variance of $Y_T$.

First note that by \cref{lem:regularized-linear} there are at least
\[\frac{1}{\qbinom{s}{r}_q}(1-\Delta p^{1/5})p\qbinom{n}{r}_q\cdot(1-p^{1/5})p^{\qbinom{s}{r}_q-1}\qbinom{n-r}{s-r}_q\ge(1-2\Delta p^{1/5})p^{\qbinom{s}{r}_q}\qbinom{n}{s}_q\]
many $s$-spaces in $\mr{Gr}_q(n,s)[L[B]]$, counting by starting with an element of $L^\ast$ and extending it to an $s$-space all of whose $r$-subspaces are in $L[B]$. On the other hand \cref{lem:J-boundedness} shows $|\mr{Gr}_q(n,s)[L]|=(1\pm q^{-n/20})p^b\qbinom{n}{s}_q$, which provides a matching upper bound. So $|\mr{Gr}_q(n,s)[L[B]]|=(1\pm2\Delta p^{1/5})p^b\qbinom{n}{s}_q$. Additionally, every $R\in L^\ast$ has $(1\pm p^{1/5})p^{b-1}\qbinom{n-r}{s-r}_q$ extensions to an $s$-space in $\mr{Gr}_q(n,s)[L[B]]$.

Extend $\{v_1,\ldots,v_s\}$ to an $\mb{F}_q$-basis $\{v_1,\ldots,v_d\}$ for $\mr{Vec}(H)$ (note $d = O_{q,s}(1)$) with the additional property that for $j\in[b]$, we have $\mr{Vec}(W_j)=\mr{span}_{\mb{F}_q}\{F_j,v_{s+(k-s)(j-1)+1},\ldots,v_{s+(k-s)j}\}$. This is possible since in the definition of $H$, we glued the $b$ additional copies of the construction from \cref{prop:subspace-exchange} (namely, the $W_j$) linearly disjointly along each $F_j$, and the $W_j$ therefore can be seen to satisfy $\dim(\mr{Vec}(W_j)\cap U_j)=r$ where
\[U_j:=\mr{span}_{\mb{F}_q}\bigg(W\cup\bigcup_{j'\neq j}\mr{Vec}(W_{j'})\bigg).\]
We in fact guarantee the further additional property that
\[F_j'=\mr{span}_{\mb{F}_q}(F_j\cap F,v_{s+(k-s)(j-1)+1},\ldots,v_{s+(k-s)(j-1)+s-r}\}\]
for $j\in[b]$, which is evidently possible. Here we are using $F_j'\cap U_j=F_j'\cap F_j=F_j\cap F$ and hence also $\dim(F_j'\cap F_j)=r$ due to definition (specifically, this dimension is not bigger due to the second bullet point of \cref{prop:subspace-exchange}).

We can specify $\phi$ satisfying the first bullet above by additionally specifying where $\{v_{s+1},\ldots,v_d\}$ map into vectors of $\mb{F}_q^n$ which are jointly linearly independent with $\{v_1',\ldots,v_s'\}$. We do this in stages, starting with $\{v_{s+(k-s)(j-1)+1},\ldots,v_{s+(k-s)j}\}$ in order for $j\in[b]$ and then ending with $\{v_{s+(k-s)b+1},\ldots,v_d\}$. Let $\mc{Y}$ be the set of such maps with the additional property that for every $R\in\mr{Gr}(S_0,r)$, the unique $j\in[b]$ such that $F_j'=S_R$ satisfies $\mr{Gr}(\phi(F_j'),r)\subseteq A_{\psi(\phi(R))}\cdot L^\ast$ (which guarantees the second bullet point above). This is a deterministic set given the revealed information, and we see $|\mc{Y}| = (1\pm p^{1/6})(p^{b-1})^bq^{(d-s)n}$ for $n$ large, using that $R\in A_{\psi(R)}\cdot L^\ast$ for all $R\in\mr{Gr}(S,r)$ as well as the above extension property of $L^\ast$. Finally, for each map $\phi\in\mc{Y}$ we let $\mbm{1}_\phi$ be $1$ if the third bullet point holds for $\phi$ and $0$ otherwise, which is now purely a function of the randomness of $A_{u_i}$ for $i\in[t]$.

We have
\[\mb{E}Y_T = \sum_{\phi\in\mc{Y}}\mb{E}\mbm{1}_\phi = \big((1\pm2\Delta p^{1/5})p^b\big)^{|H|-b-1}|\mc{Y}|\]
because of the earlier estimates on the size of $\mr{Gr}_q(n,s)[L[B]]$ and because $T$ is composed of distinct colors not in $\psi(\mr{Gr}(S,r))$, because $\psi$ is injective, and because the definition of $\mc{Y}$ already guarantees the first two bullet points in the definition of $Y_T$.

Now to compute $\mb{E}Y_T^2$, note that $Y_T^2$ counts pairs $(\phi_1,\phi_2)\in\mc{Y}^2$ satisfying the third bullet point above. Let us first consider such pairs where additionally $\{v_1',\ldots,v_s'\}$ and $\{\phi_1(v_{s+1}),\ldots,\phi_1(v_d)\}$ and $\{\phi_2(v_{s+1}),\ldots,\phi_2(v_d)\}$ together have some linear dependency, call this set $\mc{Z}$. By definition of $\mc{Y}$, this implies we have some linear combination where a vector from the second set and a vector from the third set both have a nonzero coefficient. This easily implies there are $|\mc{Z}|\le(q^n)^{2(d-s)-1/2}$ such terms, each of which is counted by $Y_T^2$ with probability at most $1$. For the other possibilities of $(\phi_1,\phi_2)\in\mc{Y}^2\setminus\mc{Z}$ we find
\[\mb{E}[\mbm{1}_{\phi_1}\mbm{1}_{\phi_2}]\le\big((1+q^{-n/25})p^{2b}\big)^{|H|-b-1}.\]
This is computed by using the fact that for $i\in[t]$ the images of $A_{u_i}^{-1}\phi_1(S_i)$ and $A_{u_i}^{-1}\phi_2(S_i)$ are jointly uniform over pairs of subspaces whose intersection is $r_i$-dimensional for some $0\le r_i\le r-1$ defined via $r_i = \dim(S_i\cap S_0)$, using the definition of $\mc{Z}$. Then typicality of $L$ coming from \cref{lem:J-boundedness} gives the desired upper bound on the counts of such configurations in $L[B]\subseteq L$.

Finally this means
\[\mb{E}Y_T^2\le\big((1+q^{-n/25})p^{2b}\big)^{|H|-b-1}|\mc{Y}|^2+q^{(2(d-s)-1/2)n}\]
and we deduce for $\xi$ small and $n$ sufficiently large with respect to our parameters that
\[\mr{Var}[Y_T]\le p^{1/6}\big(p^b\big)^{2(|H|-b-1)}|\mc{Y}|^2.\]

Now if $Y_T > 0$ then some $\phi\in\mc{Y}$ with $\mbm{1}_\phi=1$ exists. Considering $\phi(H)$, we see that iterative use of \cref{prop:subspace-exchange} as described earlier provides a representation of $\partial_{s,r}e_S$ as a signed sum of $\partial_{s,r}e_{S'}$ for $S'\in\phi(H)\setminus\{S\}$. Furthermore, the second and third bullets imply that each $S'$ satisfies $\mr{Gr}(S',r)\subseteq A_i\cdot L[B]$ for some $i$, providing a desired representation and showing that $\mc{E}_{S,\psi}$ fails. Thus $\mc{E}_{S,\psi}\subseteq\{Y_T = 0\}$.

We estimate by Chebyshev's inequality that
\[\mb{P}[Y_T=0]\le p^{1/7}\]
as long as $\xi$ is sufficiently small. Unfortunately, this probability is not small enough to take a union bound over $S,\psi$. However, if we take $u_{\ref{lem:rainbow-decomposition}}(\xi,q,s)$ sufficiently large then we can find $T_1,\ldots,T_{u'}$ with $u'=\lceil 8s/\xi\rceil$ which are disjoint, and then we see by (conditional) independence that
\[\mb{P}[\mc{E}_{S,\psi}]\le\prod_{i=1}^{u'}\mb{P}[Y_{T_i}=0]\le p^{u'/7}\le q^{-8sn/7}.\]
We have a union bound of size at most $q^{sn}\cdot u^{|H|}$, so for $n$ large we have that whp, none of the $\mc{E}_{S,\psi}$ hold. This directly implies the desired conclusion.
\end{proof}

Now we prove \cref{prop:color-decomposition}. The strategy is somewhat similar to the proof of \cref{lem:rainbow-decomposition}, considering the probability a random extension to a configuration as in \cref{prop:subspace-exchange} gives something with good color properties. In \cref{lem:rainbow-decomposition} we were careful to ensure that each color in some sense only appears once in the random variable $Y_T$, and the most difficult property we needed was that $r$-spaces of $L^\ast$ have a predictable number of extensions to $s$-spaces of $\mr{Gr}_q(n,s)[L[B]]$. Here, the argument is substantially simpler since \cref{lem:rainbow-decomposition} tells us it suffices to find a representation where each new clique is rainbow, rather than monochromatic.

\begin{proof}[Proof of \cref{prop:color-decomposition}]
We will only prove the second bullet point; the first is similar and simpler. Let $H$ be the $s$-dimensional $q$-system on $\mb{F}_q^k$ for $k = k_{\ref{prop:subspace-exchange}}(s)$ which is $\Ups\cup\Ups'$ from \cref{prop:subspace-exchange}. As in the proof of \cref{lem:rainbow-decomposition} let $S_0$ be a specific $s$-space of $H$. Let $R_0\leqslant S_0$ be a specific $r$-subspace. Let $S_1$ be the unique $s$-space of $H$ other than $S_0$ containing $R_0$. Label the $r$-spaces in $(\bigcup_{S\in H}\mr{Gr}(S,r))\setminus\mr{Gr}(S_1,r)$ as $R_1,\ldots,R_t$, where here clearly $t\le\qbinom{k}{r}_q$. Finally, fix a basis $\{v_1,\ldots,v_s\}$ for $S_1$ (within $\mb{F}_q^k$), suppose $R_0 = \mr{span}_{\mb{F}_q}\{v_{s-r+1},\ldots,v_s\}$, and then extend this to a basis $\{v_{s-r+1},\ldots,v_{2s-r}\}$ for $S_0$. Extend this to a basis $\{v_1,\ldots,v_d\}$ of $\mr{span}_{\mb{F}_q}(\bigcup_{S'\in H}S')$.

Let us again consider any outcomes of $L,\mc{S},B,L^\ast$ satisfying the conditions of \cref{prop:greedy-sparse-basis,lem:regularized-linear} for the given parameters as well as \cref{lem:J-boundedness} for $J=\mbf{1}$, and now only consider the randomness of $A_1,\ldots,A_u$. For $S,S^\ast\in\mr{Gr}_q(n,s)$ with $\dim(S\cap S^\ast) = r$ define $\mc{E}_{S,S^\ast,\psi}$ for $\psi\colon(\mr{Gr}(S,r)\cup\mr{Gr}(S^\ast,r))\setminus\{S\cap S^\ast\}\to[u]$ to be the event that (a) $R\in A_{\psi(R)}\cdot L^\ast$ for all $R$ in the domain of $\psi$ and (b) such a desired representation for the second bullet point does not exist. Again it suffices to give a sufficiently strong bound on $\mb{P}[\mc{E}_{S,S^\ast,\psi}]$ so that we can take a union bound.

Again let us further condition on $A_{\psi(R)}^{-1}v$ for all choices of $v\in R\in(\mr{Gr}(S,r)\cup\mr{Gr}(S^\ast,r))\setminus\{S\cap S^\ast\}$. In order to contribute to $\mc{E}_{S,S^\ast,\psi}$ we may again assume that $R\in A_{\psi(R)}\cdot L^\ast$ for all $R\in\mr{dom}(\psi)$. From now on $A_i$ for $i\notin\mr{range}(\psi)$ are uniformly random elements of $\mr{GL}(\mb{F}_q^n)$ while $A_{\psi(R)}^{-1}$ for $R\in\mr{dom}(\psi)$ are uniformly random in $\mr{GL}(\mb{F}_q^n)$ conditional on knowing the images of all $v\in R$.

Since $\dim(S\cap S^\ast) = r$, we can fix an $\mb{F}_q$-basis $\{v_1',\ldots,v_{2s-r}'\}$ with the property that $\{v_1',\ldots,v_s'\}$ is a basis for $S$ and $\{v_{s-r+1}',\ldots,v_{2s-r}'\}$ is a basis for $S^\ast$ (within $\mb{F}_q^n$). Consider some sequences $T = (u_1,\ldots,u_t)\in[u]^t$ and $T^\ast = (u_1^\ast,\ldots,u_t^\ast)\in[u]^t$ of mutually distinct colors that are not in $\mr{range}(\psi)$, which exists as long as $u$ is large enough. Finally consider $Y_{T,T^\ast}$, the number of pairs of $q$-embeddings $\phi,\phi^\ast\colon\mb{F}_q^k\to\mb{F}_q^n$ of $H$ in $\mr{Gr}_q(n,s)$ such that:
\begin{itemize}
    \item $\phi(S_1) = S$ and $\phi^\ast(S_1) = S^\ast$ and in fact $\phi(v_i) = v_i'$ and $\phi^\ast(v_i) = v_{2s-r+1-i}'$ for $i\in[s]$. In particular $\phi(R_0) = \mr{span}_{\mb{F}_q}\{v_{s-r+1}',\ldots,v_s'\} = \phi^\ast(R_0)$.
    \item $\phi(S_0) = \phi^\ast(S_0)$.
    \item $\{v_1',\ldots,v_{2s-r}'\}$ and $\{\phi(v_{s+1}),\ldots,\phi(v_{2s-r})\}$ and $\{\phi(v_i)\colon 2s-r<i\le d\}$ and $\{\phi^\ast(v_i)\colon 2s-r<i\le d\}$ are jointly $\mb{F}_q$-linearly independent.
    \item For each $i\in[t]$ we have $\phi(R_i)\in\mr{Gr}_q(n,s)[A_{u_i}\cdot L^\ast]$ and $\phi^\ast(R_i)\in\mr{Gr}_q(n,s)[A_{u_i^\ast}\cdot L^\ast]$.
\end{itemize}
This can be thought of as gluing ``linearly disjoint'' two copies of $H$ along $S_0$ and then making their respective images of $S_1$ map to $S$ and $S^\ast$ (which share an $r$-subspace). If $Y_{T,T^\ast} > 0$ then, similar to the proof of \cref{lem:rainbow-decomposition}, we can use \cref{prop:subspace-exchange} to rewrite $\partial_{s,r}(e_S-e_{S^\ast})$ as a signed sum of $\partial_{s,r}e_{S'}$ for $S'\in\phi(H\setminus\{S_0,S_1\})$ and $S'\in\phi^\ast(H\setminus\{S_0,S_1\})$, and each such $S'$ satisfies the conditions of \cref{lem:rainbow-decomposition}: every $S'\in H\setminus\{S_0,S_1\}$ has images under $\phi,\phi^\ast$ which are ``rainbow'' in the correct way, and $\phi(S_0) = \phi^\ast(S_0)$ (which are the only $s$-spaces other than $S,S^\ast$ in this configuration that contain the $r$-space $S\cap S^\ast$ for which we have no guarantees) cancel each other out in this representation. Applying \cref{lem:rainbow-decomposition} would therefore show that $S,S^\ast$ do have a valid representation. That is, if we let $\mc{E}_{\ref{lem:rainbow-decomposition}}$ be the event that the conclusion of \cref{lem:rainbow-decomposition} holds for these parameters, then
\[\mc{E}_{S,S^\ast,\psi}\cap\mc{E}_{\ref{lem:rainbow-decomposition}}\subseteq\{Y_{T,T^\ast} = 0\}.\]

Now we perform an analogous second-moment computation for $Y_{T,T^\ast}$. Defining $\mc{Y}$ as the number of pairs of maps $\phi,\phi^\ast\colon\mb{F}_q^k\to\mb{F}_q^n$ satisfying the first three bullets above, we have $|\mc{Y}| = \Theta_{q,s}(q^{(2d-3s+r)n})$. Thus from \cref{lem:J-boundedness,lem:regularized-linear} we find
\begin{align*}
\mb{E}Y_{T,T^\ast}&=((1\pm\Delta p^{1/5})p)^{2t}|\mc{Y}|,\\
\mb{E}Y_{T,T^\ast}^2&\le((1+q^{-n/4})p)^{2t}|\mc{Y}|^2+q^{(2(2d-3s+r)-1/2)n}
\end{align*}
in a manner similar to the proof of \cref{lem:rainbow-decomposition}. We similarly deduce $\mb{P}[Y_{T,T^\ast}=0]\le p^{1/7}$ by Chebyshev's inequality, and now if $u$ is sufficiently large in terms of $\xi$ then taking $u'=\lceil 16s/\xi\rceil$ and appropriate disjoint $T_i,T_i^\ast$ for $i\in[u']$ yields
\[\mb{P}[\mc{E}_{S,S^\ast,\psi}\cap\mc{E}_{\ref{lem:rainbow-decomposition}}]\le\prod_{i=1}^{u'}\mb{P}[Y_{T_i,T_i^\ast}=0]\le p^{u'/7}\le q^{-16sn/7}\]
Taking a union bound of size at most $q^{2sn}\cdot u^{2\qbinom{s}{r}_q}$ and combining with \cref{lem:rainbow-decomposition} finishes.
\end{proof}

Next we wish to show that an arbitrary ``flat'' $J$ can be split up in some sense as the image of a bounded element plus something whose support is purely within $\bigcup_{i=1}^u(A_i\cdot L^\ast)$. The proof goes by covering the parts outside in a spread-out manner using \cref{lem:J-boundedness}. We will first need the following estimate for extending $r$-spaces into some $A_i\cdot L^\ast$.
\begin{lemma}\label{lem:color-extension-count}
Let $0 < \xi\le\xi_{\ref{lem:color-extension-count}}(q,s)$, fix some positive integers $u\ge u_{\ref{lem:color-extension-count}}(\xi,q,s)$ and $\Delta\ge 1$, and suppose $n$ is large with respect to these parameters. Let $p = q^{-\xi n}$, and then let $L,\mc{S},B$ be as in \cref{prop:greedy-sparse-basis} and $L^\ast$ be as in \cref{lem:regularized-linear} (all existing whp). Sample $A_1,\ldots,A_u\in\mr{GL}(\mb{F}_q^n)$ uniformly at random. Whp, for every $R\in\mr{Gr}_q(n,r)$ there is an index $i_R\in[u]$ so that the number of $s$-spaces $S\geqslant R$ with $\mr{Gr}(S,r)\setminus\{R\}\subseteq A_{i_R}\cdot L^\ast$ is $(1\pm p^{1/6})p^{\qbinom{s}{r}_q-1}\qbinom{n-r}{s-r}_q$.
\end{lemma}
\begin{proof}
By the first part of \cref{lem:J-boundedness} whp for all $R\in\mr{Gr}_q(n,r)$ and all $i\in[u]$ the number of $s$-spaces $S\geqslant R$ with $\mr{Gr}(S,r)\setminus\{R\}\subseteq A_i\cdot L$ is $(1\pm q^{-n/20})p^{\qbinom{s}{r}_q-1}\qbinom{n-r}{s-r}_q$. We also deduce that for every $i\in[u]$ and distinct $R_1,R_2\in\mr{Gr}_q(n,r)$ with $d=\dim\mr{span}_{\mb{F}_q}\{R_1,R_2\}\le s$ that there are $(1\pm q^{-n/20})p^{\qbinom{s}{r}_q-2}\qbinom{n-d}{s-d}_q$ extensions to an $s$-space $S$ containing $R_1,R_2$ with $\mr{Gr}(S,r)\setminus\{R_1,R_2\}\subseteq A_i\cdot L$.

Now for fixed $i$ and for some $R\in\mr{Gr}_q(n,r)$ let us consider $Y_{R,i}$, the number of $s$-spaces $S\geqslant R$ with $\mr{Gr}(S,r)\setminus\{R\}\subseteq A_i\cdot L$ and such that at least one element of $\mr{Gr}(S,r)$ is in $A_i\cdot(L\setminus L^\ast)$. We consider only the randomness of $A_i$ and see that $\mb{E}Y_{R,i}$ is independent of $R$, and satisfies the bound
\[\mb{E}Y_{R,i}\le\qbinom{s}{r}_q\frac{|L\setminus L^\ast|}{\qbinom{n}{r}_q}\cdot(1+q^{-n/20})p^{\qbinom{s}{r}_q-2}\qbinom{n-r}{s-r}_q\lesssim_{q,s}(\Delta p^{1/5}+q^{-n/3})\cdot p^{\qbinom{s}{r}_q-1}\qbinom{n-r}{s-r}_q,\]
by using \cref{lem:regularized-linear} to bound $|L\setminus L^\ast|$ and using the above fact about extension counts within $L$ (and summing appropriately). By Markov's inequality, we have $Y_{R,i}\ge p^{2/11}p^{\qbinom{s}{r}_q-1}\qbinom{n-r}{s-r}_q$ with probability at most $p^{1/64}$ as long as $n$ is large. Since the various values of $i$ are independent, we see that $Y_{R,i}\le p^{2/11}p^{\qbinom{s}{r}_q-1}\qbinom{n-r}{s-r}_q$ for some $i\in[u]$ with probability at least $1-p^{u/64}$. For $u$ large in terms of $\xi$, we can take a union bound over all values of $R$ and find that whp, every $R$ has some index $i_R\in[u]$ where $Y_{R,i_R}\le p^{2/11}p^{\qbinom{s}{r}_q-1}\qbinom{n-r}{s-r}_q$.

Finally, combining with the observations in the first paragraph we see that for every $R$ the number of $s$-spaces $S\geqslant R$ with $\mr{Gr}(S,r)\setminus\{R\}\subseteq A_{i_R}\cdot L^\ast$ is at most $(1+q^{-n/3})p^{\qbinom{s}{r}_q-1}\qbinom{n-r}{s-r}_q$ and at least
\[(1-q^{-n/20}-p^{2/11})p^{\qbinom{s}{r}_q-1}\qbinom{n-r}{s-r}_q\ge(1-p^{1/6})p^{\qbinom{s}{r}_q-1}\qbinom{n-r}{s-r}_q.\qedhere\]
\end{proof}

Now we show that we can cover into $\bigcup_{i=1}^u(A_i\cdot L^\ast)$ in a bounded manner.

\begin{lemma}\label{lem:cover-into-random}
Let $J\in\mb{Z}^{\mr{Gr}_q(n,r)}$ be $\theta$-bounded for some $\theta\ge q^{-n/20}$ with $|J_R|\le q^{0.1n}$ for all $R\in\mr{Gr}_q(n,r)$. Let $0 < \xi\le\xi_{\ref{lem:cover-into-random}}(q,s)$, fix some positive integers $u\ge u_{\ref{lem:cover-into-random}}(\xi,q,s)$ and $\Delta\ge 1$, let $C = C_{\ref{lem:cover-into-random}}(u,\xi,q,s)$, and suppose $n$ is large with respect to these parameters. Let $p = q^{-\xi n}$, and then let $L,\mc{S},B$ be as in \cref{prop:greedy-sparse-basis} and $L^\ast$ be as in \cref{lem:regularized-linear} (all existing whp). Sample $A_1,\ldots,A_u\in\mr{GL}(\mb{F}_q^n)$ uniformly at random. Whp, there is $\Phi\in\mb{Z}^{\mr{Gr}_q(n,s)}$ with $\partial_{s,r}\Phi = J-J'$ such that $J',\partial_{s,r}\Phi^\pm$ are $C\theta$-bounded and $\mr{supp}(J')\subseteq\bigcup_{i=1}^u(A_i\cdot L^\ast)$.
\end{lemma}
\begin{proof}
Whp $L$ satisfies \cref{lem:J-boundedness} for $J$, which we now assume. Next, for each $R\in\mr{Gr}_q(n,r)$, let $i_R$ be as in \cref{lem:color-extension-count} (existing whp). For each signed element $R$ of $J$, counted with multiplicity, we consider a uniformly random extension $S^R\in\mr{Gr}_q(n,s)$ with the property that $\mr{Gr}(S^R,r)\setminus\{R\}\subseteq A_{i_R}\cdot L^\ast$. Let $\Phi$ be the sum of said elements with the corresponding signs, and write $J' = J-\partial\Phi$. Note that there are $(1\pm p^{1/6})p^{\qbinom{s}{r}_q-1}\qbinom{n-r}{s-r}_q$ choices of each $S^R$ by \cref{lem:color-extension-count} and that clearly $\mr{supp}(J')\subseteq\bigcup_{i=1}^u(A_i\cdot L^\ast)$. It suffices to show the necessary boundedness.

For any $R'\in\bigcup_{i=1}^u(A_i\cdot L^\ast)$ we consider the expected number of $S^R$ that it is contained in. If we consider the contribution from $R$ with $\dim(R\cap R') = k\in\{\max(0,2r-s),\ldots,r-1\}$ and $i_R = i$, there are at most $2\theta p^{\qbinom{2r-k}{r}_q-2}q^{(r-k)n}$ choices of $R$ (counting multiplicity) by \cref{lem:J-boundedness}, each of which has at least $(1/2)p^{\qbinom{s}{r}_q-1}\qbinom{n-r}{s-r}_q$ choices of $S^R$. On the other hand, for given $R$ we know $R'$ is contained in at most $2p^{\qbinom{s}{r}_q-\qbinom{2r-k}{r}_q}q^{(s-(2r-k))n}$ many $S^R$ by \cref{lem:J-boundedness}. Thus the expected number of $S^R$ that $R'$ is contained in is bounded by
\[u\cdot\sum_{k=\max(0,2r-s)}^{r-1}2\theta p^{\qbinom{2r-k}{r}_q-2}q^{(r-k)n}\cdot 2p^{\qbinom{s}{r}_q-\qbinom{2r-k}{r}_q}q^{(s-(2r-k))n}\cdot\frac{1}{(1/2)p^{\qbinom{s}{r}_q-1}\qbinom{n-r}{s-r}_q}\le O_{u,q,s}(\theta p^{-1}).\]

Finally, for any $Q\in\mr{Gr}_q(n,r-1)$ by \cref{lem:J-boundedness} it is in at most $2upq^n$ many $r$-spaces $R'\in L$, so say $(\partial\Phi^+)_Q$ has mean bounded by $O_{u,q,s}(\theta q^n)$. It is therefore easy to see that we can apply Chernoff (to each value of $k$ separately, similar to the proof of \cref{prop:robust-local-decodability}) to show that $\partial\Phi^\pm$ are $O_{u,q,s}(\theta)$-bounded.
\end{proof}

Next, we show that something in the lattice whose support is within $\{S\in\mr{Gr}_q(n,s)\colon\mr{Gr}(S,r)\subseteq\bigcup_{i=1}^u(A_i\cdot L[B])\}$ can be generated purely by $\partial_{s,r}e_S$ for $S$ a monochromatic $s$-space. This is done by taking an arbitrary representation as $\partial_{s,r}\Phi$ and using a ``subspace exchange'' process to reduce to $s$-spaces for which we can apply \cref{lem:rainbow-decomposition}.

\begin{proposition}\label{prop:sparse-into-gadgets}
Let $0 < \xi\le\xi_{\ref{prop:sparse-into-gadgets}}(q,s)$, fix some positive integers $u\ge u_{\ref{prop:sparse-into-gadgets}}(\xi,q,s)$ and $\Delta\ge 1$, and suppose $n$ is large with respect to these parameters. Let $p = q^{-\xi n}$, and then let $L,\mc{S},B$ be as in \cref{prop:greedy-sparse-basis} and $L^\ast$ be as in \cref{lem:regularized-linear} (all existing whp). Sample $A_1,\ldots,A_u\in\mr{GL}(\mb{F}_q^n)$ uniformly at random. Whp, for any $J\in\partial_{s,r}\mb{Z}^{\mr{Gr}_q(n,s)}$ with $\mr{supp}(J)\subseteq\bigcup_{i=1}^u(A_i\cdot L^\ast)$ there is a representation
\[J' = \sum_{S\in\mc{I}}a_S\partial_{s,r}e_S\]
such that for all $S\in\mc{I}$, $a_S\in\mb{Z}$ and there is $i\in[u]$ such that $\mr{Gr}(S,r)\subseteq A_i\cdot L[B]$.
\end{proposition}
\begin{proof}
Assume \cref{prop:color-decomposition,lem:color-extension-count} hold (whp). Write
\[J = \sum_{S\in\mc{I}_0}a_S\partial_{s,r}e_S\]
for $\mc{I}_0\subseteq\mr{Gr}_q(n,s)$ and $a_S\in\mb{Z}$. For each $S\in\mc{I}_0$, we can find a representation of $\partial_{s,r}e_S$ as a signed sum of $\partial_{s,r}e_{S'}$ where each $S'$ is such that all of $\mr{Gr}(S',r)$ except perhaps one $r$-space is in $\bigcup_{i=1}^u(A_i\cdot L^\ast)$. This can be proved similar to \cref{prop:color-decomposition,lem:rainbow-decomposition}: we use a subspace exchange from \cref{prop:subspace-exchange} to flip out into a bunch of ``rainbow'' $s$-spaces. The $s$-spaces which share an $r$-space with $S$ are the only ones that may have an exceptional $r$-space. We forgo the details of the proof as it is essentially the same. Explicitly, we have
\[J = \sum_{S\in\mc{I}_1}b_S\partial_{s,r}e_S+\sum_{S\in\mc{I}_2}b_S\partial_{s,r}e_S\]
for $\mc{I}_1\cup\mc{I}_2\subseteq\mr{Gr}_q(n,s)$ with the property that $b_S\in\mb{Z}$ and for all $S\in\mc{I}_1$, $\mr{Gr}(S,r)\subseteq\bigcup_{i=1}^u(A_i\cdot L^\ast)$ and for all $S\in\mc{I}_2$, there is $R_S\in\mr{Gr}(S,r)$ such that $\mr{Gr}(S,r)\setminus\{R_S\}\subseteq\bigcup_{i=1}^u(A_i\cdot L^\ast)$ and $R_S\notin\bigcup_{i=1}^u(A_i\cdot L^\ast)$.

Now we can write $J$ as an integer linear combination of vectors of the form $\partial_{s,r}e_S$ for $\mr{Gr}(S,r)\subseteq\bigcup_{i=1}^u(A_i\cdot L^\ast)$ and $\partial_{s,r}(e_S-e_{S'})$ where there is some $r$-space $R\subseteq S\cap S'$ such that $(\mr{Gr}(S,r)\cup\mr{Gr}(S',r))\setminus\{R\}\subseteq\bigcup_{i=1}^r(A_i\cdot L^\ast)$. Indeed, just note that for each $r$-space $R\notin\bigcup_{i=1}^r(A_i\cdot L^\ast)$ the $s$-spaces containing $R$ must have weights summing to $0$ by the support condition, and two such $R,R'$ cannot have the same $s$-space containing both by the above condition on $\mc{I}_2$.

By the first bullet of \cref{prop:color-decomposition}, the vectors $\partial_{s,r}e_S$ of the first type can be reduced to a sum of the desired form. The vectors $\partial_{s,r}(e_S-e_{S'})$ almost can via the second bullet, except that we do not necessarily know that $\dim(S\cap S') = r$. However, given such a term with corresponding $r$-space $R$ note that by \cref{lem:color-extension-count} there is $i_R\in[u]$ so that the number of $s$-spaces $S^\ast\geqslant R$ with $\mr{Gr}(S^\ast,r)\setminus\{R\}\subseteq A_{i_R}\cdot L^\ast$ is at least $(1/2)p^{\qbinom{s}{r}_q-1}\qbinom{n-r}{s-r}_q$. This is greater than $q^{(s-r-1/2)n}$ for $\xi$ small and $n$ sufficiently large. On the other hand, the number of choices of $S^\ast$ where $\dim(S\cap S^\ast) > r$ or $\dim(S^\ast\cap S') > r$ is at most $O_{q,s}(q^{(s-r-1)n})$. Therefore there is a choice of $S^\ast$ so that, writing
\[\partial_{s,r}(e_S-e_{S'})=\partial_{s,r}(e_S-e_{S^\ast})+\partial_{s,r}(e_{S^\ast}-e_{S'}),\]
we can apply \cref{prop:color-decomposition} to both terms. The result follows.
\end{proof}

Finally, we are ready to prove \cref{thm:bounded-inverse}.

\begin{proof}[Proof of \cref{thm:bounded-inverse}]
Let $\Delta = \Delta_{\ref{prop:robust-local-decodability}}(q,s,r)$. Choose $\xi$ very small in terms of $q,s$ and then $u$ very large so that, taking $p = q^{-\xi n}$, we can apply all of the prior lemmas for these parameters. We will choose $c_{\ref{thm:bounded-inverse}},C_{\ref{thm:bounded-inverse}}^{-1}$ very small with respect to these at the end to satisfy various inequalities. In particular, let us now let $L,\mc{S},B$ be as in \cref{prop:greedy-sparse-basis} and $L^\ast$ as in \cref{lem:regularized-linear}, existing whp and suppose \cref{prop:sparse-into-gadgets} hold, which can be done whp.

We are given $J$ which is $\theta$-bounded such that $J\in\partial_{s,r}\mb{Z}^{\mr{Gr}_q(n,s)}$. By \cref{lem:J-flattening} we can write $J = \partial_{s,r}\Phi_1+J^{(1)}$ with $J^{(1)},\partial_{s,r}\Phi_1^\pm$ being $O_{q,s}(\theta)$-bounded and $|J_R^{(1)}|\le q^{0.1n}$ for all $R\in\mr{Gr}_q(n,r)$ (in a way that does not depend on our conditioned randomness). Now by \cref{lem:cover-into-random}, whp we can write $J^{(1)} = \partial_{s,r}\Phi_2+J^{(2)}$ with $J^{(2)},\partial_{s,r}\Phi_2^\pm$ being $O_{q,s}(\theta)$-bounded and such that $\mr{supp}(J^{(2)})\subseteq\bigcup_{i=1}^u(A_i\cdot L^\ast)$. By \cref{prop:sparse-into-gadgets}, since $J^{(2)}$ is in the lattice, there is a representation
\[J^{(2)} = \sum_{S\in\mc{I}}a_S\partial_{s,r}e_S\]
where for all $S\in\mc{I}$, $a_S\in\mb{Z}$ and $\mr{Gr}(S,r)\subseteq A_i\cdot L[B]$ for some $i$.

Now for each $S\in\mc{I}$, consider $i$ so that $\mr{Gr}(S,r)\subseteq A_i\cdot L[B]$. By the conclusion of \cref{prop:greedy-sparse-basis} and applying the invertible linear map $A_i$, we have $\partial_{s,r}e_S\in\partial_{s,r}(\mb{Z}/\Delta\mb{Z})^{A_i\cdot\mc{S}}$ since $S\in\mr{Gr}_q(n,s)[A_i\cdot L[B]]$. Thus the above implies that
\[J^{(2)}\in\partial_{s,r}(\mb{Z}/\Delta\mb{Z})^{\bigcup_{i=1}^u(A_i\cdot\mc{S})},\]
so we can write $J^{(2)} = \partial_{s,r}\Phi_3+J^{(3)}$ where $J^{(3)}\in\Delta\mb{Z}^{\mr{Gr}_q(n,r)}$ and $\Phi_3$ is a nonnegative sum of $e_S$ for $S\in\bigcup_{i=1}^u(A_i\cdot\mc{S})$ with coefficients in $\{0,\ldots,\Delta-1\}$. (We are treating $\Phi_3$ as an integral vector, not $\imod{\Delta}$.) By \cref{prop:greedy-sparse-basis}, $\partial_{s,r}\Phi_3^\pm$ are $u\Delta p^{1/2}$-bounded hence $\theta$-bounded as long as $c_{\ref{thm:bounded-inverse}}$ is much smaller than $\xi$. Thus $J^{(3)}$ is $O_{q,s}(\theta)$-bounded.

Finally, by \cref{prop:robust-local-decodability} we can write $J^{(3)} = \partial_{s,r}\Phi_4$ with $\partial_{s,r}\Phi_4^\pm$ being $O_{q,s}(\theta)$-bounded. Now $J = \partial_{s,r}(\Phi_1+\Phi_2+\Phi_3+\Phi_4)$ satisfies the desired by the triangle inequality.
\end{proof}

\section{Template}\label{sec:template}
We define the template via an algebraic construction similar to that appearing in the proof of \cref{prop:subspace-exchange}. Recall from \cref{sub:setup} that $V = \mb{F}_q^n, K=\mb{F}_{q^{\ell m}}\hookleftarrow\mb{F}_{q^\ell}=L$, and $\iota_1,\ldots,\iota_z\colon K\to V$ are uniformly random embeddings of our field into the vector space $V$. Recall $K_i=\iota_i(K)$ and $L_i=\iota_i(L)$. Our template will consist of (the $r$-subspaces of) certain $s$-spaces realized by the vectors $\iota_i(Nx)\in K_i^s$ for some fixed $N\in L^{s\times r}$ (applying $\iota_i$ element-wise in the obvious way) and where $x\in K^r$ is a varying parameter. However, as discussed in \cref{sub:outline-template}, in order to use each $r$-space at most once while performing this we must in fact only take a dense subset, and also guarantee that there are no overlaps between the different values $i\in[z]$. The most natural way to do this is to specify for each $r$-space a ``configuration'' defining how it is allowed to be used in the template as well as which template it can be used in, and only including those $s$-spaces coming from some $\iota_i(Nx)$ all of whose $r$-subspaces agree with the given configuration. After this, we further subsample these $s$-spaces for later use.

To define the template, we first define the possible space of configurations.
\begin{definition}\label{def:rref}
Let $\mr{Red}_q^{r\times s}$ be the set of matrices $\Pi\in\mb{F}_q^{r\times s}$ of rank $r$ that are in reduced row echelon form. That is, there are $1\le j_1<\cdots<j_r\le s$ so that $\Pi_{i,j_i}=1$ and $\Pi_{i,j}=0$ for $j<j_i$ and $\Pi_{k,j_i}=0$ for $k<i$.
\end{definition}
By basic facts about Gaussian elimination, we see that for every full rank $\Pi\in\mb{F}_q^{r\times s}$ there is a unique $\Pi'\in\mr{GL}(\mb{F}_q^r)$ so that $\Pi'\Pi\in\mr{Red}_q^{r\times s}$, and hence we can compute $|\mr{Red}_q^{r\times s}|=\qbinom{s}{r}_q$. In particular, given an $s$-space $S\leqslant V=\mb{F}_q^n$ with basis given by the elements of $b\in V^r$, the distinct $r$-subspaces are given by $\mr{span}_{\mb{F}_q}(\Pi b)$ for $\Pi\in\mr{Red}_q^{r\times s}$. (For notational convenience, given a vector $b$ composed of $r$ or $s$ elements of $\mb{F}_q^n$ or a similar vector space, we will abuse notation and write $\mr{span}_{\mb{F}_q}(b)$ for the result of taking the $\mb{F}_q$-span of the elements of $b$, and we will allow ourselves to apply linear maps coordinate-wise on the elements.)

\begin{definition}\label{def:template}
Given the setup in \cref{sub:setup}, choose some $\tau\in(0,1]$. Recall $N=N_\mr{abs}\in L^{s\times r}$ is an $\mb{F}_q$-generic matrix of degree $d$. For each $r$-space $R$ within $V$, sample independent $y_R\sim\mr{Ber}(\tau)$ and $i_R\sim\mr{Unif}([z])$. For each $r$-space $R$ in $V$, let $b_R\in R^r$ be a uniformly random basis of $R$ arranged as a column vector and let $\Pi_R$ be a uniformly random element of $\mr{Red}_q^{r\times s}$. The \emph{$s$-template} $\mc{S}_\mr{tem}$ is the set of $s$-spaces in $\mb{F}_q^n$ that are of the form $\mr{span}_{\mb{F}_q}(\iota_i(Nx))$ with $i\in[z]$ and $x\in K^r$ such that $\dim_L\mr{span}_L(x)=r$ and:
\begin{itemize}
    \item For each $r$-space $R$ within $\mr{span}_{\mb{F}_q}(\iota_i(Nx))$, we have $i_R = i$.
    \item For each $\Pi\in\mr{Red}_q^{r\times s}$, we have for $R = \mr{span}_{\mb{F}_q}(\iota_i(\Pi Nx))$ that $\iota_i(\Pi_RNx) = b_R$;
    \item For each $r$-space $R$ within $\mr{span}_{\mb{F}_q}(\iota_i(Nx))$, we have $y_R = 1$.
\end{itemize}
The template, $G_\mr{tem}\subseteq G$, is the $r$-dimensional multi-$q$-system formed by taking the multiset of $r$-subspaces of $s$-spaces in the $s$-template. We let $G_{\mr{tem},i}$ be the portion of $G_\mr{tem}$ arising from $\iota_i$ for index $i\in[z]$ and define $\mc{S}_{\mr{tem},i}$ similarly. We will often think of the indices $i\in[z]$ as ``colors''.
\end{definition}
\begin{remark}
Since the coordinates of $x\in K^r$ are $L$-linearly independent, $\mr{span}_{\mb{F}_q}(\iota_i(Nx))$ will always form an $s$-space over $\mb{F}_q$ (otherwise the coefficients of $N$ satisfy a nontrivial $\mb{F}_q$-linear relation, which is of degree $1\le d$).
\end{remark}

We record some basic facts about the template, in particular that it is well-defined and $G_\mr{tem}$ is actually a $q$-system and not a multi-$q$-system.
\begin{lemma}\label{lem:template-basic}
Given \cref{sub:setup,def:template}, we have the following as long as $c=c_{\ref{lem:template-basic}}(q,s) > 0$, $q^{-cn}\le\tau\le c$, $d\ge d_{\ref{lem:template-basic}}(r)$, $\ell\ge\ell_{\ref{lem:template-basic}}(d,s)$, and $n$ is large.
\begin{itemize}
    \item The template is well-defined.
    \item Every $r$-space $R\leqslant V$ appears at most once in an $s$-space in $\mc{S}_\mr{tem}$;
    \item For $t\in\{s,r+s\}$ call a $t$-space \emph{obstructed} by the template if any of its $r$-subspaces is contained in $G_\mr{tem}$. Whp, for each $r$-space $R\notin G_\mr{tem}$ at most a $\tau^{1/2}$-fraction of $\{T\in\mr{Gr}(V,t)\colon R\leqslant T\}$ is obstructed by the template.
    \item Every $R\in G_{\mr{tem},i}$ satisfies $\dim_L\mr{span}_L(\iota_i^{-1}(R))=r$.
\end{itemize}
\end{lemma}
\begin{proof}
First, the template is well-defined as long as we can choose $N$ with the desired genericity property and as long as $\mr{span}_{\mb{F}_q}(\iota_i(\Pi Nx))$ is an $r$-space always for possible choices of $i,x,\Pi$. The former holds as long as $\ell\ge\ell_{\ref{lem:template-basic}}(d,s)$ as in the beginning of the proof of \cref{lem:generic}. For the remainder of the argument we will specify $d_{\ref{lem:template-basic}}(r)$ sufficiently large so that $\mb{F}_q$-genericity guarantees that for all $\Pi\in\mb{F}_q^{r\times s}$ of rank $r$, $\Pi N\in L^{r\times r}$ is invertible, which follows from \cref{lem:generic}. This guarantees the latter property, since now the $r$ elements of $\Pi Nx$ span $r$ dimensions over $\mb{F}_q$: $\dim_L\mr{span}_L(\Pi Nx)=\dim_L\mr{span}_L(x)=r$ hence $\dim\mr{span}_{\mb{F}_q}(\Pi Nx)\ge r$.

Second, suppose some $r$-space $R$ appears twice, associated to $\iota_i$ and $x$ as well as $\iota_j$ and $x'$. The first condition in \cref{def:template} ensures that $i=j$. Next note from the discussion following \cref{def:rref} that every $r$-subspace of $S=\mr{span}_{\mb{F}_q}(\iota_i(Nx))$ will show up as $\mr{span}_{\mb{F}_q}(\Pi\iota_i(Nx))=\mr{span}_{\mb{F}_q}(\iota_i(\Pi Nx))$ for some $\Pi\in\mr{Red}_q^{r\times s}$ (and similar for $x'$). Thus the second condition forces $\iota_i(\Pi_RNx)=\iota_i(\Pi_RNx') = b_R$ and hence $\Pi_RNx=\Pi_RNx'$ since $\iota_i$ is an injective map. By the condition on $N$ above, we have that $\Pi_RN$ is invertible so $x = x'$.

The third result follows from concentration. For $t\in\{s,r+s\}$ call a $t$-space \emph{trivially unobstructed} if each constituent $r$-subspace $R'$ satisfies $y_{R'} = 0$. Given $R\in\mr{Gr}(V,r)$, the expected fraction of $t$-spaces extending $R$ (of which there are $\qbinom{n-r}{t-r}_q = \Theta_{q,s}(q^{(t-r)n})$ total) that are trivially unobstructed other than possibly due to $R$ is $1-O_s(\tau)$. Furthermore there are at most $q^{(r-u)n}$ many $r$-spaces $R'$ such that $\dim(R\cap R') = u$ for $\max(0,2r-s)\le u\le r-1$ and each such $r$-space can cause at most $q^{(t-(2r-u))n}$ many $t$-spaces extending $R$ to become obstructed. The result then follows from \cref{lem:azuma} with variance proxy bounded by $O_{q,s}(\max_{u\le r-1}q^{2(t-2r+u)n}\cdot q^{(r-u)n}) = O_{q,s}(q^{(2t-2r-1)n})$ and taking a union bound: whp every $R\in\mr{Gr}(V,r)$ has at most a $\tau^{1/2}$ fraction of extending $t$-spaces obstructed by the template via an $r$-space other than $R$, which implies the desired condition for those $r$-spaces satisfying $R\notin G_\mr{tem}$.

Finally, any possible $R$ has a basis of the form $\iota_i(\Pi Nx)$ where $x\in K^r$ is $L$-linearly independent. Thus it suffices to show $\dim_L\mr{span}_L(\Pi Nx)=r$. But $\Pi N$ is invertible if $d\ge r$ by genericity, and $x\in K^r$ being $L$-linearly independent thus shows the desired.
\end{proof}

\section{Absorber analysis}\label{sec:absorber}
We first prove that the templates are robust in configurations that are useful for absorption. However, there are some natural constraints that are necessary to impose in order for an $s$-space to be absorbable into the template.

The absorber configuration we use is essentially the same as constructed in the proof of \cref{prop:subspace-exchange} with the field inclusion $X\hookrightarrow Y$ replaced by $L\hookrightarrow K$, except we specialize $x^{(1)}=0$ to ensure one of the two $s$-space decompositions is fully compatible with the definition of the template.

\begin{definition}\label{def:absorber}
Given the setup in \cref{sub:setup,def:template}, consider $u\ge 1$ and let $x^\ast\in L^{s\times u}$ be such that $N=N_\mr{tem}$ and $x^\ast$ are jointly $\mb{F}_q$-generic of degree $d$. For $w'\in K^r$ and $w\in K^u$ such that the elements of $(w',w)$ are $L$-linearly independent, define $\mc{P}_\mr{out}^{(w',w)}=\{\mr{span}_{\mb{F}_q}(Nw'+(Nx+x^\ast)w)\colon x\in L^{r\times u}\}$ and $\mc{P}_\mr{in}^{(w',w)}=\{\mr{span}_{\mb{F}_q}(Nw'+Nxw)\colon x\in L^{r\times u}\}$. An \emph{absorber with parameters} $(w',w)$ is $\mc{P}_\mr{out}^{(w',w)},\mc{P}_\mr{in}^{(w',w)}$, which we respectively call the \emph{out- and in-flips}. We say that the \emph{root} of the absorber is $\mr{span}_{\mb{F}_q}(Nw'+x^\ast w)\in\mc{P}_\mr{out}^{(w',w)}$. The absorber is \emph{valid} for template index $i\in[z]$ if $\mc{P}_\mr{in}^{(w',w)}\subseteq\iota_i^{-1}(\mc{S}_{\mr{tem},i})$.
\end{definition}

We will always count the absorber by the number of choices of parameters; the next lemma shows that this will not make a big difference (and also shows that the absorber behaves as expected, e.g., we have two families of $s$-spaces with no degeneracies that each contain the same overall collection of distinct $r$-spaces).

\begin{lemma}\label{lem:absorber-admissible}
Given the setup of \cref{def:absorber} and assuming $d\ge d_{\ref{lem:absorber-admissible}}(r)$, for $w'\in K^r$ and $w\in K^u$ whose elements are jointly $L$-linearly independent, we have the following:
\begin{itemize}
    \item $w'+xw$ is $L$-linearly independent for all $x\in L^{r\times u}$;
    \item The spaces in $\mc{P}_\mr{out}^{(w',w)}$ generated by $x\in L^{r\times u}$ are distinct and $s$-dimensional;
    \item For distinct $P,P'\in\mc{P}_\mr{out}^{(w',w)}$ we have $\dim(P\cap P') < r$, and same for $P,P'\in\mc{P}_\mr{in}^{(w',w)}$;
    \item For $P\in\mc{P}_\mr{out}^{(w',w)}$ and $P'\in\mc{P}_\mr{in}^{(w',w)}$ we have $\dim(P\cap P')\le r$;
    \item $\partial_{s,r}\mc{P}_\mr{out}^{(w',w)} = \partial_{s,r}\mc{P}_\mr{in}^{(w',w)}$.
\end{itemize}
Additionally, if $d\ge r$, given the identities of all the $s$-spaces in $\mc{P}_\mr{in}^{(w',w)}$ then there are at most $O_{u,q,s}(1)$ choices of $L$-linearly independent $(w',w)$ producing them.
\end{lemma}
\begin{proof}
For the first bullet point, it is trivial that $w'+xw\in K^r$ is $L$-linearly independent for all $x\in L^{r\times u}$: any nontrivial $L$-dependence of these coordinates would give a nontrivial $L$-dependence of $(w',w)$ since it will involve at least one coordinate of $w'$.

For the other four bullets, inspection of the proof of \cref{prop:subspace-exchange}, which has the same subspace setup, reveals that it suffices for the following to hold where we let $x^{(1)}=0$ and $x^{(2)}=x^\ast$:
\begin{itemize}
    \item For all $y\in\mb{F}_q^{1\times s}$ and $j\in\{1,2\}$, we have $y(Nw'+(Nx+x^{(j)})w)\neq 0$.
    \item For all distinct $x,x'\in L^{r\times u}$, $\Pi_1,\Pi_2\in\mb{F}_q^{r\times s}$ of rank $r$ and $j\in\{1,2\}$, we have $\Pi_1(Nw'+(Nx+x^{(j)})w)\neq\Pi_2(Nw'+(Nx'+x^{(j)})w)$.
    \item For all $x,x'\in L^{r\times u}$ (not necessarily distinct) and $\Pi_1,\Pi_2\in\mb{F}_q^{(r+1)\times s}$ of rank $r+1$, we have $\Pi_1(Nw'+(Nx+x^{(1)})w)\neq\Pi_2(Nw'+(Nx'+x^{(2)})w)$.
    \item For all $x'\in L^{r\times u}$ and $\Pi'\in\mb{F}_q^{r\times s}$ of rank $r$ and $j\in\{1,2\}$, there are $x\in L^{r\times u}$ and $\Pi\in\mb{F}_q^{r\times s}$ of rank $r$ such that $\Pi(Nw'+(Nx+x^{(j)})w)=\Pi'(Nw'+(Nx'+x^{(3-j)})w)$.
\end{itemize}
The first of these guarantees every space in the in-flip and out-flip is actually $s$-dimensional, the second guarantees that each flip is composed of distinct spaces (whose intersections have dimension strictly less than $r$), and the third guarantees that $s$-spaces of the out-flip and in-flip share at most a single $r$-space each. The fourth shows that every $r$-space contained within an $s$-space of one flip will show up in the other, and combining these facts does show $\partial_{s,r}\mc{P}_\mr{out}^{(w',w)}=\partial_{s,r}\mc{P}_\mr{in}^{(w',w)}$.

The fourth bullet here is trivial: take $\Pi=\Pi'$ and $x=x'+(\Pi'N)^{-1}(x^{(j)}-x^{(3-j)})$ (similar to the last part of the proof of \cref{prop:subspace-exchange}). For the first three, the proof in \cref{prop:subspace-exchange} shows that these hold for $(w',w)$ which are $L$-linearly independent (since these are linear equations in $w',w$ with coefficients in $L$ and constant term $0$, the only failure can be due to an $L$-degeneracy of the defining equations, which is ruled out in the proof of \cref{prop:subspace-exchange}).

For the final part of this lemma, suppose we are given the identities of the $s$-spaces in $\mc{P}_\mr{in}^{(w',w)}$. Just choose which spaces correspond to which values $Nw'+Nxw$ for $x=0$ and for values of $x$ with a single nonzero coordinate; this gives a comprehensive set of equations for $w',w$ since the first $r$ rows of $N$ form an invertible matrix.
\end{proof}

We wish to show that there exist many valid absorbers rooted at $s$-spaces. To do this, we first codify a (necessary) condition under which this will hold.
\begin{definition}\label{def:configuration-compatible}
Given the setup of \cref{sub:setup,def:template} and $i\in[z]$, we say that $S\in\mr{Gr}(V,s)$ is \emph{configuration compatible for $i$} if:
\begin{itemize}
    \item $S\in\mr{Gr}_q(K_i,s)$;
    \item $\dim_L\mr{span}_L(\iota_i^{-1}(S))=s$;
    \item There is an $\mb{F}_q$-basis $b\in K_i^s$ of $S$ so that for every $\Pi\in\mr{Red}_q^{r\times s}$, we have for $R=\mr{span}_{\mb{F}_q}(\Pi b)$ that $\Pi_R=\Pi$, $b_R=\Pi b$, and $R\in G_{\mr{tem},i}$.
\end{itemize}
\end{definition}

We now provide the desired lower bound on absorbers involving $s$-spaces. We also bound the number of potential absorbers involving certain $r$-spaces and $s$-spaces, which will be useful in understanding the total influence that certain interactions will have on random quantities for purposes of concentration.

\begin{proposition}\label{prop:absorber-count}
Given the setup of \cref{def:absorber}, we have the following as long as $u\ge s$, $c=c_{\ref{prop:absorber-count}}(u,q)>0$, $C=C_{\ref{prop:absorber-count}}(u,q)$, $q^{-cn}\le\tau\le c$, $d\ge d_{\ref{prop:absorber-count}}(s)$, $\ell\ge\ell_{\ref{prop:absorber-count}}(d,u)$, $C'=C_{\ref{prop:absorber-count}}'(u,\ell,q)$, and $n$ is large:
\begin{itemize}
    \item Given $r$-space $R\leqslant K$ and $s$-space $S\leqslant K$ with $\dim_L\mr{span}_L(R)=r$ and $\dim_L\mr{span}_L(S)=s$ as well as $\dim_L(\mr{span}_L(R)\cap\mr{span}_L(S))=t$, there are at most $C'q^{(u-s+t)n}$ many absorbers with root $S$ such that the in-flip contains an $s$-space $S'\geqslant R$.
    \item Whp over the randomness of the template, for every $i\in[z]$ and every $S\in\mr{Gr}(V,s)$ which is configuration compatible for $i$ there are at least $(\tau/z)^Cq^{(u-s+r)n}$ many absorbers valid for $i$ with root $\iota_i^{-1}(S)$.
\end{itemize}
\end{proposition}
\begin{proof}
For the first bullet point, we claim $\dim_L\mr{span}_L(\mc{P}_\mr{out}^{(w',w)}\cup\mc{P}_\mr{in}^{(w',w)})\le r+u$. Indeed, note that every coordinate of $Nw'+Nxw,Nw'+(Nx+x^\ast)w\in K^s$ can be formed as an $L$-linear combination of coordinates of $w',w$, of which there are $r+u$. Thus given $R,S$ which clearly satisfy $\dim_L\mr{span}_L(R\cup S)=r+s-t$ by the given, there are at most $(r+u)|K|^{(r+u)-(r+s-t)}=q^{(u-s+t)n}$ ways to extend to a space of at most the required dimension. There are then $O_{u,\ell,q,s}(1)$ potential collections within this that form an absorber and by the last part of \cref{lem:absorber-admissible} there are $O_{u,q,s}(1)$ ways to then choose the absorber parameters (note $u\ge s$). Notice this argument used no randomness.

For the second bullet point, we first condition on any revelation of $\iota_1,\ldots,\iota_z$. Now let us consider any $i\in[z]$, any $S\in\mr{Gr}_q(K_i,s)$ satisfying $\dim_L\mr{span}_L(\iota_i^{-1}(S))=s$, and any $b\in K_i^s$ an $\mb{F}_q$-basis of $S$. Consider the event $\mc{E}$ that (a) for all $\Pi\in\mr{Red}_q^{r\times s}$, the $r$-space $R=\mr{span}_{\mb{F}_q}(\Pi b)$ satisfies $\Pi_R=\Pi$, $b_R=\Pi b$, and $R\in G_{\mr{tem},i}$ but also (b) the number of absorbers valid for $i$ with root $\iota_i^{-1}(S)$, call it $X$, satisfies $X<(\tau/z)^Cq^{(r+u-s)n}$. (One of these events occurs if the second bullet fails, by \cref{def:configuration-compatible}.)

In such a situation, let us examine what possible absorbers rooted at $S'=\iota_i^{-1}(S)$ look like. Writing $b'=\iota_i^{-1}(b)$, let us look in particular at such absorbers with $Nw'+x^\ast w=b'$. As we vary $\Pi\in\mr{Red}_q^{r\times s}$, we see that $(\Pi N)w'+\Pi x^\ast w=\Pi b'$ spans the $r$-subspaces of $S'$, call them $R_\Pi'$. Since $\Pi N$ is invertible from \cref{lem:generic}, this satisfies
\[\Pi b' = (\Pi N)(w'+(\Pi N)^{-1}\Pi x^\ast w)\]
so for $x_\Pi = (\Pi N)^{-1}\Pi x^\ast$ we see that $R_\Pi'$ is within the $s$-space generated by $x_\Pi$ in $\mc{P}_\mr{in}^{(w',w)}=\{\mr{span}_{\mb{F}_q}(Nw'+Nxw)\colon x\in L^{r\times u}\}$. Let $\mc{P}_\mr{adj}^{(w',w)}=\{\mr{span}_{\mb{F}_q}(Nw'+Nx_\Pi w)\colon\Pi\in\mr{Red}_q^{r\times s}\}$ be the set of these $s$-spaces (note these $x_\Pi$ are independent of $S,b$, and any randomness). Now let $Y$ be the number of ($L$-linearly independent) choices of parameters so that $S'$ is still the root, but now we only require the absorber to be almost valid in the sense that $\mc{P}_\mr{in}^{(w',w)}\setminus\mc{P}_\mr{adj}^{(w',w)}\subseteq\iota_i^{-1}(\mc{S}_{\mr{tem},i})$.

We claim that under $\mc{E}$, we have $Y<(\tau/z)^Cq^{(r+u-s)n}$. In fact, any absorber counted by $Y$ will actually be counted by $X$ in such a circumstance, implying $Y\le X<(\tau/z)^Cq^{(r+u-s)n}$. The argument is as follows. For each $s$-space in $\mc{P}_\mr{adj}^{(w',w)}$, say $S_\Pi'=\mr{span}_{\mb{F}_q}(Nw'+Nx_\Pi w)$, we have from the above analysis that it shares $R_\Pi'$ with $S'$. Now fix some $\Pi\in\mr{Red}_q^{r\times s}$ and note that $R=R_\Pi=\iota_i(R_\Pi')$ has basis $\Pi b$. By part (a) of $\mc{E}$, we have that $\Pi_R=\Pi$, $b_R=\Pi b$, and $R\in G_{\mr{tem},i}$. But for $R$ to be included in template $i$ given that $\Pi_R=\Pi$ and $b_R=\Pi b$, we must have $i_R=i$ and $\iota_i(\Pi Nx')=\Pi b$ for the defining value $x'\in K^r$ which gives rise to the inclusion of $R$ in this outcome of the template. This gives $x'=(\Pi N)^{-1}\iota_i^{-1}(\Pi b) = (\Pi N)^{-1}\Pi b'$. Recalling $b'=Nw'+x^\ast w$ for the choice of absorber parameters, we have $x'=w'+(\Pi N)^{-1}\Pi x^\ast w=w'+x_\Pi w$. That is, the $r$-space corresponding so $\mr{span}_{\mb{F}_q}(N(w'+x_\Pi w))$ must be in $G_{\mr{tem},i}$. Varying over $\Pi\in\mr{Red}_q^{r\times s}$, this means that $\mc{P}_\mr{adj}^{(w',w)}\subseteq\iota_i^{-1}(\mc{S}_{\mr{tem},i})$. This along with the conditions for $(w',w)$ to be counted for $Y$ mean that it is counted for $X$, and the claim is indeed true.

So now we consider some fixed as above $S,b,i$ (so $\dim_L\mr{span}_L(\iota_i^{-1}(S))=s$ and $b$ is an $\mb{F}_q$-basis of $S$) and wish to study $\mb{P}[Y<(\tau/z)^Cq^{(r+u-s)n}]$ where recall we have conditioned on $\iota_1,\ldots,\iota_z$. We thus are only using the randomness of $i_R,y_R,b_R,\Pi_R$ for all $R\in\mr{Gr}_q(n,r)$. We first show that $Y$ is concentrated around its mean. Let us consider what $Y$ depends on. Note that from the fourth bullet of \cref{lem:template-basic}, we see that the only possible $R$ that $Y$ depends on satisfy $R\leqslant K_i$ and $\dim_L\mr{span}_L(\iota_i^{-1}(R))=r$ (this collection is deterministic given the revealed information).

Let $S'=\iota_i^{-1}(S)$ and consider the collection of such $r$-spaces $R$ satisfying $\dim_L(\mr{span}_L(R')\cap\mr{span}_L(S'))=t$ where $R'=\iota_i^{-1}(R)$ (so we have $\dim_L\mr{span}_L(R')=r$). There are clearly $O_{\ell,q,s}(q^{(r-t)n})$ such spaces: choose a dimension $t$ subspace of $\mr{span}_L(S')$ over $L$ and then choose ways to extend this to an $r$-dimensional space over $L$, and then choose an $\mb{F}_q$-subspace of dimension $r$. Furthermore, changing $i_R,y_R,b_R,\Pi_R$ for one of these can affect $Y$ by at most an amount depending on the number of total absorbers with root $S'$ and containing $R'$. But note that $S',R'$ satisfy the conditions of the first bullet point (which holds deterministically): this count is bounded by $C'q^{(u-s+t)n}\le zq^{(u-s+t)n}$, say. Furthermore, it suffices to consider $0\le t\le r-1$ since $Y$ is not affected by $R\leqslant S$ (due to the exclusion of $\mc{P}_\mr{adj}^{(w',w)}$, and noting that the last four bullet points of \cref{lem:absorber-admissible} imply that there are no other possible interactions of such $R$ with the condition). Thus using \cref{lem:azuma} with variance proxy $O_{\ell,q,s}(\max_{0\le t\le r-1}q^{(r-t)n}\cdot z^2q^{2(u-s+t)n})\le q^{2u-2s+2r-1/2}/2$ yields
\[\mb{P}[|Y-\mb{E}Y|\ge q^{(u-s+r-1/8)n}]\le\exp(-q^{n/4}).\]
This is enough to take a union bound over all $S,b,i$.

Thus it remains to show $\mb{E}Y\ge 2(\tau/z)^Cq^{(u-s+r)n}$. For any $L$-linearly independent parameters $(w',w)$ where $\iota_i^{-1}(S)$ is the corresponding root, we consider the probability that $\mc{P}_\mr{in}^{(w',w)}\subseteq\iota_i^{-1}(\mc{S}_{\mr{tem},i})$. For every $s$-space in $\mc{P}_\mr{in}^{(w',w)}$ of the form $\mr{span}_{\mb{F}_q}(Nw'+Nxw)$ for $x\in L^{r\times u}$, we merely require that it is included in the template. Note that by the first bullet of \cref{lem:absorber-admissible}, $w'+xw\in K^r$ has $L$-linearly independent coordinates so this is possible, and the probability of these events occurring simultaneously is at least $((\tau/z)^A)^{|L|^{ru}}$ for some appropriately chosen $A=A(q,s)$: we make every constituent subspace $R$ have the right set of parameters so that the corresponding $s$-spaces are chosen, all simultaneously (there are no overlaps in conditions due to the third bullet point of \cref{lem:absorber-admissible}).

Therefore it merely remains to give a sufficient lower bound for $Z$, the number of parameters $(w',w)$ with $L$-linearly independent coordinates whose associated absorber has root $S'=\iota_i^{-1}(S)$ (given $\iota_i$, or even just $S'$, $Z$ is non-random). Recall $b'\in K^s$ is an $\mb{F}_q$-basis for $S'$. Equivalently, we need to count $(w',w)\in K^r\times K^u$ with $L$-linearly independent coordinates satisfying $Nw'+x^\ast w=b'$. Let $Z_0$ be the number of solutions to $Nw'+x^\ast w=b'$ for $(w',w)\in K^r\times K^u$ with no conditions, and for each $(v_1,v_2)\in(L^{1\times r}\times L^{1\times u})\setminus\{(0,0)\}$ let $Z_{(v_1,v_2)}$ be the number of solutions to $Nw'+x^\ast w=b'$ satisfying $v_1w'+v_2w=0$. First, since $N,x^\ast$ are jointly $\mb{F}_q$-generic of degree $d\ge d_{\ref{prop:absorber-count}}(s)\ge s$, and since $u\ge s$, we know that the matrix $N'$ obtained by augmenting $N$ by $x^\ast$ on the right has rank $s$ (e.g.~since the initial $s\times s$ block has nonzero determinant). Therefore $Z_0=(q^n)^{r+u-s}$. For $(v_1,v_2)$ that is not in the $L$-span of the rows of $N'$, we see that adding in the equation $v_1w'+v_2w=0$ drops the total number of solutions to $Z_{(v_1,v_2)}=(q^n)^{r+u-s-1}$. Finally, for $(v_1,v_2)\neq(0,0)$ in the span of the rows of $N'$, we can find $y\in L^{1\times s}$ with $y\neq 0$ so that
\[yN=v_1,\quad yx^\ast=v_2\]
and then $0=v_1w'+v_2w=y(Nw'+x^\ast w)=yb'$. Thus if there is a solution then $b'$ is $L$-linearly dependent. But recall again that $\dim_L\mr{span}_L(S')=\dim_L\mr{span}_L(\iota_i^{-1}(S))=s$, so this cannot happen. Thus for these values, $Z_{(v_1,v_2)}=0$.

We are done: this analysis implies that for large $n$ we have
\[Z\ge Z_0-\sum_{(v_1,v_2)\neq(0,0)}Z_{(v_1,v_2)}\ge(q^n)^{u-s+r}/2.\qedhere\]
\end{proof}

\section{Using the absorber}\label{sec:absorption}
The goal of this section is to provide a general setup for using the abundance of absorbers in the template, via a random process, to complete a decomposition. At a high level, we will show that a collection of $s$-spaces satisfying certain properties with respect to the template can be changed instead to a sum of distinct template $s$-spaces minus other $s$-spaces. This will be used in the final step of our absorption algorithm; in \cref{sec:approximate} we will create an approximate decomposition and then in \cref{sec:cover} we will show how we can arrive at the desired situation from it.

Since the availability of template absorbers involves various conditions defined over $L$ as seen in \cref{prop:absorber-count}, beyond the notion of boundedness from \cref{def:bounded} we will require a notion of boundedness of a collection of $r$-spaces with respect the field structure $L$ imposed.

\begin{definition}\label{def:field-bounded}
Given the setup of \cref{def:template}, and for $t\in\{r,s\}$, we define $\iota_i^\ast\colon\mb{Z}^{\mr{Gr}_q(n,t)}\to\mb{Z}^{\mr{Gr}(K,t)}$ to map $e_T\mapsto e_{\iota_i^{-1}(T)}$ if $T\leqslant K_i$ and $\dim_L\mr{span}_L(\iota_i^{-1}(T))=t$, and $e_T\mapsto 0$ otherwise. We define $\partial_{t,r-1}^L\colon\mb{Z}^{\mr{Gr}_{\mb{F}_q}(K,t)}\to\mb{Z}^{\mr{Gr}_L(K,r-1)}$ via
\[e_T\mapsto\sum_{Q^\ast\leqslant\mr{span}_L(T)}e_{Q^\ast}.\]
We say a $t$-dimensional signed multi-$q$-system $\Phi$ on $V=\mb{F}_q^n$ is \emph{$(\theta,L)$-field bounded with respect to the template} if for all $i\in[z]$ we have $\snorm{\partial_{t,r-1}^L\iota_i^\ast(\Phi^\pm)}_\infty\le\theta q^n$.
\end{definition}
\begin{remark}
Here $\mr{Gr}_L(K,r-1)$ denotes the $L$-subspaces of $K$ with $L$-dimension $r-1$. Note that this definition gives no bound on the parts of $\mc{T}$ which are not contained in some $K_i$ or fail the $L$-dimension condition, so it will generally be applied only to $\mc{T}$ that are already supported on such $t$-spaces. In the case $t=r$, note that all $R\in G_\mr{tem}$ will appear in these bounds by the fourth bullet of \cref{lem:template-basic}.
\end{remark}

We briefly note that in relevant situations, boundedness in the sense of \cref{def:bounded} is weaker than field boundedness, which will simplify matters later.
\begin{lemma}\label{lem:strength-bounded}
Given the setup of \cref{def:template} and as long as $c=c_{\ref{lem:strength-bounded}}(q,s) > 0$, $q^{-cn}\le\tau\le c$, $d\ge d_{\ref{lem:strength-bounded}}(r)$, $\ell\ge\ell_{\ref{lem:strength-bounded}}(d,s)$, any $\mc{R}\in\mb{Z}^{G_\mr{tem}}$ which is $(\theta,L)$-field bounded is also $z\theta$-bounded.
\end{lemma}
\begin{proof}
Note that $\mc{R}$ is supported on $G_\mr{tem}\subseteq\mr{Gr}(V,r)$. For any $R\in\mc{R}$ there is $i\in[z]$ with $R\in G_{\mr{tem},i}$ and by the fourth bullet of \cref{lem:template-basic} this satisfies $\dim_L\mr{span}_L(\iota_i^{-1}(R))=r$.

Now consider $Q\in\mr{Gr}(V,r-1)$. We wish to bound the number of $r$-spaces in $\mc{R}^+$, say, that contain $Q$ (with multiplicity). For each $i\in[z]$ consider the number of $r$-spaces in $\mc{R}^+$ and $G_{\mr{tem},i}$ containing $Q$. If we do not have $Q\leqslant K_i$ then there are $0$ such spaces, and otherwise we may assume $Q\leqslant K_i$. If $\dim_L\mr{span}_L(\iota_i^{-1}(Q))=r-1$ then let $Q_i^\ast=\mr{span}_L(\iota_i^{-1}(Q))$ and apply \cref{def:field-bounded} to $Q_i^\ast$ and index $i$. If $R\geqslant Q$ and $R\in G_{\mr{tem},i}$ then by the above consideration $\dim_L\mr{span}_L(\iota_i^{-1}(R))=r$ and this means the number of occurrences of $R$ in $\mc{R}$ will be counted by \cref{def:field-bounded} for $Q_i^\ast$. Thus we obtain a bound of $\theta q^n$ in this case. Now suppose that $\dim_L\mr{span}_L(\iota_i^{-1}(Q))<r-1$. We claim no $R\in G_{\mr{tem},i}$ contains $Q$. Indeed, if this did occur then $\iota_i^{-1}(Q)\leqslant\iota_i^{-1}(R)=:R'$ would occur. Let $\{v_1,\ldots,v_{r-1}\}\subseteq K$ be a basis for $\iota_i^{-1}(Q)$ and $v_r\in K$ be chosen to extend it to a basis of $R'$. Again \cref{lem:template-basic} tells us that $R\in G_{\mr{tem},i}$ implies $\dim_L\mr{span}_L(R')=r$, hence $\{v_1,\ldots,v_r\}$ is $L$-independent. Thus so is $\{v_1,\ldots,v_{r-1}\}$, so in fact $\iota_i^{-1}(Q)$ spans $r-1$ dimensions over $L$, contradicting our assumption!

Putting everything together, the contribution of $r$-spaces in $\mc{R}^+$ containing $Q$ coming from $G_{\mr{tem},i}$ is bounded by $\theta q^n$. Adding up over all $i\in[z]$ yields a bound of at most $z\theta q^n$, so $\mc{R}$ is $z\theta$-bounded as desired.
\end{proof}

We now prove the main absorption statement.

\begin{proposition}\label{prop:final-absorber}
Given the setup of \cref{def:template}, if we have $c=c_{\ref{prop:final-absorber}}(q,s)>0$, $q^{-cn}\le\tau\le c$, $d\ge d_{\ref{prop:final-absorber}}(s)$, $\ell\ge\ell_{\ref{prop:final-absorber}}(d,s)$, and $n$ is large then whp (over the randomness of the template) the following holds for $\theta\le(\tau/z)^{1/c}$. Consider any $\Phi\in\{0,1\}^{\mr{Gr}_q(n,s)}$ such that $\snorm{\partial_{s,r}\Phi}_{\infty}\le 1$ and $\Phi$ is $(\theta,L)$-field bounded with respect to the template. Suppose that for each $S\in\mr{supp}(\Phi)$ there is $i\in[z]$ such that (a) $\dim_L\mr{span}_L(\iota_i^{-1}(S))=s$ and (b) there is an $\mb{F}_q$-basis $b\in K_i^s$ of $S$ so that for every $\Pi\in\mr{Red}_q^{r\times s}$, we have for $R=\mr{span}_{\mb{F}_q}(\Pi b)$ that $\Pi_R=\Pi$, $b_R=\Pi b$, and $R\in G_{\mr{tem},i}$. Suppose additionally that (c) for all $i\in[z]$ and distinct $S_1,S_2\in\mr{supp}(\Phi)\cap\mr{Gr}_q(n,s)[G_{\mr{tem},i}]$ we have $\dim_L(\mr{span}_L(\iota_i^{-1}(S_1)\cap\iota_i^{-1}(S_2)))<r$.

Then there exist $\Phi_1$, $\Phi_2$ such that $\Phi_1,\Phi_2\in\{0,1\}^{\mr{Gr}_q(n,s)}$, $\mr{supp}(\Phi_2)\subseteq\mc{S}_\mr{tem}$, and
\[\partial_{s,r}\Phi=\partial_{s,r}(\Phi_2-\Phi_1).\]
\end{proposition}
\begin{remark}
Note that (b) above implies that $S\leqslant K_i$ and hence $\iota_i^{-1}(S)$ in (a) is well-defined. Also, the condition (c) is required in order to rule out the case that $S_1,S_2$ involve two different $r$-spaces participating in the same template $s$-space $S\in\mc{S}_\mr{tem}$, which would destroy any hopes of using template absorbers for $S_1,S_2$ simultaneously (both would require use of $S$).
\end{remark}
\begin{proof}
Fix $u=s$ and note that $N=N_\mr{tem},x^\ast=x_\mr{abs}^\ast$ as defined in \cref{sub:setup} satisfy the setup of \cref{def:absorber}. (One could substitute any value of $u$ which is at least $s$; for the sake of clarifying the forthcoming calculations we will differentiate $u$ and $s$.)  Then note that whp over the randomness of the template \cref{prop:absorber-count} holds, and let $C=C_{\ref{prop:absorber-count}}(u,q,s)$. Now condition on any revelation of these random outcomes satisfying this. Let $\{S_1,\ldots,S_{|\Phi|}\}$ be a labeling of the $s$-spaces in $\Phi$. By the given condition (b), $S_t$ has all its $r$-spaces in $G_{\mr{tem},i_t}$ for some $i_t\in[z]$. Consider the collection $\mc{W}_t$ of parameters $(w',w)$ for absorbers with root $S_t$ that are valid for template index $i_t$. For each $t\in[|\Phi|]$ we sample a uniformly random choice of parameters $(w_t',w_t)\sim\mr{Unif}(\mc{W}_t)$. Call an outcome successful if $\mc{P}_\mr{in}^{(w_t',w_t)}$ for $t\in[|\Phi|]$ are simultaneously disjoint.

By the fifth bullet of \cref{lem:absorber-admissible} we can write
\[\partial_{s,r}e_{S_t}=\partial_{s,r}\mc{P}_\mr{in}^{(w_t',w_t)}-\partial_{s,r}(\mc{P}_\mr{out}^{(w_t',w_t)}\setminus\{S_t\})\]
for all $t\in[|\Phi|]$, and then summing this up shows that for $\Phi_2=\bigcup_{t=1}^{|\Phi|}\mc{P}_\mr{in}^{(w_t',w_t)}$ and $\Phi_1=\bigcup_{t=1}^{|\Phi|}(\mc{P}_\mr{out}^{(w_t',w_t)}\setminus\{S_t\})$ as multisets, we have $\partial_{s,r}\Phi=\partial_{s,r}(\Phi_2-\Phi_1)$. We see that $\Phi_2$ has no repetitions due to the definition and success of the process. Then we claim $\Phi_1$ is therefore forced to have no repetitions. Indeed, $\partial_{s,r}\Phi$ has no repetitions by $\snorm{\partial_{s,r}\Phi}_\infty\le 1$, and we showed $\Phi_2$ has no repetitions but it is within $\mc{S}_\mr{tem}$ hence $\partial_{s,r}\Phi_2$ has no repetitions as well (since the $s$-spaces defining the template have disjoint constituent $r$-spaces). Since $\partial_{s,r}\Phi=\partial_{s,r}(\Phi_2-\Phi_1)$ and $\Phi_1$ has nonnegative coefficients, we deduce that $\snorm{\partial_{s,r}\Phi_1}_\infty\le 1$. Therefore $\Phi_1,\Phi_2$ satisfy the necessary conditions.

Now it suffices to show that a successful outcome occurs with positive probability. For this we use the Lov\'asz Local Lemma (\cref{lem:lll}). For $1\le t_1,t_2\le|\Phi|$ with $t_1\neq t_2$ let $\mc{B}_{t_1,t_2}$ be the event that $\mc{P}_\mr{in}^{(w_{t_1}',w_{t_1})},\mc{P}_\mr{in}^{(w_{t_2}',w_{t_2})}$ share an $s$-space, or equivalently that $i_{t_1} = i_{t_2}$ and the randomly chosen absorbers with roots $S_{t_1}$ and $S_{t_2}$ have in-decompositions sharing a template $s$-space. Furthermore let $S_{t_j}^\ast=\mr{span}_L(\iota_{i_{t_j}}^{-1}(S_{t_j}))$ for $j\in\{1,2\}$ and $d_{t_1,t_2} = \dim_L(S_1^\ast\cap S_2^\ast)$. By the given condition (c), we have $d_{t_1,t_2}<r$.

Assuming $\mc{B}_{t_1,t_2}$ holds, we must have $i_{t_1}=i_{t_2}=i$ for some $i\in[z]$. Call the shared $s$-space in the in-decompositions $S'$, let $S=\iota_i^{-1}(S')$, and let $R\leqslant S$ be an arbitrary $r$-subspace. Suppose $u_j=\dim_L(\mr{span}_L(R)\cap S_j^\ast)$ for $j\in\{1,2\}$ and let $u_0=\dim_L\mr{span}_L(R\cup S_1^\ast\cup S_2^\ast)-\dim_L\mr{span}_L(S_1^\ast\cup S_2^\ast)$. By the first bullet point of \cref{prop:absorber-count}, for $C'=C_{\ref{prop:absorber-count}}'(u,\ell,q)$ there are at most
\[O_{\ell,q,r}(q^{u_0n})\cdot C'q^{(u-s+u_1)n}\cdot C'q^{(u-s+u_2)n}\]
choices of absorber parameters that yield such a situation: first choose $R$ by picking $R\cap\mr{span}_L(S_1^\ast\cup S_2^\ast)$ in $O_{\ell,q,r}(1)$ ways and then extending to $u_0$ further dimensions, and second choose the absorbers containing $S_{t_1},R$ and $S_{t_2},R$ respectively. But then using the second bullet of \cref{prop:absorber-count} and the given conditions on all $S\in\mr{supp}(\Phi)$, we see for $C=C_{\ref{prop:absorber-count}}(u,q,s)$ there are at least $(\tau/z)^Cq^{(u-s+r)n}$ choices of absorber for $S_{t_1}$ and $S_{t_2}$ each. Finally, we have
\begin{align*}
r-u_0&=\dim_L\mr{span}_L(R)-\dim_L\mr{span}_L(R\cup S_1^\ast\cup S_2^\ast)+\dim_L\mr{span}_L(S_1^\ast\cup S_2^\ast)\\
&=\dim_L(\mr{span}_L(R)\cap\mr{span}_L(S_1^\ast\cup S_2^\ast))\ge u_1+u_2-d_{t_1,t_2}
\end{align*}
where the first equality uses the fourth bullet of \cref{lem:template-basic} and where the inequality uses
\[\dim(V_1\cap(V_2+V_3))\ge\dim((V_1\cap V_2)+(V_1\cap V_3))=\dim(V_1\cap V_2)+\dim(V_1\cap V_3)-\dim(V_1\cap V_2\cap V_3)\]
for finite-dimensional vector spaces $V_1,V_2,V_3$ within some host vector space (note this is not an equality in general). Therefore a union bound over possible choices of $u_0,u_1,u_2$ yields
\begin{align*}
\mb{P}[\mc{B}_{t_1,t_2}]&\le O_{\ell,q,r}(1)\cdot(z/\tau)^{2C}q^{-2(u-s+r)n}\max_{r-u_0\ge u_1+u_2-d_{t_1,t_2}}q^{u_0n}q^{(2(u-s)+u_1+u_2)n}\\
&\le O_{\ell,q,r}((z/\tau)^{2C}q^{(d_{t_1,t_2}-r)n}).
\end{align*}

Now we construct a dependency graph for the defined events: connect $\mc{B}_{t_1,t_2}$ and $\mc{B}_{t_3,t_4}$ if $|\{t_1,t_2\}\cap \{t_3,t_4\}|\ge 1$. Furthermore let $x_{t_1,t_2} = q^{-(r-d_{t_1,t_2})n}(z/\tau)^{2C+1}$. Notice that for fixed $(t_1,t_2)$,
\begin{align*}
x_{t_1,t_2}\prod_{t\neq t_1}(1-x_{t_1,t})\prod_{t\neq t_2}(1-x_{t,t_2})&\ge q^{-(r-d_{t_1,t_2})n}(z/\tau)^{2C+1}\prod_{d=0}^{r-1}(1-q^{-(r-d)n}(z/\tau)^{2C+1})^{O_{q,s,\ell}(z\theta q^{(r-d)n})}\\
&\ge q^{-(r-d_{t_1,t_2})n}(z/\tau)^{2C+1}\exp(-O_{q,s,\ell}(z\theta (z/\tau)^{2C+1}))\\
&\ge\mb{P}[\mc{B}_{t_1,t_2}]
\end{align*}
as long as $c=c_{\ref{prop:final-absorber}}(q,s)$ is small enough in terms of $C$ so that $1/c\ge 2C+2$ and hence $z\theta(z/\tau)^{2C+1}\le 1$. We have used $d_{t_1,t_2}<r$ and that for any $S^\ast$ which is an $s$-dimensional space over $L$ and for any $0\le d\le r-1$,
\[\#\{t\in[|\Phi|]\colon\dim_L(S^\ast\cap\mr{span}_L(\iota_{i_t}^{-1}(S_t)))=d\}\le z\cdot O_{q,s,\ell}(\theta q^{(r-d)n})\]
due to the $(\theta,L)$-field boundedness of $\Phi$. Indeed, we argue as follows. First we reduce to the case $d=r-1$ by considering at most $q^{(r-d-1)n}$ ways to augment $S^\ast$ to a slightly larger space $T^\ast$ such that the intersection with $\mr{span}_L(\iota_{i_t}^{-1}(S_t))$ has $L$-dimension $r-1$. Then there are $O_{q,s,\ell}(1)$ ways to choose the $(r-1)$-dimensional $L$-space $Q^\ast=T^\ast\cap\mr{span}_L(\iota_{i_t}^{-1}(S_t))$. Finally, $(\theta,L)$-field boundedness shows that for each $i\in[z]$, $\Phi$ contains in its support at most $\theta q^n$ many $s$-spaces $S_t$ with $i_t=i$ so that $\mr{span}_L(\iota_i^{-1}(S_t))$ contains $Q^\ast$ (the condition \ref{prop:final-absorber}(a) shows that the bound applies).

Finally, \cref{lem:lll} demonstrates that with positive probability none of the $\mc{B}_{t_1,t_2}$ hold, and we are done by the earlier discussion.
\end{proof}

\section{Approximate covering}\label{sec:approximate}
In this section we will use a randomized process to cover almost all of the $r$-spaces in $\mr{Gr}_q(n,r)$ outside the template using $s$-spaces, leaving a small portion. To do so we will use well-studied general purpose tools regarding greedy random matching processes; we will use the work of \cite{EGJ20} for our situation.

\subsection{Regularizing \texorpdfstring{$s$}{s}-spaces with respect to the template}\label{sub:template-regularize}
The size of the leftover from the greedy random process will depend on a certain degree of irregularity that we start with, induced by the removal of the template. In order to ensure that the leftover can actually be small in an appropriate sense with respect to the template, we have to first regularize the complement. More specifically, we find a subset of $s$-spaces such that the number of $s$-spaces containing a fixed $r$-space is substantially more regular than if we used the entire set of $s$-spaces.

At a high level we need a version of \cref{lem:template-basic} (but for the complement of the template) in which the regularity of the collection of $s$-spaces not involving $r$-spaces of $G_\mr{tem}$ is decoupled from $\tau$. This can be achieved by finding an appropriate weighting of $\mr{Gr}_q(n,s)[G\setminus G_\mr{tem}]$, constructed using that the lattice $\mc{L}$ is locally decodable (see e.g.~\cref{prop:robust-local-decodability}), and then sampling.

\begin{lemma}\label{lem:template-boosted}
Given the setup of \cref{sub:setup,def:template}, we have the following as long as $q^{-c_{\ref{lem:template-boosted}}(q,s)n}\le\tau\le c_{\ref{lem:template-boosted}}(q,s)$, $d\ge d_{\ref{lem:template-boosted}}(r)$, $\ell\ge\ell_{\ref{lem:template-boosted}}(d,s)$, and $n$ is large. Recall an $s$-space is \emph{obstructed} by the template if any of its $r$-dimensional subspaces is contained in $G_\mr{tem}$. Whp over the randomness of the template, there exists a set $\mc{S}_\mr{reg}\subseteq\mr{Gr}_q(n,s)$ of unobstructed $s$-spaces such that for each $r$-space $R\notin G_\mr{tem}$ we have
\[\#\{S\in\mc{S}_\mr{reg}\colon S\geqslant R\}\ge (1-\tau^{1/4})\#\{S\in\mr{Gr}(V,s)\colon S\geqslant R\}\]
and for distinct $r$-spaces $R_1,R_2\notin G_\mr{tem}$ we have
\[\#\{S\in\mc{S}\colon S\geqslant R_1\} = (1\pm q^{-n/3})\#\{S\in\mc{S}\colon S\geqslant R_2\}.\]
\end{lemma}
\begin{proof}
From the first part of the proof of \cref{prop:robust-local-decodability}, given $r$-space $R$ and $(r+s)$-space $T$ such that $R\leqslant T$, there exists $f\colon\{0,1,\ldots,r\}\to\mb{Q}$ depending only on $q,s,r$ such that 
\[e_R = \sum_{S\in\mr{Gr}(T,s)}f(\dim(S\cap R))\partial_{r,s}e_S.\]
(We derive this by starting with \cref{eq:local-decodability}, dividing by $\Delta$, and then averaging the resulting equation over all automorphisms of $T$ which permute $R$ among itself; one can check that the resulting equation has rational coefficients for each $\partial_{r,s}e_S$ that depend only on $\dim(S\cap R)$.)

Thus given $R\leqslant T$ (of dimensions $r,r+s$ respectively) and $R'\in\mr{Gr}_q(n,r)$ we have 
\[\mbm{1}_{R=R'} = \sum_{S\in\mr{Gr}(T,s)}f(\dim(S\cap R))\mbm{1}_{R'\leqslant S}.\]

Let $\mc{T}_r$ denote the set of $r$-spaces in $\mr{Gr}(V,r)\setminus G_\mr{tem}$ and let $\mc{T}_s = \{S\in\mr{Gr}(V,s)\colon\mr{Gr}(S,r)\subseteq\mc{T}_r\}$ (the unobstructed $s$-spaces). For each $r$-space $R\in\mc{T}_r$, let $\mc{T}_s(R) = \{S\in\mc{T}_s\colon S\geqslant R\}$ and let $\mc{T}_{r+s}(R) = \{T\in\mr{Gr}(V,r+s)\colon T\geqslant R,~\mr{Gr}(T,r)\subseteq\mc{T}_r\}$.

For each $R\in\mc{T}_r$, define 
\[c_R = \frac{\qbinom{n-r}{s-r}_q-|\mc{T}_s(R)|}{|\mc{T}_{r+s}(R)|}.\]
Note that by definition, $|\mc{T}_s(R)|\le\qbinom{n-r}{s-r}_q$ and $|\mc{T}_{r+s}(R)|\le\qbinom{n-r}{s}_q$ (these being the number of total $t$-spaces within $V$ extending $R$ for $t\in\{s,r+s\}$ respectively). Additionally, by the third bullet of \cref{lem:template-basic} we have that $|\mc{T}_s(R)|\ge (1-\tau^{1/2})\qbinom{n-r}{s-r}_q$ and also $|\mc{T}_{r+s}(R)|\ge (1-\tau^{1/2})\qbinom{n-r}{s}_q$. Therefore it follows that for each $r$-space $R\notin G_\mr{tem}$,
\[|c_R|\le\frac{2\tau^{1/2}}{\qbinom{n-r}{s}_q/\qbinom{n-r}{s-r}_q}.\]

For $R$ and $T\in\mc{T}_{r+s}(R)$, define $\psi_{R,T}\colon\mc{T}_s\to\mb{R}$ by
\[\psi_{R,T}(S) = \begin{cases}f(\dim(S\cap R))&\text{if }S\leqslant T,\\0&\text{if }S\not\leqslant T.\end{cases}\]
By definition for any $r$-spaces $R, R'\in\mc{T}_r$ and any $T\in\mc{T}_{r+s}(R)$ we have
\[\sum_{R'\leqslant S\leqslant T}\psi_{R,J}(S) = \sum_{S\in\mr{Gr}(T,s)}f(\dim(S\cap R))\mbm{1}_{R'\leqslant S} = \mbm{1}_{R = R'}.\]
Now, define $\psi\colon\mc{T}_s\to\mb{R}$ by
\[\psi(S) = 1 + \sum_{R\in \mc{T}_r}c_R\sum_{T\in\mc{T}_{r+s}(R)}\psi_{R,T}(S).\]
Notice that 
\begin{align*}
|\psi(S) - 1|&\le\sum_{R\in\mc{T}_r}c_R\sum_{T\in\mc{T}_{r+s}(R)}|\psi_{R,T}(S)|\le\sum_{R\in\mc{T}_r}\frac{2\tau^{1/2}}{\qbinom{n-r}{s}_q/\qbinom{n-r}{s-r}_q}\sum_{\substack{T\in\mc{T}_{r+s}(R)\\T\geqslant S}}|\psi_{R,T}(S)|\\
&\le\frac{2\tau^{1/2}\snorm{f}_\infty}{\qbinom{n-r}{s}_q/\qbinom{n-r}{s-r}_q}\sum_{R\in\mc{T}_r}\sum_{\substack{T\in\mc{T}_{r+s}(R)\\T\geqslant S}}1\le\frac{2\tau^{1/2}\snorm{f}_\infty}{\qbinom{n-r}{s}_q/\qbinom{n-r}{s-r}_q}\qbinom{n-s}{r}_q\qbinom{r+s}{r}_q\le\tau^{3/8},
\end{align*}
where the second-to-last inequality is proven by counting choices of $T$ containing $S$ and then $R$ within $T$ and the last inequality is valid as long as $\tau$ is small with respect to $q,s$.

Furthermore for each $R\in\mc{T}_r$ we have
\begin{align*}
\sum_{S\in\mc{T}_s(R)}\psi(S) &= |\mc{T}_{s}(R)| + \sum_{R'\in \mc{T}_r}c_{R'}\sum_{T\in\mc{T}_{s+r}(R')}\sum_{S\in\mc{T}_s(R)}\psi_{R',T}(S)\\
&= |\mc{T}_{s}(R)| + \sum_{R'\in \mc{T}_r}c_{R'}\sum_{T\in\mc{T}_{s+r}(R')}\sum_{\substack{R\leqslant S\leqslant T\\\dim S = s}}f(\dim(S\cap R'))\\
&=  |\mc{T}_s(R)| + \sum_{R'\in \mc{T}_{r}}c_{R'}\sum_{T\in\mc{T}_{s+r}(R')}\mbm{1}_{R=R'} \;=\; |\mc{T}_{s}(R)| + c_{R}|\mc{T}_{s+r}(R)|\\
&=\qbinom{n-r}{s-r}_q.
\end{align*}

Finally, we define a random $\mc{S}_\mr{reg}\subseteq\mc{T}_s$ by independently including each $S\in\mc{T}_s$ with probability $\psi(S)/(1+\tau^{3/8})$ (this is well defined as $|\psi(T)-1|\le\tau^{3/8}$). For every $R\in\mc{T}_r$, the expected number of $T\in\mc{S}_\mr{reg}$ containing $R$ is exactly $(1+\tau^{3/8})^{-1}\qbinom{n-r}{s-r}_q$ by above. We immediately see $\mc{S}_\mr{reg}$ satisfies the desired conditions with positive probability by the Chernoff bound, and we are done.
\end{proof}

\subsection{Approximate covering given regularization}\label{sub:approximate-cover}
We now cover the majority of subspaces outside the template using results from hypergraph matching. For a hypergraph $\mc{H}$, define
\[\Delta(\mc{H}):=\max_{v\in V(\mc{H})}\deg_\mc{H}(v),\quad\Delta^\mr{co}(\mc{H}):=\max_{v_1,v_2\in V(\mc{H})}\mr{codeg}_{\mc{H}}(v_1,v_2).\]

Call a function $\omega\colon E(\mc{H})\to\mb{R}_{\ge 0}$ a \textit{weight function}, and for $X\subseteq E(\mc{H})$ let $\omega(X)=\sum_{x\in X}\omega(x)$. We will use the following result of Ehard, Glock, and Joos \cite{EGJ20} which guarantees the existence of hypergraph matchings which are pseudorandom with respect to a collection of weight functions.

\begin{theorem}[{\cite[Theorem~1.2]{EGJ20}}]\label{thm:matching}
Suppose $\beta\in(0,1)$ and $r\in\mb{N}$ with $r\ge 2$, and let $\eps:=\beta/(50r^2)$. Then there exists $\Delta_0$ such that for all $\Delta\ge\Delta_0$ the following holds:
Let $\mc{H}$ be an $r$-uniform hypergraph with $\Delta(\mc{H})\le\Delta$ and $\Delta^\mr{co}(\mc{H})\le\Delta^{1-\beta}$ as well as $e(\mc{H})\le\exp(\Delta^{\eps^2})$. 
Suppose that $\mc{W}$ is a set of at most $\exp(\Delta^{\eps^2})$ weight functions on~$E(\mc{H})$.
Then, there exists a matching $\mc{M}$ in~$\mc{H}$ such that $\omega(\mc{M})=(1\pm\Delta^{-\eps})\omega(E(\mc{H}))/\Delta$ for all $\omega\in\mc{W}$ with $\omega(E(\mc{H}))\ge\max_{e\in E(\mc{H})}\omega(e)\Delta^{1+\beta}$.
\end{theorem}

Given this, we record the following consequence in our setting.
\begin{proposition}\label{prop:approximate-covering}
Given the setup of \cref{sub:setup,def:template}, we have the following as long as $q^{-c_{\ref{prop:approximate-covering}}(q,s)n}\le\tau\le c_{\ref{prop:approximate-covering}}(q,s)$, $d\ge d_{\ref{prop:approximate-covering}}(r)$, $\ell\ge\ell_{\ref{prop:approximate-covering}}(d,s)$, and $n$ is large. There is $\eta = \eta_{\ref{prop:approximate-covering}}(q,s)>0$ so that whp over the template, there exists a set $\mc{S}_\mr{approx}\subseteq\mr{Gr}(V,s)$ such that:
\begin{itemize}
    \item Every $r$-space $R\leqslant V$ appears at most once in an $s$-space of $\mc{S}_\mr{tem}\cup\mc{S}_\mr{approx}$;
    \item For $G_\mr{approx}=\partial_{s,r}\mc{S}_\mr{approx}$ we have that $G\setminus(G_\mr{tem}\cup G_\mr{approx})$ is $q^{-\eta n}$-bounded (\cref{def:bounded}).
\end{itemize}
\end{proposition}
\begin{remark}
Note here that $\eta$ is independent of $\tau$, and recall $G=\mr{Gr}_q(n,r)$.
\end{remark}
\begin{proof}
This is essentially an immediate application of \cref{lem:template-boosted} and \cref{thm:matching}.

Consider the collection $\mc{S}_\mr{reg}$ of $s$-spaces guaranteed by \cref{lem:template-boosted} (existing whp). Let the vertices of hypergraph $\mc{H}$ be the collection of $r$-spaces not in $G_\mr{tem}$ and the edges of $\mc{H}$ be the collections of $r$-spaces contained in an $s$-space of $\mc{S}_\mr{reg}$. In particular, $\mc{H}$ is a $\qbinom{s}{r}_q$-uniform hypergraph. Notice that (recall $G=\mr{Gr}_q(n,r)$)
\[\Delta(\mc{H}) = \max_{R\in G\setminus G_\mr{tem}}\#\{S\in\mc{S}_\mr{reg}\colon S\geqslant R\},\quad\delta(\mc{H}) = \min_{R\in G\setminus G_\mr{tem}}\#\{S\in\mc{S}_\mr{reg}\colon S\geqslant R\}\]
hence $\Delta(\mc{H}) = (1\pm q^{-n/3})\delta(\mc{H})\ge \qbinom{n-r}{s-r}_q/2$ and $\Delta(\mc{H})\le\qbinom{n-r}{s-r}_q$. Furthermore note that 
$\Delta^\mr{co}(\mc{H})\le q^{(s-(r+1))n}\le q^{-n/2}\Delta(\mc{H})$ since distinct $r$-spaces $R,R'$ satisfy $\dim(\mr{span}_{\mb{F}_q}(R\cup R'))\ge r+1$.

We now specify the weight functions. For each $(r-1)$-space $Q$ define $\omega_Q\colon\mc{S}_\mr{reg}\to\mb{R}_{\ge 0}$ via $\omega_Q(S) = \mbm{1}_{S\geqslant Q}$. We have that $\omega_Q(E(\mc{H}))\ge 2^{-1}\qbinom{n-r+1}{s-r+1}_q$ given $\tau$ is sufficiently small, and there are at most $q^{(r-1)n}$ weight functions, and therefore (say) $\beta = 1/(4(r+s+1))$ and $\Delta=\Delta(\mc{H})$ are admissible for $n$ sufficiently large in \cref{thm:matching}. Thus there exists a matching such that $\mc{M}$ in $\mc{H}$ such that $\omega_Q(\mc{M})=(1\pm \Delta^{-\eps})\omega_Q(\mc{S}_\mr{reg})/\Delta$ for all $Q\in\mr{Gr}_q(n,r-1)$. This is equivalent to 
\[\Delta\sum_{S\in\mc{M}}\mbm{1}_{Q\leqslant S} = (1\pm\Delta^{-\eps})\sum_{S\in\mc{S}_\mr{reg}}\mbm{1}_{Q\leqslant S}.\] 
This implies
\begin{equation}\label{eq:regular-remainder}
\Delta\sum_{\substack{\dim R=r\\S\in\mc{M}}}\mbm{1}_{Q\leqslant R\leqslant S} = (1\pm\Delta^{-\eps})\sum_{\substack{\dim R=r\\S\in\mc{S}_\mr{reg}}}\mbm{1}_{Q\leqslant R\leqslant S}.
\end{equation}
We write $R\in V(\mc{M})$ to mean $R$ is covered by one of the hyperedges (corresponding to $s$-spaces) in $\mc{M}$. Therefore it follows that for fixed $Q\in\mr{Gr}_q(n,r-1)$,
\begin{align}
\sum_{\substack{R\notin G_\mr{tem}\\R\notin V(\mc{M})}}\mbm{1}_{Q\leqslant R}&=\sum_{R\notin G_\mr{tem}}\mbm{1}_{Q\leqslant R}-\sum_{R\in V(\mc{M})}\mbm{1}_{Q\leqslant R}=\frac{1\pm 2q^{-n/3}}{\Delta}\sum_{\substack{R\notin G_\mr{tem}\\S\in \mc{S}_\mr{reg}}}\mbm{1}_{Q\leqslant R\leqslant S}-\sum_{\substack{\dim R=r\\S\in\mc{M}}}\mbm{1}_{Q\leqslant R\leqslant S}\notag\\
&=\frac{1\pm 2q^{-n/3}}{\Delta}\sum_{\substack{\dim R=r\\S\in \mc{S}_\mr{reg}}}\mbm{1}_{Q\leqslant R\leqslant S}-\sum_{\substack{\dim R=r\\S\in\mc{M}}}\mbm{1}_{Q\leqslant R\leqslant S}\notag\\
&=(1\pm 4(q^{-n/3}+\Delta^{-\eps}))\sum_{\substack{\dim R=r\\S\in\mc{M}}}\mbm{1}_{Q\leqslant R\leqslant S}-\sum_{\substack{\dim R=r\\S\in\mc{M}}}\mbm{1}_{Q\leqslant R\leqslant S}\label{eq:regularity-boost}
\end{align}
by definition, using the relationship between $\Delta(\mc{H}),\delta(\mc{H})$, simplifying, and finally using \cref{eq:regular-remainder}. Now \cref{eq:regularity-boost} is bounded by $q^{-cn}\cdot q^n$ for some appropriate $c$ depending only on $q,s$ since there are at most $q^n$ many $R$ containing $Q$ and each is in exactly one or zero $S\in\mc{M}$. We used that $\eps$ is purely a function of $q$, $s$, and $r$. The desired result follows, noting that $\mc{S}_\mr{tem}$ covers the $r$-spaces in $G_\mr{tem}$ precisely and our matching constructs $\mc{S}_\mr{approx}=\mc{M}$.
\end{proof}

\section{Extension Counts in Template}\label{sec:extension-template}
We will next require that the template has in some sense, up to some constants and factors of $z=n^2$, ``the correct number'' of $q$-extensions of bounded complexity, including rainbow and colored combinations. However, for technical reasons the color classes will be required to not be ``too large''. Furthermore, we will need to be able to maintain some control over the $\Pi_R$ and $b_R$ for $r$-spaces $R$ that end up embedded, so that we can guarantee configuration compatibility as needed in \cref{prop:final-absorber}. Additionally, we will need to know that every ``new vector'' that is embedded can be put into various prescribed spaces $K_i$ as a necessary precondition to turning our $s$-spaces monochromatic in \cref{prop:final-disjoint-monochromatic}.

We first detail the setup for the precise statement, as it is somewhat involved.
\begin{definition}\label{def:extension-setup}
Given \cref{sub:setup,def:template}, consider the following data.
\begin{itemize}
    \item A $q$-extension $E=(\phi,F,H)$ in $G=\mr{Gr}_q(n,r)$ with $v_E > 0$;
    \item A \emph{$r$-space coloring function} $\psi\colon H\setminus H[F]\to[z]$ such that $|\psi^{-1}(i)|\le\qbinom{s}{r}_q$ for all $i\in[z]$;
    \item Injective \emph{$r$-space configuration functions} $\pi_i\colon\psi^{-1}(i)\to\mr{Red}_q^{r\times s}$ for all $i\in[z]$;
    \item Bases $x_R\in R^r$ for $R\in H\setminus H[F]$;
    \item A basis $(v_1^\ast,\ldots,v_{\dim\mr{Vec}(H)}^\ast)$ of $\mr{Vec}(H)$ so that the last $\dim F$ vectors span $F$;
    \item \emph{Extension basis coloring sets} $\mc{C}_1,\ldots,\mc{C}_{v_E}\subseteq[z]$.
\end{itemize}
We define $\mc{X}(E,\psi,(\pi_i)_{i\in[z]},(x_R)_{R\in H\setminus H[F]},(v_t^\ast)_{t\in[\dim\mr{Vec}(H)]},(\mc{C}_t)_{t\in[v_E]})$ to be the set
\[\bigg\{\phi^\ast\in\mc{X}_E(G_\mr{tem})\colon\genfrac{}{}{0pt}{}{\phi^\ast(R)\in G_{\mr{tem},\psi(R)},~\Pi_{\phi^\ast(R)}=\pi_{\psi(R)}(R),~b_{\phi^\ast(R)}=\phi^\ast(x_R)\text{ for all }R\in H\setminus H[F];}{\phi^\ast(H)\subseteq\bigcap_{i\in\psi(H\setminus H[F])}K_i\text{ and }\phi^\ast(v_t^\ast)\in\bigcap_{i\in\mc{C}_t}K_i\text{ for all }t\in[v_E]}\bigg\}.\]
\end{definition}
\begin{proposition}\label{prop:template-extendable}
Given \cref{sub:setup,def:template}, we have the following as long as $h\ge 1$, $C=C_{\ref{prop:template-extendable}}(h,q,s) > 0$, $c=c_{\ref{prop:template-extendable}}(h,q,s) > 0$, $q^{-cn}\le\tau\le c$, $d\ge d_{\ref{prop:template-extendable}}(r)$, $\ell\ge\ell_{\ref{prop:template-extendable}}(d,s)$, and $n$ is large. Whp over the randomness of the template, for any choice of data as in \cref{def:extension-setup} with $\dim\mr{Vec}(H)\le h$ and $|\mc{C}_t|\le h$ for $t\in[v_E]$ such that (a) $\phi(F)\leqslant\bigcap_{i\in\psi(H\setminus H[F])}K_i$ and (b) $\dim_L\mr{span}_L(\iota_{\psi(R)}^{-1}(\phi(R\cap F)))=\dim_{\mb{F}_q}(R\cap F)$ for all $R\in H\setminus H[F]$, we have
\[\#\mc{X}(E,\psi,(\pi_i)_{i\in[z]},(x_R)_{R\in H\setminus H[F]},(v_t^\ast)_{t\in[\dim\mr{Vec}(H)]},(\mc{C}_t)_{t\in[v_E]})\ge(\tau/z)^Cq^{v_En}.\]
\end{proposition}
\begin{remark}
The condition (a) is a bit stronger than what is actually needed, but something of the sort is required to ensure that we can extend within prescribed template parts correctly. The condition (b) is necessary since every completed $r$-space in the template corresponds to something of maximum $L$-dimension when pulled back to $K$. We will only need the choice of functions $\pi_i$ and bases $x_R$ in \cref{prop:final-disjoint-monochromatic}; in the other application, \cref{lem:master-disjointness}, arbitrary choices are implicitly made (such choices satisfying injectivity of $\pi_i$ do exist since $|\psi^{-1}(i)|\le\qbinom{s}{r}_q$). Furthermore, the sets $\mc{C}_t$ are only needed in the application of \cref{lem:master-disjointness} to \cref{lem:final-color-consistent}, in which we prepare our signed $s$-space decomposition to be turned into a monochromatic decomposition.
\end{remark}

Similarly to the proof of \cref{lem:one-step-typical}, we will prove it for $E$ with $v_E = \dim\mr{Vec}(H)-\dim F=1$ and then prove the result inductively. However, to ensure that we can maintain the condition $\phi(F)\leqslant\bigcap_{i\in\psi(H\setminus H[F])}K_i$ when we iteratively embed the extension, we must prove a slightly stronger statement. The role of the additional element in \cref{lem:iterative-color-embedding}, the global set of colors $\mc{C}$, is to ensure that embedded vectors are good with respect to colors that might not yet have been seen in the partially embedded extension so far. This plays a slightly different role than the $\mc{C}_t$, since we must assume that $\phi(F)$, the base of the extension, is contained within $K_j$ for $j\in\mc{C}$ to successfully embed everything.

\begin{lemma}\label{lem:iterative-color-embedding}
Given \cref{sub:setup,def:template}, we have the following as long as $h\ge 1$, $C=C_{\ref{lem:iterative-color-embedding}}(h,q,s) > 0$, $c=c_{\ref{lem:iterative-color-embedding}}(h,q,s) > 0$, $q^{-cn}\le\tau\le c$, $d\ge d_{\ref{lem:iterative-color-embedding}}(r)$, $\ell\ge\ell_{\ref{lem:iterative-color-embedding}}(d,s)$, and $n$ is large. Whp over the randomness of the template, for any choice of data as in \cref{def:extension-setup} with $v_E=\dim\mr{Vec}(H)-\dim F=1$, $\dim\mr{Vec}(H)\le h$, $|\mc{C}_1|\le h$ and choice of $\mc{C}\subseteq[z]$ of size at most $q^{rh}$ such that (a) $\phi(F)\leqslant\bigcap_{i\in\psi(H\setminus H[F])\cup\mc{C}}K_i$ and (b) $\dim_L\mr{span}_L(\iota_{\psi(R)}^{-1}(\phi(R\cap F)))=\dim_{\mb{F}_q}(R\cap F)$ for all $R\in H\setminus H[F]$, we have
\begin{align*}
\#\bigg\{\phi^\ast\in\mc{X}(E,\psi,(\pi_i)_{i\in[z]},(x_R)_{R\in H\setminus H[F]},(v_t^\ast)_{t\in[\dim\mr{Vec}(H)]},(\mc{C}_1))\colon\phi^\ast(v_1^\ast)\in\bigcap_{i\in\mc{C}}K_i\bigg\}\ge(\tau/z)^Cq^n.
\end{align*}
\end{lemma}
\begin{remark}
Since $v_E=1$, $v_1^\ast$ is the unique element of the chosen basis of $\mr{Vec}(H)$ that is outside of $F$. Also, compared to \cref{prop:template-extendable} the condition (a) has the addition of the color set $\mc{C}$.
\end{remark}
\begin{proof}
As in \cref{lem:template-basic}, we may assume the template is well-defined as long as $\ell$ is large with respect to $d,s$, and we may assume that $d$ is large enough to ensure that the statements in \cref{lem:generic} are satisfied. Of key importance is that for all $\Pi\in\mb{F}_q^{r\times s}$ of rank $r$, $\Pi N\in L^{r\times r}$ and all its square submatrices are nonsingular; the other conditions will naturally arise when ruling out certain degeneracies.

As seen in the proof of \cref{lem:one-step-typical}, our extension can be represented by picking $a\le q^{rh}$ and considering $Q_1,\ldots,Q_a\leqslant F$ and $v_1,\ldots,v_a\in F$ such that if $i\neq j$ and $Q_i=Q_j$ then $\mr{span}_{\mb{F}_q}(Q_i\cup\{v_i\})\neq\mr{span}_{\mb{F}_q}(Q_j\cup\{v_j\})$. Then $E=(\phi,F,H)$ where $H$ is on the vector space $F\oplus\mb{F}_q$ and the $r$-spaces of $H\setminus H[F]$ are $R_i=\mr{span}_{\mb{F}_q}(Q_i\cup\{(v_i,1)\})$ where we abusively extend $Q_i$ by zeros in the obvious way here. Note that we can set things up so that furthermore $v_1^\ast=(0,1)$.

Thus $\mc{X}_E(G_\mr{tem})$ contains embeddings $\phi^\ast$ which agree with $\phi$ on $F$ and where, writing $v=\phi^\ast((0,1))=\phi^\ast(v_1^\ast)$, we have $v\notin\phi(F)$ and $\mr{span}_{\mb{F}_q}(\phi(Q_i)\cup\{v+\phi(v_i)\})\in G_\mr{tem}$ for all $i\in[a]$. To prove the desired statement, we are further given colors $c_i:=\psi(R_i)\in[z]$ for $i\in[a]$ (which do not contain too many repetitions by the given conditions from \cref{def:extension-setup}) and $\mc{C}\subseteq[z]$ of size at most $q^{rh}$ and $\mc{C}_1\subseteq[z]$ of size at most $h$, and we wish to understand $X$, the number of such extensions where in fact $\mr{span}_{\mb{F}_q}(\phi(Q_i)\cup\{v+\phi(v_i)\})\in G_{\mr{tem},c_i}$ for all $i\in[a]$ and also $v\in\bigcap_{i\in\mc{C}\cup\mc{C}_1}K_i$, and furthermore $\Pi_{\phi^\ast(R_i)}=\pi_{c_i}(R_i)$ and $b_{\phi^\ast(R_i)}=\phi^\ast(x_{R_i})$ for $i\in[a]$. It suffices to show that whp, for all appropriate data the associated variable $X$ is always sufficiently large: combining the second-to-last condition just listed with the assumed event $\phi(F)\leqslant\bigcap_{j\in\mc{C}}K_j$ gives the necessary property $\phi^\ast(H)\subseteq\bigcap_{i\in\mc{C}}K_i$ for embeddings $\phi^\ast$ that are counted for extensions and colorings that we care about. Fix a basis $\{w_{i,1},\ldots,w_{i,r-1}\}$ for $\phi(Q_i)$ for all $i\in[a]$.

Note that $X=\sum_{v\in\mb{F}_q^n\setminus\phi(F)}X_v$ where $X_v\in\{0,1\}$, dependent on the randomness of the template, is the random variable which is $1$ precisely when setting $\phi^\ast((0,1))=v$ satisfies the above conditions given $\phi$ and $Q_i,v_i,c_i,x_{R_i}$ for $i\in[a]$ and the $\pi_j$ for $j\in[z]$. Given $v\in\mb{F}_q^n\setminus\phi(F)$, let us consider for $i\in[a]$ the basis $(v+\phi(v_i),w_{i,1},\ldots,w_{i,r-1})$ for $\mr{span}_{\mb{F}_q}(\phi(Q_i)\cup\{v+\phi(v_i)\})$ and arrange it as a size $r$ column vector $b_{v,i}\in(\mb{F}_q^n)^r$. (This could be different from the sampled value $b_R$ for this $r$-space $R$.) Write $\phi^\ast(x_{R_i})=\Gamma_ib_{v,i}$, where $\Gamma_i\in\mr{GL}(\mb{F}_q^r)$ is fixed and does not depend on $v$; such a choice is easily seen to exist.

Now for each $\Pi\in\mr{Red}_q^{r\times s}$ consider $x_{v,i,\Pi}\in K^r$, if it exists, such that
\[\iota_{c_i}(\Pi Nx_{v,i,\Pi})=\Gamma_ib_{v,i}=\phi^\ast(x_{R_i}).\]
This is unique when it exists, namely when $\{v+\phi(v_i),w_{i,1},\ldots,w_{i,r-1}\}\subseteq K_{c_i}$, in which case $x_{v,i,\Pi}=(\Pi N)^{-1}\iota_{c_i}^{-1}(\Gamma_ib_{v,i})$. (Recall $\Pi N$ is invertible.) If $x_{v,i,\Pi}\in K^r$ spans $r$ dimensions over $L$ then for any further $\Pi'\in\mr{Red}_q^{r\times s}$ we can then consider $\mr{span}_{\mb{F}_q}(\iota_{c_i}(\Pi'Nx_{v,i,\Pi}))\leqslant K_{c_i}$: this corresponds to an $r$-space contained within the same $s$-space which is potentially contributing to the template via $x_{v,i,\Pi}$ and $\iota_i$ (it contributes when the random event $\mr{span}_{\mb{F}_q}(\iota_{c_i}(Nx_{v,i,\Pi}))\in\mc{S}_{\mr{tem},c_i}$ holds).

Fix $\Pi_i^\ast=\pi_{c_i}(R_i)\in\mr{Red}_q^{r\times s}$ for $i\in[a]$ and note that if $i\neq j$ and $c_i=c_j$ then $\Pi_i^\ast\neq\Pi_j^\ast$ (from injectivity of the $\pi$ functions in \cref{def:extension-setup}). Given $\iota_1,\ldots,\iota_z$, we call $v$ \emph{nonoverlapping} if (a) $v\in\bigcap_{j\in\mc{C}\cup\mc{C}_1}K_j$, (b) for all $i\in[a]$, the vector $x_{v,i,\Pi_i^\ast}=(\Pi_i^\ast N)^{-1}\iota_{c_i}^{-1}(\Gamma_ib_{v,i})\in K^r$ exists and is such that $\dim_L\mr{span}_L(x)=r$, and (c) for all choices of $\Pi_i\in\mr{Red}_q^{r\times s}$ we have that $\iota_{c_i}(\Pi_iNx_{v,i,\Pi_i^\ast})$ span distinct $r$-spaces over $\mb{F}_q$ as we vary $i\in[a]$. The idea is that this provides situations where we can guarantee $X_v=1$ occurs with nonnegligible probability. Here (c) is the key property which will show that nonoverlapping vectors are likely to contribute to $X$, whereas the other conditions are more akin to feasibility conditions.

Consider any nonoverlapping $v$ (which is a function only of $\iota_1,\ldots,\iota_z$) and take $x_{v,j,\Pi_j^\ast}$ as in the definition of nonoverlapping. Consider the event that for each $j\in[a]$ and $\Pi\in\mr{Red}_q^{r\times s}$ the $r$-space $R=\mr{span}_{\mb{F}_q}(\iota_{c_j}(\Pi Nx_{v,j,\Pi_j^\ast}))$ satisfies $i_R=c_j$, $b_R=\iota_{c_j}(\Pi Nx_{v,j,\Pi_j^\ast})$, $y_R=1$, and $\Pi_R = \Pi$. In the case $\Pi=\Pi_j^\ast$ the second equation simplifies to $b_{\phi^\ast(R_j)}=b_R=\Gamma_jb_{v,j}=\phi^\ast(x_{R_j})$ and the fourth gives $\Pi_{\phi^\ast(R_j)}=\Pi_j^\ast=\pi_{\psi(R_j)}(R_j)$. Additionally, such an event would guarantee by definition that $\mr{span}_{\mb{F}_q}(\iota_{c_j}(Nx_{v,j,\Pi_j^\ast}))\in\mc{S}_{\mr{tem},c_j}$ for all $j\in[a]$. Combining these properties, as a consequence we would have $X_v=1$. Furthermore, since the nonoverlapping condition guarantees that all these $r$-spaces $R$ are distinct as $j,\Pi$ vary, this can be simultaneously accomplished and occurs with probability at least $\lambda^{\qbinom{s}{r}_qa}$ where $\lambda=((q^r-1)\cdots(q^r-q^{r-1})\qbinom{s}{r}_q)^{-1}\tau/z$, which implies $\mb{E}X_v\ge(\tau/z)^A$ for some appropriately chosen $A\ge 1$ depending only on $h,q,s$ due to the bound on $a$ (the value $A=2q^{r(h+s)}$ certainly suffices).

Now let $Y$ be the number of $v\in\mb{F}_q^n\setminus\phi(F)$ which are nonoverlapping, which is a random function of $\iota_1,\ldots,\iota_z$ only. We see
\[\mb{E}[X|\iota_1,\ldots,\iota_z]\ge(\tau/z)^AY,\]
and additionally note that the remaining randomness is purely over the independent choices of $i_R,b_R,y_R,\Pi_R$ for all $R\in\mr{Gr}_q(n,r)$. Furthermore, it is easily seen that (if we condition on $\iota_1,\ldots,\iota_z$) each of these random variables can influence the identity of $X_v$ for at most $O_{q,s}(z)$ values $v$ since every $R$-space can only be included in $z\cdot O_{q,s}(1)$ ways into some potential $s$-space which in total accounts for at most $O_{q,s}(z)$ other $r$-spaces that can be affected by the choice of defining variables for $R$. Therefore, \cref{lem:azuma} shows that conditional on $\iota_1,\ldots,\iota_z$ the variable $X$ concentrates in a window of size $q^{2n/3}$ with probability at least $1-\exp(-\Omega(q^{-n/4}))$, which is sufficient to take a union bound over all possible $E$ and colors involved in $\psi$ and $\mc{C},\mc{C}_1$ (of which there are at most $q^{O_{h,q,s}(n)}$ total choices).

It remains to show that whp over the randomness of $\iota_1,\ldots,\iota_z$, for all extensions $E$, functions $\psi$, color sets $\mc{C},\mc{C}_1$ as above, and the remaining choice of auxiliary information, the corresponding variable $Y$ (whose definition depends on $E,\psi,\mc{C},\mc{C}_1$ and the $\pi_i,\Gamma_j$ and whose randomness only uses $\iota_1,\ldots,\iota_z$) satisfies $Y\ge q^n/z$ whenever \ref{lem:iterative-color-embedding}(a,b) hold. To do this, we will show that for fixed choices as above the probability that $Y < q^n/z$ and $\phi(F)\leqslant\bigcap_{i\in[a]}K_{c_i}$ and $\dim_L\mr{span}_L(\iota_{\psi(R)}^{-1}(\phi(R\cap F)))=\dim_{\mb{F}_q}(R\cap F)$ for all $R\in H\setminus H[F]$ simultaneously occur is at most $q^{-n\log n}$ and then take a union bound.

Let $Y_0$ be the number of $v\in\mb{F}_q^n\setminus\phi(F)$ such that $x_{v,i,\Pi_i^\ast}=(\Pi_i^\ast N)^{-1}\iota_{c_i}^{-1}(\Gamma_ib_{v,i})\in K^r$ exists for all $i\in[a]$ and such that $v\in\bigcap_{j\in\mc{C}\cup\mc{C}_1}K_j$. For $i\in[a]$ let $Y_i$ be the number of $v$ for which the nonoverlapping condition is failed in the following way: the second part of condition (b), namely the $r$-dimensionality condition, fails for index $i$ (but the first part is valid). Similarly, for $1\le i<j\le a$ let $Y_{ij}$ be the number of $v$ where condition (c) fails for indices $i,j$ and some $\Pi_i,\Pi_j\in\mr{Red}_q^{r\times s}$. Note that
\[Y\ge Y_0-\sum_{i=1}^aY_i-\sum_{1\le i<j\le a}Y_{ij}.\]
Since $K_i\leqslant V=\mb{F}_q^n$ has codimension at most $\ell$, we see that $Y_0\ge q^{n-(a+q^{rh}+h)\ell}-q^{\dim F}$ is true because $\phi(F)\leqslant\bigcap_{i\in[a]}K_{c_i}$ (since requiring $v\in\iota_{c_i}(K)$ for all $i\in[a]$ and $v\in\iota_j(K)$ for all $j\in\mc{C}\cup\mc{C}_1$ is therefore enough).

Additionally, we claim that $Y_i=O_{\ell,q,r}(1)$ no matter the choice of $\iota_1,\ldots,\iota_z$. Indeed, if $v$ is counted in $Y_i$, then $(\Pi_i^\ast N)^{-1}\iota_{c_i}^{-1}(\Gamma_ib_{v,i})$ exists and spans less than $r$ dimensions over $L$, which means there is nonzero $y\in L^{1\times r}$ so that $y(\Pi_i^\ast N)^{-1}\iota_{c_i}^{-1}(\Gamma_ib_{v,i})=0$. If the first coordinate of $y(\Pi_i^\ast N)^{-1}\Gamma_i$ is nonzero, then given $\iota_{c_i},Q_i,v_i$ there is a fixed value of $\iota_{c_i}^{-1}(v)$ and hence at most one value of $v$. If the first coordinate is zero, then we see that $\dim_L\mr{span}_L(\iota_{\psi(R)}^{-1}(\phi(R\cap F)))=\dim_{\mb{F}_q}(R\cap F)$ fails for the edge $R=R_i\in H\setminus H[F]$ corresponding to this $i\in[a]$, since then we have a nontrivial $L$-linear relation among the embedding of $R_i$ (in $K\simeq K_{c_i}$) not involving $\iota_{c_i}^{-1}(v)$ (which violates \ref{lem:iterative-color-embedding}(b)). We ultimately deduce $Y_i\le q^{\ell r}$ due to the possible choices for $y$.

Combining the lower bound on $Y_0$ and the upper bound on $Y_i$, we see that if $Y < q^n/z$ and \ref{lem:iterative-color-embedding}(a,b) hold then we have $Y_{ij}\ge q^n/z$ for some $1\le i<j\le a$. Let us first consider the possibility that this occurs for some $i,j$ with $c_i=c_j=c^\ast$. We deduce that for some $\Pi_i,\Pi_j\in\mb{F}_q^{r\times s}$ of rank $r$, we have
\[\iota_{c^\ast}(\Pi_iN(\Pi_i^\ast N)^{-1}\iota_{c^\ast}^{-1}(\Gamma_ib_{v,i}))=\iota_{c^\ast}(\Pi_jN(\Pi_j^\ast N)^{-1}\iota_{c^\ast}^{-1}(\Gamma_jb_{v,j}))\]
and both sides are well-defined (since certain corresponding bases over $\mb{F}_q$ span the same $r$-space). This implies $(\Pi_iN)(\Pi_i^\ast N)^{-1}\iota_{c^\ast}^{-1}(\Gamma_ib_{v,i})=(\Pi_jN)(\Pi_j^\ast N)^{-1}\iota_{c^\ast}^{-1}(\Gamma_jb_{v,j})$.

First consider the case that simultaneously there are $M_i,M_j\in\mr{GL}(\mb{F}_q^r)$ with $\Pi_i=M_i\Pi_i^\ast$ and $\Pi_j=M_j\Pi_j^\ast$. Then the above becomes $M_i\iota_{c^\ast}^{-1}(\Gamma_ib_{v,i})=M_j\iota_{c^\ast}^{-1}(\Gamma_jb_{v,j})$ so $M_i\Gamma_ib_{v,i}=M_j\Gamma_jb_{v,j}$. But multiplying by one of these invertible matrices will preserve the span of the $r$ coordinates, so we see that $\mr{span}_{\mb{F}_q}(b_{v,i})=\mr{span}_{\mb{F}_q}(b_{v,j})$. However, the given conditions for the form of the extension imply either $Q_i\neq Q_j$ or $Q_i=Q_j$ and $\mr{span}_{\mb{F}_q}(Q_i\cup\{v_i\})\neq\mr{span}_{\mb{F}_q}(Q_j\cup\{v_j\})$, which ensures that for $v\in\mb{F}_q^n\setminus\phi(F)$ these two spanned spaces are distinct. Thus this cannot happen.

Furthermore, the matrices $\Pi_i^\ast,\Pi_j^\ast\in\mr{Red}_q^{r\times s}$ are distinct since $c_i=c_j$ and $i<j$. By the discussion following \cref{def:rref} they are not related via left-multiplication of an invertible $r\times r$ matrix. At this stage, we use the third bullet point from \cref{lem:generic} to deduce that the vector $(\Pi_iN)(\Pi_i^\ast N)^{-1}\Gamma_ie_1-(\Pi_jN)(\Pi_j^\ast N)^{-1}\Gamma_je_1$ is nonzero. But this means that when we consider the equation 
\[(\Pi_iN)(\Pi_i^\ast N)^{-1}\Gamma_i\iota_{c^\ast}^{-1}(b_{v,i})-(\Pi_jN)(\Pi_j^\ast N)^{-1}\Gamma_j\iota_{c^\ast}^{-1}(b_{v,j})=0\]
and isolate the parts containing $\iota_{c^\ast}^{-1}(v)$, one of the instances on the left has a nonzero coefficient. This means that given any $\iota_1,\ldots,\iota_z$, there is at most $1$ choice for $v$. Therefore we have ultimately shown that $Y_{ij}=O_{q,s}(1)$ if $c_i=c_j$.

Finally, we consider $Y_{ij}$ for some $i,j$ with $c_i\neq c_j$. Here we will use the randomness over $\iota_{c_i},\iota_{c_j}$. As before, the vector $v\in\mb{F}_q^n\setminus\phi(F)$ contributes to $Y_{ij}$ if there are $\Pi_i,\Pi_j\in\mb{F}_q^{r\times s}$ of rank $r$ for which
\begin{equation}\label{eq:rainbow-r-space-overlap}
\iota_{c_i}(\Pi_iN(\Pi_i^\ast N)^{-1}\iota_{c_i}^{-1}(\Gamma_ib_{v,i}))=\iota_{c_j}(\Pi_jN(\Pi_j^\ast N)^{-1}\iota_{c_j}^{-1}(\Gamma_jb_{v,j}))
\end{equation}
and both sides are well-defined. Let $\mc{I}_i$ be the set of indices of rows of $\Pi_i$ that are in $\mr{row}_{\mb{F}_q}(\Pi_i^\ast)$ (the row space), and similar for $\mc{I}_j$. First suppose that \cref{eq:rainbow-r-space-overlap} occurs in a case where either $\mc{I}_i=\mc{I}_j=[r]$ or $\mc{I}_i\neq\mc{I}_j$. If the former holds, then $\Pi_i=M_i\Pi_i^\ast$ and $\Pi_j=M_j\Pi_j^\ast$ for $M_i,M_j\in\mr{GL}(\mb{F}_q^r)$ and similar to earlier we find $M_i\Gamma_ib_{v,i}=M_j\Gamma_jb_{v,j}$ which cannot happen. If the latter holds then we can find an element in one set but not the other. Without loss of generality let $i^\ast\in\mc{I}_j\setminus\mc{I}_i$. Now inspect the $i^\ast$th row element of the above (vector) equality. Let $y_i,y_j$ be the $i^\ast$th rows of $\Pi_i,\Pi_j$ respectively and write $y_j=y'\Pi_j^\ast$ for $y'\in\mb{F}_q^{1\times r}$. Then \cref{eq:rainbow-r-space-overlap} implies
\[\iota_{c_i}(y_iN(\Pi_i^\ast N)^{-1}\iota_{c_i}^{-1}(\Gamma_ib_{v,i}))=y'\Gamma_jb_{v,j},\]
so $y_iN(\Pi_i^\ast N)^{-1}\Gamma_i\iota_{c_i}^{-1}(b_{v,i})=y'\Gamma_j\iota_{c_i}^{-1}(b_{v,j})$. By the second bullet of \cref{lem:generic} and $y_i\notin\mr{row}_{\mb{F}_q}(\Pi_i^\ast)$ (from $i^\ast\notin\mc{I}_i$) we have that $y_iN(\Pi_i^\ast N)^{-1}\Gamma_ie_1\in L\setminus\mb{F}_q$, so $y'\in\mb{F}_q^{1\times r}$ means that this yields a nontrivial relation for $\iota_{c_i}^{-1}(v)$. In particular, given any choice of $\iota_1,\ldots,\iota_z$ we see that there is at most $1$ choice for $v$ in this situation. Therefore in total these cases for $\Pi_i,\Pi_j$ contribute $O_{q,s}(1)$ to $Y_{ij}$.

Now we can fix $i,j,\Pi_i,\Pi_j$ such that the corresponding sets $\mc{I}_i,\mc{I}_j$ satisfy $\mc{I}:=\mc{I}_i=\mc{I}_j\neq[r]$. Consider the random variable $Z$, the number of $v\in\mb{F}_q^n\setminus\phi(F)$ satisfying \cref{eq:rainbow-r-space-overlap} (and let $\mc{Z}$ be the set of such $v$). It suffices to show \begin{equation}\label{eq:large-degeneracy-deviation}
\mb{P}[Z\ge q^n/z^2]\le q^{-2n\log n},
\end{equation}
since we can then take a union bound over $i,j,\Pi_i,\Pi_j$, add in the extra contribution of $O_{q,s}(1)$ from the other choices of $\Pi_i,\Pi_j$, and deduce the necessary bounds on $Y_{ij}$ with very high probability, which then imply the desired lower bound for $Y$ and thus $X$ whp (enough to take a union bound over $E,\psi,\mc{C},\mc{C}_1$ and the $\pi,x$ values at the end, as discussed earlier).

We show \cref{eq:large-degeneracy-deviation} via the method of moments. Technically, we cannot directly bound the moment of $Z$, and instead will count certain special tuples of elements in $\mc{Z}$, which bears similarity to e.g.~the deletion method of R\"odl and Ruci\'nski (see \cite{JR04,RR95}) used for upper tails of subgraph counts. Here, though, we will impose certain linear-algebraic conditions rather than subgraph disjointness. Choose some $i^\ast\in[r]\setminus\mc{I}$, which exists by the given conditions, and let $y_i,y_j$ be the $i^\ast$th rows of $\Pi_i,\Pi_j$, respectively. We have for $\Pi_i'=\Gamma_i^{-1}\Pi_i^\ast$ and $\Pi_j'=\Gamma_j^{-1}\Pi_j^\ast$ (which have the same respective $\mb{F}_q$-row spaces) that
\[\iota_{c_i}(y_iN(\Pi_i'N)^{-1}\iota_{c_i}^{-1}(b_{v,i}))=\iota_{c_j}(y_jN(\Pi_j'N)^{-1}\iota_{c_j}^{-1}(b_{v,j}))\]
for $v\in\mc{Z}$. Let $f_1,\ldots,f_{\dim F}$ be a $\mb{F}_q$-basis for $\phi(F)\leqslant\mb{F}_q^n$.

For $1\le k\le n/8$ let $\mc{Z}_k$ be the (random) set of tuples $(w_1,\ldots,w_k)\in\mc{Z}^k$ such that $(f_{t,i}':=\iota_{c_i}^{-1}(f_t))_{t\in[\dim F]}$, $(w_{t,i}':=\iota_{c_i}^{-1}(w_t))_{t\in[k]}$, and $(y_{t,i}':=y_iN(\Pi_i' N)^{-1}\iota_{c_i}^{-1}(b_{w_t,i}))_{t\in[k]}$ are well-defined and jointly $\mb{F}_q$-linearly independent (note they are all elements of $K$) and such that similar holds with $i$ replaced by $j$. There are at most $(|K|^{\dim F+k})^2$ possible choices of these values $f_{t,i}',f_{t,j}',w_{t,i}',w_{t,j}',y_{t,i}',y_{t,j}'$ (note that $f_{t,i}'$ for $t\in[\dim F]$ and $w_{t',i}$ determine $y_{t',i}'$ for any $t'\in[k]$). Furthermore, given such fixed choices, let us consider the event that (a) $\iota_{c_i},\iota_{c_j}$ actually map the vectors in this way, and (b) $\iota_{c_i},\iota_{c_j}$ make it so that $w_1,\ldots,w_k\in\mc{Z}$ indeed holds. This is equivalent to
\begin{align*}
\iota_{c_i}(f_{t,i}')&=f_t=\iota_{c_j}(f_{t,j}')\text{ for all }t\in[\dim F],\\
\iota_{c_i}(w_{t,i}')&=w_t=\iota_{c_j}(w_{t,j}')\text{ for all }t\in[k],\\
\iota_{c_i}(y_{t,i}')&=\iota_{c_j}(y_{t,j}')\text{ for all }t\in[k].
\end{align*}
The probability of this is at most $((1-1/n^2)q^n)^{-2\dim F-3k}$: here we are using $\mb{F}_q$-linear independence of the $f_{t,i}',w_{t,i}'$ and similar for $j$, and also using that each next equation has probability at most $1/(q^n-q^e)\le((1-1/n^2)q^n)^{-1}$ of holding if $e\le 2\dim F+3k$ equations have been processed so far since $\iota_{c_i},\iota_{c_j}$ are uniformly random injective linear maps $K\hookrightarrow\mb{F}_q^n$. Therefore a union bound over possible choices $(w_1,\ldots,w_k)\in(\mb{F}_q^n)^k$ and of choices $f_{t,i}',f_{t,j}',w_{t,i}',w_{t,j}',y_{t,i}',y_{t,j}'$ yields
\[\mb{E}|\mc{Z}_k|\le(q^n)^k|K|^{2\dim F+2k}((1-1/n^2)q^n)^{-2\dim F-3k}\le 2.\]

Now we provide a lower bound for $|\mc{Z}_k|$ in terms of $|\mc{Z}|$. We choose $w_t\in\mc{Z}$ for $1\le t\le k$ in order, so that each new $w_t$ satisfies the following: $w_{t,i}'$ is not in the $L$-span of $(f_{t',i}')_{t'\in[\dim F]}$ and $(w_{t',i}')_{t'\in[t-1]}$ together, and the same for $j$. There are at least $\max(0,Z-2q^{\ell(\dim F+t-1)})$ such choices since there are at most $(q^\ell)^{\dim F+t-1}$ elements in this $L$-span for $i$, and each of these potential values of $w_{t,i}'=\iota_{c_i}^{-1}(w_t)$ can correspond to at most $1$ value of $w_t=v\in\mc{Z}$ to rule out (and the same for $j$). Call the collection of such tuples $\mc{G}_k$. We see from this analysis that if $Z\ge 2q^{\ell(\dim F+k)}$ then $|\mc{G}_k|\ge(Z/2)^k$.

Finally, we claim that $\mc{G}_k\subseteq\mc{Z}_k$. Given this, we deduce
\[\mb{E}(Z/2)^k\mbm{1}_{Z\ge 2q^{\ell(\dim F+k)}}\le 2\]
and Markov's inequality therefore shows $\mb{P}[Z\ge q^n/z^2]\le 2(2z^2/q^n)^k$ for any integer $1\le k\le n/(2\ell)$. Thus \cref{eq:large-degeneracy-deviation} follows, so as discussed earlier the argument is finished.

Suppose the claim is false, i.e., $\mc{G}_k\not\subseteq\mc{Z}_k$, so that there is $(w_1,\ldots,w_k)\in\mc{Z}^k$ constructed iteratively as above such that $(f_{t,i}')_{t\in[\dim F]}$, $(w_{t,i}')_{t\in[k]}$, and $(y_{t,i}')_{t\in[k]}$ are $\mb{F}_q$-linearly dependent:
\[\sum_{t=1}^{\dim F}\alpha_tf_{t,i}'+\sum_{t=1}^k(\beta_tw_{t,i}'+\gamma_ty_{t,i}')=0\]
for $\alpha_t,\beta_t,\gamma_t\in\mb{F}_q$ (or similar for $j$, but the argument is symmetric in that case). Note that $y_{t,i}'=y_iN(\Pi_i' N)^{-1}\iota_{c_i}^{-1}(b_{w_t,i})$ is an $L$-linear combination of $w_{t,i}'=\iota_{c_i}^{-1}(w_t)$ and $(f_{t,i}'=\iota_{c_i}^{-1}(f_t))_{t\in[\dim F]}$ (recalling that the ``other parts'' of the basis $b_{v,i}$ come from $\phi(F)$). That is, we can write
\[y_{t,i}'=\zeta_tw_{t,i}'+\sum_{t'=1}^{\dim F}\zeta_{t,t'}f_{t',i}'\]
for $\zeta_t,\zeta_{t,t'}\in L$. Plugging in, we find
\[\sum_{t=1}^{\dim F}\bigg(\alpha_t+\sum_{t'=1}^k\gamma_{t'}\zeta_{t',t}\bigg)f_{t,i}'+\sum_{t=1}^k(\beta_t+\gamma_t\zeta_t)w_{t,i}'=0.\]

Furthermore, the condition $i^\ast\in[r]\setminus\mc{I}$ implies that $y_i\notin\mr{span}_{\mb{F}_q}(\Pi_i^\ast)=\mr{span}_{\mb{F}_q}(\Pi_i')$, so that $y_iN(\Pi_i' N)^{-1}\in L^{1\times r}$ has all coordinates in $L\setminus\mb{F}_q$ by the second bullet of \cref{lem:generic}. This implies $\zeta_t\in L\setminus\mb{F}_q$. Now first suppose that not all the $\beta,\gamma$ values are $0$. Let $t^\ast$ be the largest index with $(\beta_{t^\ast},\gamma_{t^\ast})\in\mb{F}_q^2\setminus\{(0,0)\}$. Since $\zeta_{t^\ast}\in L\setminus\mb{F}_q$, this means $\beta_{t^\ast}+\gamma_{t^\ast}\zeta_{t^\ast}\in L^\times$. Dividing out by this value, we easily see that $w_{t^\ast,i}'$ is in the $L$-span of $(f_{t,i}')_{t\in[\dim F]}$ and $(w_{t,i}')_{t\in[t^\ast-1]}$, which contradicts the definition of $\mc{G}_k$. Therefore we must have $\beta_t=\gamma_t=0$ for all $t\in[k]$. We deduce
\[0=\sum_{t=1}^{\dim F}\alpha_tf_{t,i}'=\iota_{c_i}^{-1}\bigg(\sum_{t=1}^{\dim F}\alpha_tf_t\bigg).\]
This is a contradiction since $(f_t)_{t\in[\dim F]}$ is $\mb{F}_q$-linearly independent by definition. We are done.
\end{proof}

We now briefly deduce \cref{prop:template-extendable}.
\begin{proof}[Proof of \cref{prop:template-extendable}]
Let $\{v_1^\ast,\ldots,v_{\dim\mr{Vec}(H)}^\ast\}$ be the given basis for $H$ whose last $\dim F$ values span $F$. Consider $W_t=\mr{span}_{\mb{F}_q}(F\cup\{v_1^\ast,\ldots,v_t^\ast\})$ for $t\in[v_E]$. We will iteratively embed $v_t^\ast$. To this end, let $X_t$ for $0\le t\le v_E$ be the number of injective linear embeddings $\phi^\ast$ of $W_t$ into $\mb{F}_q^n$ which agree with $\phi$ on $F$ such that (a) the image is within $\bigcap_{i\in\psi(H\setminus H[F])}K_i$, (b) for each $R\in(H\setminus H[F])\cap\mr{Gr}(W_t,r)$ we have $\phi^\ast(R)\in G_{\mr{tem},\psi(R)}$, (c) $\dim_L\mr{span}_L(\iota_{\psi(R)}^{-1}(\phi^\ast(R\cap W_t)))=\dim_{\mb{F}_q}(R\cap W_t)$ for all $R\in H\setminus H[F]$, (d) $\Pi_R=\pi_{\psi(R)}(R)$ and $b_{\phi^\ast(R)}=\phi^\ast(x_R)$ for all $R\in(H\setminus H[F])\cap\mr{Gr}(W_t,r)$, and (e) for all $t'\in[t]$, $\phi^\ast(v_{t'}^\ast)\in\bigcap_{i\in\mc{C}_{t'}}K_i$. For $i=0$ we have $X_0=1$ by the given conditions.

Given some $t\ge 1$, we apply \cref{lem:iterative-color-embedding} with $\mc{C} = \psi(H\setminus H[F])$ (clearly $|\mc{C}|\le|H|\le q^{rh}$) and $\mc{C}_1$ replaced by $\mc{C}_t$ to show that
\[X_t\ge(\tau/z)^{C'}q^nX_{t-1}-q^{rh}\cdot(q^\ell)^rX_{t-1}\ge(\tau/z)^{C'+1}q^nX_{t-1}\]
where $C'=C_{\ref{lem:iterative-color-embedding}}(h,q,s)$. The subtracted term comes from guaranteeing (c) holds: for each of at most $q^{rh}$ total $R\in H\setminus H[F]$, there are at most $(q^\ell)^r$ choices of embedding for $v_t^\ast$ that would cause an unexpected $L$-linear dependence, which is the only way to ensure that $\dim_L\mr{span}_L(\iota_{\psi(R)}^{-1}(\phi^\ast(R\cap W_{t'})))$ does not grow by $1$ between $t'=t-1$ and $t'=t$ in cases where $R\cap W_t\neq R\cap W_{t-1}$. The second inequality is true as long as $c=c_{\ref{prop:template-extendable}}(h,q,s)$ is chosen small enough. The result follows by taking $C=h(C'+1)$.
\end{proof}

\section{Covering the remainder and absorbable decomposition}\label{sec:cover}
In this section we provide tools to go from the approximate decomposition provided by \cref{prop:approximate-covering} to something of the form taken in by \cref{prop:final-absorber}. To do this, we first provide a general lemma which will be used multiple times to cover collections of $r$-spaces and convert given signed $s$-space decompositions into ones with better properties, including certain disjointness conditions. The approach here uses a ``disjoint random process'', as opposed to the application of the Lov\'asz Local Lemma in the proof of \cref{prop:final-absorber}, since we need to provide intermediate guarantees such as boundedness and field boundedness.

The precise details are quite technical, but at a high level we have some collection of inputs which is bounded (\cref{def:bounded}) and whose extension to $L$ is bounded (similar to field boundedness) in an appropriate sense. We seek to process them one at a time, flipping them randomly to some extension into $G_\mr{tem}$ (which we additionally enforce is rainbow, for later application) which is in some sense disjoint from what has happened so far. We wish to show this process will run to completion and that it produces outcomes which are not too concentrated anywhere (which allows extraction of various boundedness and field boundedness conditions when applied).

For ease of definition of the process, we make some definitions that will allow us to talk about the extension of our $r$-spaces to $L$ in a convenient way. We also introduce the notion of field disjointness.
\begin{definition}\label{def:field-disjoint}
Given the setup of \cref{sub:setup,def:template}, for $R\in\mr{Gr}_q(n,r)$ let $\mr{ind}(R)\in[z]\cup\{\ast\}$ be either the unique index $i\in[z]$ with $R\in G_{\mr{tem},i}$ or $\ast$ if $R\notin G_\mr{tem}$. Let $\chi(R)=\{R'\in G_{\mr{tem},i}\colon\mr{span}_L(\iota_i^{-1}(R'))\leqslant\mr{span}_L(\iota_i^{-1}(R))\}$ if $\mr{ind}(R)=i$ otherwise let $\chi(R)=\{R\}$ if $\mr{ind}(R)=\ast$. We extend $\chi(\mc{R})=\bigcup_{R\in\mc{R}}\chi(R)$ to sets in the obvious way, without multiplicity. Finally, a collection $\mc{R}\subseteq\mr{Gr}_q(n,r)$ is \emph{field disjoint with respect to the template} if there are not distinct $R_1,R_2\in\mc{R}$ with $R_1\in\chi(R_2)$. Similarly $\Phi\in\mb{Z}^{\mr{Gr}_q(n,r)}$ is \emph{field disjoint with respect to the template} if $\mr{supp}(\Phi)$ is.
\end{definition}
We briefly observe that field boundedness implies boundedness of $\chi$.
\begin{lemma}\label{lem:chi-bounded}
Given the setup of \cref{def:field-disjoint}, if $\mc{R}\subseteq G_\mr{tem}$ is $(\theta,L)$-field bounded, then $\chi(R)$ is $z\theta$-bounded.
\end{lemma}
\begin{proof}
Given $Q\in\mr{Gr}_q(n,r-1)$, we wish to count spaces of $\chi(R)$ containing it. There are $z$ choices for possible template index $i$ to use. If $\mr{span}_L(\iota_i^{-1}(Q))<r-1$ then any $r$-space $R\in\mr{Gr}_q(n,r)$ containing $Q$ will not span $r$ dimensions when pulled back via $\iota_i^{-1}$ and extended to $L$, and hence cannot be in $G_{\mr{tem},i}$. Therefore we may focus on cases where $\mr{span}_L(\iota_i^{-1}(Q))=r-1$, in which case we can then apply the field boundedness condition (\cref{def:field-bounded}) with $Q^\ast=\mr{span}_L(Q)$. We are using that $Q\leqslant R'$ for some $R'\in\chi(R)$ and $R\in G_{\mr{tem},i}$ implies $\mr{span}_L(\iota_i^{-1}(Q))\leqslant\mr{span}_L(\iota_i^{-1}(R))$.
\end{proof}

Now we prove the master disjoint process lemma. We first detail the setup for the precise statement, as it is somewhat involved.

\begin{definition}\label{def:disjointness-setup}
Given \cref{sub:setup,def:template}, consider a parameter $\theta$ and the following data.
\begin{itemize}
    \item An \emph{extension type}: $r$-dimensional $q$-system $H$ and $F\leqslant\mr{Vec}(H)$ so that $H\setminus H[F]$ is nonempty and $v=\dim\mr{Vec}(H)-\dim F$;
    \item An \emph{extension core}: $r$-dimensional $q$-system $H'$ on $F$ such that for all $R\in H\setminus H[F]$, there is $R'\in H'$ so that $R\cap F\leqslant R'$;
    \item A basis $(v_1^\ast,\ldots,v_{\dim\mr{Vec}(H)}^\ast\}$ for $\mr{Vec}(H)$ so that the last $\dim F$ vectors span $F$
    \item An \emph{avoidance set}: $r$-dimensional $q$-system $\mc{R}_\mr{init}\subseteq\mr{Gr}_q(n,r)$ which is $\theta$-bounded (\cref{def:bounded});
    \item A bounded list of \emph{roots}: a sequence of injective linear maps $\phi_1,\ldots,\phi_x\colon F\hookrightarrow\mb{F}_q^n$ such that for every $(r-1)$-space $Q$, $\#\{t\in[x]\colon Q\leqslant R\text{ for some }R\in\phi_t(H')\}\le\theta q^n$.
    \item A list of \emph{avoided color sets}: for all $t\in[x]$, sets $\mc{C}_t\subseteq[z]$ of size at most $z^{1/2}$;
    \item A list of \emph{preparatory color sets}: $\mc{C}_{t,j}\subseteq[z]$ of size at most $h$ for all $t\in[x],j\in[v]$;
    \item \emph{Extensions}: $E_t:=(\phi_t,F,H)$;
    \item \emph{Potential allowed embeddings}: $\mc{H}_t':=\{\phi^\ast\in\mc{X}_{E_t}(G_\mr{tem})\colon\text{for all }R\in H\setminus H[F],~\phi^\ast(H)\subseteq K_{\mr{ind}(\phi^\ast(R))}\text{ and }\mr{ind}(\phi^\ast(R))\notin\mc{C}_t\text{ are distinct};~\phi^\ast(v_j^\ast)\in\bigcap_{i\in\mc{C}_{t,j}}K_i\text{ for all }j\in[v]\}$.
\end{itemize}
\end{definition}

\begin{lemma}\label{lem:master-disjointness}
Given \cref{sub:setup,def:template}, we have the following whp over the randomness of the template, as long as $h\ge 1$, $C=C_{\ref{lem:master-disjointness}}(h,q,s) > 0$, $c=c_{\ref{lem:master-disjointness}}(h,q,s) > 0$, $q^{-cn}\le\tau\le c$, $d\ge d_{\ref{lem:master-disjointness}}(r)$, $\ell\ge\ell_{\ref{lem:master-disjointness}}(d,s)$, and $n$ is large.

For any choice of data as in \cref{def:disjointness-setup} with $\dim\mr{Vec}(H)\le h$ and $q^{-cn}\le\theta\le(\tau/z)^{1/c}$, let $\mc{R}_0:=\emptyset$ and then consider running the following random process for $1\le t\le x$:
\begin{itemize}
    \item Let $\mc{H}_t=\{\phi^\ast\in\mc{H}_t'\colon\emph{for all}~R\in H\setminus H[F],~\phi^\ast(R)\notin\mc{R}_\mr{init}\cup\chi(\mc{R}_{t-1})\}$.
    \item If $|\mc{H}_t|<(\tau/z)^Cq^{vn}$, we call the process \emph{failed}, we set $\phi_t^\ast=\phi_{t+1}^\ast=\cdots=\phi_x^\ast:=\ast$, a special wildcard value, we set $\mc{R}_t=\cdots=\mc{R}_x:=\mc{R}_{t-1}$, and we stop the iteration.
    \item Otherwise, we sample $\phi_t^\ast\sim\mr{Unif}(\mc{H}_t)$, let $\mc{R}_t:=\mc{R}_{t-1}\cup\phi_t^\ast(H\setminus H[F])$ as a set, and continue.
\end{itemize}
Then the random process has the property that it whp never fails, and furthermore for any $\mc{F}\subseteq\mr{Gr}_q(n,r)$ we have $\mb{P}[|\mc{F}\cap\chi(\mc{R}_x)|\ge\theta^{2/3}|\mc{F}|]\le\exp(-\theta|\mc{F}|)$.
\end{lemma}
\begin{remark}
Note that the property that the random process whp never fails is only considering the randomness of the random process, not the template. One should think that we condition on a suitable outcome of the template, and almost all outcomes are such that these random processes whp run to completion.
\end{remark}
\begin{proof}
Whp the template satisfies \cref{prop:template-extendable} for $h$. Further, given a $(\dim F)$-space $F'\leqslant\mb{F}_q^n$, Chernoff shows that with probability at least $1-\exp(-\Omega_{q,\ell}(z))$ over the random injections $\iota_1,\ldots,\iota_z$, there are at least $2z^{1/2}$ many $i\in[z]$ so that $F'\leqslant K_i$ and $\dim_L\mr{span}_L(\iota_i^{-1}(F'))=\dim F$. Taking a union bound, whp this property holds for all $(\dim F)$-spaces. We thus condition on such an outcome and treat it as non-random. Additionally, note that the condition on the $\phi_t$ in \cref{def:disjointness-setup} implies that if $\dim Q'=u\in\{0,1,\ldots,r-1\}$, then
\begin{equation}\label{eq:master-disjointness-boundedness}
\#\{t\in[x]\colon Q'\leqslant R\text{ for some }R\in\phi_t(H')\}\le\theta q^{(r-u)n},
\end{equation}
since all such $t$ have the property that there is an $(r-1)$-space $Q$ with $Q\geqslant Q'$ which is contained in $R$ for some $R\in\phi_t(H')$, and there are clearly at most $q^{(r-u-1)n}$ choices for $Q$ containing $Q'$.

Let $C$ be large to be chosen later in terms of $h,q,s$. Note that
\begin{equation}\label{eq:cover-down-conditional}
\mb{P}[\phi_t^\ast=\phi|\phi_1^\ast,\ldots,\phi_{t-1}^\ast]\le(z/\tau)^Cq^{-vn}
\end{equation}
for all $t\in[x]$ and $\phi\in\mc{X}_{E_t}(G)$ (recall $G=\mr{Gr}_q(n,r)$) due to the fact that we stop the process if failure occurs (and put $\ast$ everywhere after failure). Now consider any family $\mc{F}$ of $r$-spaces. Let us consider the random variable $X=X_\mc{F}:=\#\{(R,R')\in\mc{F}\times\mc{R}_x\colon R\in\chi(R')\}\ge|\mc{F}\cap\chi(\mc{R}_x)|$. For every $R\in\mr{Gr}_q(n,r)$ and $t\in[x]$ let $\mc{H}_{R,t}=\{\phi\in\mc{H}_t'\colon R\in\phi(H\setminus H[F])\}$ and note that
\[\qbinom{s}{r}_q^{-1}X\le\sum_{t=1}^x\bigg(\sum_{R\in\chi(\mc{F})}\mbm{1}_{\phi_t^\ast\in\mc{H}_{R,t}}\bigg)=:\sum_{t=1}^xX_t.\]
The inequality is since $R\in\chi(R')$ and $R'\in\chi(R)$ are equivalent, and since every space in $\chi(\mc{F})$ can come from at most $\qbinom{s}{r}_q$ spaces in $\mc{F}$. Now, the right is a sum of nonnegative integer random variables $X_t$, each of which is bounded by $|H\setminus H[F]|\le q^{rh}$ (given $\phi_t^\ast$, the only $R\in\mc{F}$ that can contribute $1$ are those with $R\in\phi_t^\ast(H\setminus H[F])$). Let $Y_t=\mbm{1}_{X_t\neq 0}$ and note $X_t\le q^{rh}Y_t$, so $Y=\sum_{t=1}^xY_t\ge q^{-r(h+s)}X$. Furthermore, \cref{eq:cover-down-conditional} shows that
\[\mb{P}[Y_t=1|\phi_1^\ast,\ldots,\phi_{t-1}^\ast]\le\sum_{R\in\chi(\mc{F})}|\mc{H}_{R,t}|(z/\tau)^Cq^{-vn}=:\mu_t\]
for all $t\in[x]$. Let $\mu=\sum_{t=1}^x\mu_t$. Thus we can apply \cref{lem:bernoulli-domination} with $p_t=\min(\mu_t,1)$ and then \cref{lem:chernoff} to deduce
\begin{align}\label{eq:disjoint-family-bounded}
\mb{P}[X_\mc{F}\ge\theta^{2/3}|\mc{F}|]\le\mb{P}[Y\ge\theta^{3/4}|\mc{F}|]\le\exp(-\Omega(\theta^{3/4}|\mc{F}|))\le\exp(-\theta|\mc{F}|)
\end{align}
as long as $2\mu\le\theta^{3/4}|\mc{F}|$. We now turn to demonstrating this inequality in order to establish \cref{eq:disjoint-family-bounded}.

For fixed $R\in\chi(\mc{F})$ we have
\begin{align}
\sum_{t=1}^x|\mc{H}_{R,t}|&=\sum_{u=0}^{r-1}\sum_{\substack{t\in[x]\\\dim(R\cap\phi_t(F))=u}}|\mc{H}_{R,t}|\le\sum_{u=0}^{r-1}\bigg(\sum_{\substack{R_0\in H\setminus H[F]\\R'\in H'\\\dim(R_0\cap R')=u}}\sum_{\substack{t\in[x]\\\phi_t(R_0\cap R')\leqslant R}}\#\{\phi\in\mc{H}_t'\colon R=\phi(R_0)\}\bigg)\notag\\
&\le\sum_{u=0}^{r-1}O_{h,q,r}(\theta q^{(r-u)n})\cdot O_{q,s}(q^{(\dim\mr{Vec}(H)-(\dim F+r-u))n})\le O_{h,q,s}(\theta q^{vn}).\label{eq:master-disjointness-mean}
\end{align}
The equality follows since for every $\phi\in\mc{H}_{R,t}$, we have $R=\phi(R_0)$ for some $R_0\in H\setminus H[F]$ and thus $R\cap\phi_t(F)=\phi(R_0\cap F)$ has some dimension $u\le r-1$ (recall $\phi\in\mc{H}_{R,t}\subseteq\mc{H}_t'$ implies $\phi,\phi_t$ agree on $F$). The first inequality of \cref{eq:master-disjointness-mean} follows since in such a situation, there furthermore is $R'\in H'$ with $R_0\cap F\leqslant R'$ by the given conditions. Since $R'\leqslant F$ this implies $R_0\cap F=R_0\cap R'$ hence $\dim(R_0\cap R')=u$ and $\phi_t(R_0\cap R')=\phi_t(R_0\cap F)\leqslant R$.

Finally, the inequality in the second line of \cref{eq:master-disjointness-mean} is proven by noting there are at most $O_{h,q,r}(1)\cdot\theta q^{(r-u)n}$ terms in the inner double sum by counting ways to choose $R_0,R'$ and then by \cref{eq:master-disjointness-boundedness} applied to $Q'=\phi_t(R_0\cap R')$. Then note that $\#\{\phi\in\mc{H}_t'\colon R=\phi(R_0)\}=O_{q,r}(q^{(\dim\mr{Vec}(H)-(\dim F+r-u))n})$ since this is bounded by the number of extensions $\phi\in\mc{X}_{E_t}(G)$ which map $R_0$ to $R$ (up to some permutation which has $O_{q,r}(1)$ choices), and then the remaining number of free dimensions to embed is $\dim\mr{Vec}(H)-(\dim F+r-u)$ since $\dim\mr{span}_{\mb{F}_q}(R_0\cup F)=\dim F+r-\dim(R_0\cap F)=\dim F+r-u$.

Now summing \cref{eq:master-disjointness-mean} over $R\in\chi(\mc{F})$ and the definition of $\mu_t$ shows that $\mu\le O_{h,q,s}(1)\cdot\theta(z/\tau)^C|\chi(\mc{F})|$, which demonstrates the desired inequality $2\mu\le\theta^{3/4}|\mc{F}|$ as long as $1/c\ge 4C+1$ is chosen appropriately and $n$ is large. We are using $|\chi(\mc{F})|\le\qbinom{s}{r}_q|\mc{F}|$. Thus \cref{eq:disjoint-family-bounded} indeed holds. This demonstrates the last part of the lemma.

Next we claim that the process does not fail whp, which will finish. Fix some $t\in[x]$, and consider the event $\mc{E}_t$ that $|\mc{H}_t|\le(\tau/z)^Cq^{vn}$. Note that \cref{prop:template-extendable} shows
\[\mc{H}_t'=\bigg\{\phi^\ast\in\mc{X}_{E_t}(G_\mr{tem})\colon\genfrac{}{}{0pt}{}{\forall R\in H\setminus H[F],~\phi^\ast(H)\subseteq K_{\mr{ind}(\phi^\ast(R))}\text{ and }\mr{ind}(\phi^\ast(R))\notin\mc{C}_t\text{ are distinct};}{\phi^\ast(v_j^\ast)\in\bigcap_{i\in\mc{C}_{t,j}}K_i\text{ for all }j\in[v]}\bigg\}\]
satisfies $|\mc{H}_t'|\ge(\tau/z)^{C/2}q^{vn}$ if $C$ was chosen appropriately large in terms of $h,q,r$. Indeed, the $(\dim F)$-space $\phi_t(F)$ is in at least $2z^{1/2}$ many $K_i$ with $\dim_L\mr{span}_L(\iota_i^{-1}(\phi_t(F)))=\dim F$ and thus we assign distinct such $i\notin\mc{C}_t$ to all of $H\setminus H[F]$ to construct the coloring function $\psi$ (recall $|\mc{C}_t|\le z^{1/2}$). The condition $|\psi^{-1}(i)|\le 1\le\qbinom{s}{r}_q$ is trivial, while \ref{prop:template-extendable}(a,b) follow from the condition on the $i\in\mr{range}(\psi)$ (note that if an $r$-space $R$ in $K$ extends to $r$ dimensions over $L$, then every $\mb{F}_q$-subspace of $R$ of dimension $u$ extends to $u$ dimensions over $L$). Also, we can choose the $\pi_i,x_R$ arbitrarily as discussed in the remark following \cref{prop:template-extendable}. Finally, we let the basis for $\mr{Vec}(H)$ in \cref{prop:template-extendable} be the same as the one we are given, and let $\mc{C}_j=\mc{C}_{t,j}$ for $j\in[v_{E_t}]=[v]$.

Thus the event $\mc{E}_t$ implies that $|\mc{H}_t'\setminus\mc{H}_t|\ge((\tau/z)^{C/2}-(\tau/z)^C)q^{vn}$. But
\[\mc{H}_t'\setminus\mc{H}_t\subseteq\bigcup_{R\in H\setminus H[F]}(\{\phi^\ast\in\mc{X}_{E_t}(G)\colon\phi^\ast(R)\in\mc{R}_\mr{init}\}\cup\{\phi^\ast\in\mc{X}_{E_t}(G)\colon\phi^\ast(R)\in\chi(\mc{R}_x)\}).\]
Given $R\in H\setminus H[F]$ with $\dim(R\cap F)=u\in\{0,\ldots,r-1\}$, the size of the first set within the union is bounded by $\theta q^{(r-u)n}\cdot O_{q,r}(q^{(\dim\mr{Vec}(H)-(\dim F+r-u))n})=O_{q,r}(\theta q^{vn})$ due to a similar argument to the proof of the second inequality of \cref{eq:master-disjointness-mean}: choose $\phi^\ast(R)\in\mc{R}_\mr{init}$ which contains the fixed $u$-space $\phi^\ast(R\cap F)=\phi_t(R\cap F)$, leading to $\theta q^{(r-u)n}$ possibilities due to the $\theta$-boundedness (\cref{def:bounded}) of $\mc{R}_\mr{init}$, and then consider the number of ways to extend to all of $\mr{Vec}(H)$ given the map on $\mr{span}_{\mb{F}_q}(F\cup R)$. And if $\chi(\mc{R}_x)$ were $\theta^{2/3}$-bounded then we would obtain a similar bound on the size of the second set, which would ultimately show
\[|\mc{H}_t'\setminus\mc{H}_t|\le O_{h,q,r}(\theta^{2/3}q^{vn})\le(\tau/z)^Cq^{vn}\]
as long as $1/c\ge 2C+1$ is chosen appropriately. This contradicts $\mc{E}_t$!

Therefore $\mc{E}_t$ implies that $\chi(\mc{R}_x)$ is not $\theta^{2/3}$-bounded, i.e., there is some $(r-1)$-space $Q$ so that with $\mc{F}_Q=\{R\in\mr{Gr}_q(n,r)\colon R\geqslant Q\}$ we have $X_{\mc{F}_Q}>\theta^{2/3}q^n\ge\theta^{2/3}|\mc{F}_Q|$. But using \cref{eq:disjoint-family-bounded} on at most $q^{(r-1)n}$ possible families $\mc{F}_Q$ (of size $\Theta_{q,r}(q^n)$ each), recalling $X_\mc{F}\ge|\mc{F}\cap\chi(\mc{R}_x)|$, and taking a union bound, we see that $\mb{P}[\mc{E}_t]\le q^{-3rn}$, say. Finally, note that \cref{eq:master-disjointness-boundedness} implies that $x\le q^{2rn}$ so taking a union bound over all $t$ finishes the proof.
\end{proof}

Given \cref{lem:master-disjointness}, we establish \cref{lem:final-covering} which shows that we can cover all of the leftover coming from \cref{prop:approximate-covering} disjointly by $s$-spaces in a way that only touches template $r$-spaces; this will allow us to focus on a ``spill'' within the template in the proof of \cref{thm:main}. Additionally, we can guarantee that the spillover is field bounded and field disjoint.

\begin{lemma}\label{lem:final-covering}
Given the setup of \cref{sub:setup,def:template}, we have the following as long as $c=c_{\ref{lem:final-covering}}(q,s)>0$, $q^{-cn}\le\tau\le c$, $d\ge d_{\ref{lem:final-covering}}(r)$, $\ell\ge\ell_{\ref{lem:final-covering}}(d,s)$, and $n$ is large. Whp over the randomness of the template, for any $r$-dimensional $q$-system $\mc{R}$ on $V$ with $\mc{R}\cap G_\mr{tem}=\emptyset$ which is $\theta$-bounded (\cref{def:bounded}) for some $q^{-cn}\le\theta\le(\tau/z)^{1/c}$ there is an $s$-dimensional $q$-system $\mc{S}$ such that $\snorm{\partial_{s,r}\mc{S}}_\infty\le 1$ and $\partial_{s,r}\mc{S}=\mc{R}\cup\mc{R}'$ where $\mc{R}'\subseteq G_\mr{tem}$ and $\mc{R}'$ is field disjoint and $(\theta^{1/2},L)$-field bounded with respect to the template.
\end{lemma}
\begin{proof}
Let $\mc{R}=\{R_1,\ldots,R_{|\mc{R}|}\}$ be an arbitrary ordering of the $r$-spaces. We run the process in \cref{lem:master-disjointness} with the following choices:
\begin{itemize}
    \item $H$ is the set of all $r$-spaces in $\mb{F}_q^s$, $F\leqslant\mb{F}_q^s$ has dimension $r$, and $H'=\{F\}$;
    \item $\mc{R}_\mr{init}=\emptyset$;
    \item $x=|\mc{R}|$ and $\phi_t$ is an arbitrary linear injection such that $\phi_t(F)=R_t$ for $t\in[|\mc{R}|]$;
    \item The basis $v^\ast$ is arbitrary and all color sets $\mc{C}_t,\mc{C}_{t,j}$ are empty.
\end{itemize}
The $\theta$-boundedness of $\mc{R}$ implies the boundedness of the \emph{roots} in \cref{def:disjointness-setup}, so the conclusions apply. Namely, this process does not fail whp and for any $\mc{F}\subseteq\mr{Gr}_q(n,r)$ we have $\mb{P}[|\mc{F}\cap\mc{R}_x|\ge\theta^{2/3}|\mc{F}|]\le\exp(-\theta|\mc{F}|)$.

Let $S_t$ be the random $s$-space $\phi_t^\ast(\mr{Vec}(H))$ that results and let $\mc{S}=\{S_1,\ldots,S_{|\mc{R}|}\}$, ignoring wildcard values. We see that when the process does not fail, $\partial_{s,r}\mc{S}$ is a set of $r$-spaces (using the disjointness inherent to the process) and $\partial_{s,r}\mc{S}=\mc{R}\cup\mc{R}'$ (using $\mc{R}\cap G_\mr{tem}=\emptyset$) where $\mc{R}'\subseteq G_\mr{tem}$, and in fact $\mc{R}'=\mc{R}_x$. Thus $(\theta^{1/2},L)$-field boundedness of $\mc{R}'$ occurs whp due to considering a union bound over at most $zq^{(r-1)n}$ sets $\mc{F}$ of the following form: take every possible $i\in[z]$ and $(r-1)$-space $Q^\ast$ over $L$ and consider $\mc{F}$ which is the collection of $r$-spaces which contribute to $Q^\ast$ in \cref{def:field-bounded}, then augment $\mc{F}$ so it has size $\lfloor\theta^{-1/6}q^n\rfloor$.

Finally, to demonstrate field disjointness, note that the definition of the process in \cref{lem:master-disjointness} makes each new embedding have its new $r$-spaces at each time $t$ avoid $\chi$ of the previous new $r$-spaces. Furthermore, at each time $t$ the new $r$-spaces are mutually field disjoint since they are involved in different template indices. We are done.
\end{proof}

The next result shows that the spill, an $r$-dimensional $q$-system, can be integrally decomposed via $s$-spaces which are supported on $G_\mr{tem}$ and are rainbow (the constituent $r$-spaces are all in different parts $G_{\mr{tem},i}$ of the template) in a way that is field bounded. Furthermore, we ensure every $r$-space occurs at most once in the positive and negative parts of the integral decomposition. We also ensure that the $r$-spaces in the positive part are field disjoint (assuming the spillover is field disjoint).
\begin{lemma}\label{lem:final-rainbow}
Given the setup of \cref{sub:setup,def:template}, we have the following as long as $c=c_{\ref{lem:final-rainbow}}(q,s)>0$, $q^{-cn}\le\tau\le c$, $d\ge d_{\ref{lem:final-rainbow}}(r)$, $\ell\ge\ell_{\ref{lem:final-rainbow}}(d,s)$, and $n$ is large. Whp over the randomness of the template, for any (set) $\mc{R}\subseteq G_\mr{tem}$ which is field disjoint and $(\theta,L)$-field bounded for some $q^{-cn}\le\theta\le(\tau/z)^{1/c}$ such that $\mc{R}\in\mc{L}=\partial_{s,r}\mb{Z}^{\mr{Gr}_q(n,s)}$ (treated as a vector), there is $\Phi\in\mb{Z}^{\mr{Gr}_q(n,s)}$ such that:
\begin{itemize}
    \item $\mc{R}=\partial_{s,r}\Phi$ and $\snorm{\partial_{s,r}\Phi^+}_\infty\le 1$;
    \item $\partial_{s,r}\Phi^\pm$ are $(\theta^{1/4},L)$-field bounded;
    \item For every $S\in\mr{supp}(\Phi)$ there are distinct $i_R\in[z]$ such that $R\in G_{\mr{tem},i_R}$ for all $R\in\mr{Gr}(S,r)$.
    \item $\partial_{s,r}\Phi^\pm$ are field disjoint.
\end{itemize}
\end{lemma}
\begin{remark}
The proof uses a subspace exchange process quite similar to \cref{prop:sparse-into-gadgets}, but we must guarantee a disjointness condition that requires use of \cref{lem:master-disjointness}.
\end{remark}
\begin{proof}
Let $k=k_{\ref{prop:subspace-exchange}}(s)$ and let $\Ups,\Ups'$ be the $s$-dimensional $q$-systems on $\mb{F}_q^k$ coming from \cref{prop:subspace-exchange}. By \cref{lem:strength-bounded} and then \cref{thm:bounded-inverse} we can write
\[\mc{R}=\partial_{s,r}\Phi_0\]
with $\partial_{s,r}\Phi_0^\pm$ being $C_{\ref{thm:bounded-inverse}}(q,s)z\theta$-bounded. Our goal is to massage $\Phi_0$ into a new signed collection $\Phi$ where every $r$-space appears at most once in the positive and negative parts $\partial_{s,r}\Phi^\pm$ and we have the necessary boundedness. A priori, though, $\partial_{s,r}\Phi_0$ could involve massive cancellation on certain $r$-spaces (although the boundedness of $\Phi_0$ helps limit this somewhat). We will first ``preprocess'' this collection using \cref{lem:master-disjointness} to make these conditions ``hold outside $\partial_{s,r}\Phi_0^+$'' in some sense. Then we will essentially go through ``cancelling pairs'' one by one and fix them via a random embedding of a finite structure coming from \cref{prop:subspace-exchange} (similar to the proof of \cref{prop:sparse-into-gadgets}, but using \cref{lem:master-disjointness}). These can be thought of as ``splitting'' and ``elimination'' in the context of \cite{Kee14}.

Let $(S_t)_{1\le t\le y'}$ be an arbitrary ordering of the positive $s$-spaces in $\Phi_0$ with multiplicity, and $(S_t)_{y'+1\le t\le y}$ be an arbitrary ordering of the negative $s$-spaces. We run the process in \cref{lem:master-disjointness} with the following choices:
\begin{itemize}
    \item $H=H^{(1)}$ is the set of all $r$-spaces contained in $s$-spaces of $\Ups$ (with $\mr{Vec}(H^{(1)})=\mb{F}_q^k$), $F=F^{(1)}$ is some fixed $s$-space in $\Ups$, and $H'=H^{(1)\prime}=\mr{Gr}(F,r)$;
    \item $\mc{R}_\mr{init} = \mc{R}_\mr{init}^{(1)}=\chi(\mc{R})\cup\bigcup_{t=1}^y\mr{Gr}(S_t,r)$;
    \item $x=y$ and $\phi_t=\phi_t^{(1)}$ is an arbitrary linear injection such that $\phi_t^{(1)}(F^{(1)})=S_t$ for $t\in[y]$;
    \item $\mc{C}_t = \{\mr{ind}(R)\colon R\in\mr{Gr}(S_t,r)\cap G_\mr{tem}\}$;
    \item The basis $v^\ast$ is arbitrary and all color sets $\mc{C}_{t,j}$ are empty;
    \item $\theta$ is replaced by $(2C_{\ref{thm:bounded-inverse}}+1)z\theta$.
\end{itemize}
Since $\partial_{s,r}\Phi_0^\pm$ are $C_{\ref{thm:bounded-inverse}}z\theta$-bounded and since $\mc{R}$ is $(\theta,L)$-field bounded hence $\chi(R)$ is $z\theta$-bounded by \cref{lem:chi-bounded}, we obtain the necessary condition on $\mc{R}_\mr{init}$ as well as the boundedness of the \emph{roots} in \cref{def:disjointness-setup}, so the conclusions apply to the random output $\phi_1^\ast,\ldots,\phi_y^\ast$. In particular, whp it runs to completion. Write $\mc{R}_y^\ast$ for the value of $\mc{R}_x$ produced by the process (recall we set $x=y$ in this application of \cref{lem:master-disjointness}).

Given the definitions of $H^{(1)},F^{(1)}$ above, let $F_1,\ldots,F_a$ for $a=\qbinom{s}{r}_q$ be the $s$-spaces in $\Ups'$ with $\dim(F_i\cap F^{(1)})=r$ (there are precisely $a$ of them by \cref{prop:subspace-exchange}). If the process runs to completion we have, writing $\mr{sgn}_t=1$ for $t\in[y']$ and $-1$ for $y'+1\le t\le y$,
\begin{align}
\mc{R}&=\partial_{s,r}\Phi_0=\sum_{t=1}^y\mr{sgn}_t\partial_{s,r}e_{S_t}=\sum_{t=1}^{y'}\mr{sgn}_t\partial_{s,r}\bigg(\sum_{S\in\phi_t^\ast(\Ups')}e_S-\sum_{S\in\phi_t^\ast(\Ups\setminus\{F\})}e_S\bigg)\notag\\
&= \partial_{s,r}\sum_{t=1}^y\sum_{j=1}^a\mr{sgn}_te_{\phi_t^\ast(F_j)}+\partial_{s,r}\bigg(\sum_{t=1}^y\sum_{S\in\phi_t^\ast(\Ups'\setminus\{F_1,\ldots,F_a\})}\mr{sgn}_te_S-\sum_{t=1}^y\sum_{S\in\phi_t^\ast(\Ups\setminus\{F\})}\mr{sgn}_te_S\bigg)\notag\\
&=:\partial_{s,r}\Phi_1+\partial_{s,r}\Phi_2\label{eq:final-rainbow-1}
\end{align}
by the third bullet of \cref{prop:subspace-exchange}, where $\Phi_1,\Phi_2$ are defined in the obvious way. (This corresponds to \emph{near} and \emph{far} cliques, respectively, in \cite{Kee14}.) Due to the inherent disjointness of the random process, all the $s$-spaces appearing in $\Phi_2$ are distinct and $\snorm{\partial_{s,r}\Phi_2^\pm}_\infty\le 1$: an $r$-space can only appear corresponding to at most one index $t\in[y]$ by definition, and then it appears positively and negatively each at most once by the first bullet of \cref{prop:subspace-exchange}. Additionally, we see that $\Phi_1\in\{-1,0,1\}^{\mr{Gr}_q(n,s)}$ and in fact $\partial_{s,r}\Phi_1^\pm\in\{0,\pm1\}^{G\setminus\mc{R}_\mr{init}^{(1)}}\times\mb{Z}^{\mc{R}_\mr{init}^{(1)}}$ by a similar argument (recall $G=\mr{Gr}_q(n,r)$). Finally, note that $\mr{supp}(\partial_{s,r}\Phi_2^\pm)\subseteq G\setminus\mc{R}_\mr{init}^{(1)}$ by definition. This implies $(\partial_{s,r}\Phi_1)_R=\mbm{1}_{R\in\mc{R}}$ for all $R\in\mc{R}_\mr{init}^{(1)}$.

Now consider pairs of $R\in\mc{R}_\mr{init}^{(1)}$ and $S\in\mr{supp}(\Phi_1)$ where $R\leqslant S$. For each $S$, there is exactly one valid value of $R$ by inspection of the definition of the process above (namely, note that $S=\phi_t^\ast(F_j)$ for some $t,j$ and then $R=\phi_t^\ast(F_j\cap F^{(1)})$ is its only $r$-space in $\mc{R}_\mr{init}^{(1)}$). Combined with $(\partial_{s,r}\Phi_1)_R=\mbm{1}_{R\in\mc{R}}$ and $\Phi_1\in\{-1,0,1\}^{\mr{Gr}_q(n,s)}$, we see that we can find a partial matching of the $s$-spaces in $\mr{supp}(\Phi_1)$ so that every space is paired with one of the opposite sign and for each $R\in\mc{R}_\mr{init}^{(1)}$ all but one of the $s$-spaces containing it are paired up. (There will be precisely $1$ unpaired $s$-space of positive sign in $\Phi_1$ exactly for those $R\in\mc{R}$.)

Let $((S_{t,1},S_{t,2}))_{1\le t\le y''}$ be an arbitrary ordering of these pairs, where $S_{t,1}$ has positive sign and $S_{t,2}$ has negative sign in $\Phi_1$. Additionally, note that $\dim(S_{t,1}\cap S_{t,2})=r$ and the intersection is an $r$-space of $\mc{R}_\mr{init}^{(1)}$ due to the following argument: if the intersection has larger dimension then there is an $r$-space shared between $S_{t,1},S_{t,2}$ other than the $R\in\mc{R}_\mr{init}^{(1)}$ that they share. By the definition of the process above, this violates either disjointness from $\mc{R}_\mr{init}^{(1)}$ or disjointness from each other. Let the remaining signed $s$-spaces in $\Phi_1\in\{-1,0,1\}^{\mr{Gr}_q(n,s)}$ be $(S_k^+)_{k\in\mc{P}}$, and note they all have a positive sign in $\Phi_1$ (by the above parenthetical). Note we can now write
\begin{equation}\label{eq:final-rainbow-2}
\Phi_1=\sum_{k\in\mc{P}}e_{S_k^+}+\sum_{t=1}^{y''}(e_{S_{t,1}}-e_{S_{t,2}}).
\end{equation}

We now run the process in \cref{lem:master-disjointness} with the following choices:
\begin{itemize}
    \item $H=H^{(2)}$ is obtained by gluing two copies of the construction in \cref{prop:subspace-exchange} along an $s$-space $S_0$ in the two copies of $\Ups$, and then considering all $r$-spaces contained within (and $\mr{Vec}(H^{(2)})=\mb{F}_q^{2k-s}$). Let $\Ups_1,\Ups_1'$ be the $s$-spaces for one copy and $\Ups_2,\Ups_2'$ for the other (so $\Ups_1\cap\Ups_2=\{S_0\}$). Let $R_0\leqslant S_0$ be an $r$-space and $S_1$ be the unique space of $\Ups_1'$ with $S_1\cap S_0=R_0$ and $S_2$ be the unique space of $\Ups_2'$ with $S_2\cap S_0=R_0$. Let $F=F^{(2)}=\mr{span}_{\mb{F}_q}(S_1\cup S_2)$ and $H'=H^{(2)\prime}=\mr{Gr}(S_1,r)\cup\mr{Gr}(S_2,r)$;
    \item $\mc{R}_\mr{init}=\mc{R}_\mr{init}^{(2)}=\mc{R}_\mr{init}^{(1)}\cup\chi(\mc{R}_y^\ast)$;
    \item $x=y''$ and $\phi_t=\phi_t^{(2)}$ is an arbitrary linear injection such that $\phi_t^{(2)}(S_1)=S_{t,1}$ and $\phi_t^{(2)}(S_2)=S_{t,2}$.
    \item $\mc{C}_t=\{\mr{ind}(R)\colon R\in(\mr{Gr}(S_{t,1},r)\cup\mr{Gr}(S_{t,2},r))\cap G_\mr{tem}\}$;
    \item The basis $v^\ast$ is arbitrary and all color sets $\mc{C}_{t,j}$ are empty.
    \item $\theta$ is replaced by $\theta^{1/2}$.
\end{itemize}
The necessary condition on $\mc{R}_\mr{init}=\mc{R}_\mr{init}^{(2)}$ as well as the boundedness of the \emph{roots} in \cref{def:disjointness-setup} are nontrivial. They both are derived from the following argument: considering various $\mc{F}=\mc{F}_Q=\{R\in\mr{Gr}_q(n,r)\colon R\geqslant Q\}$ for $Q\in\mr{Gr}_q(n,r-1)$ and similarly defined sets over $L$ and taking a union bound (using the last property of \cref{lem:master-disjointness}) shows that whp everything involved with the output of the first process thus ends up being say $O(z(z\theta)^{2/3})$-bounded. For instance, using this argument and applying \cref{lem:chi-bounded}, we see that whp $\chi(\mc{R}_y^\ast)$ is $O(z(z\theta)^{2/3})$-bounded. Furthermore, this boundedness includes not only the ``new $r$-spaces'' (which go by the name $\mc{R}_x$ in the previous application of \cref{lem:master-disjointness}) but also the ``original $r$-spaces'' ($\mc{R}_\mr{init}^{(1)}$), using the boundedness condition that allowed us our first application of \cref{lem:master-disjointness}. The total boundedness parameter is at most $\theta^{1/2}$ if $c$ is chosen appropriately small, as desired.

So whp over the randomness of outcomes of the first process, we are allowed to run the second process and the conclusions of \cref{lem:master-disjointness} apply to the random output $\phi_1^{\ast\prime},\ldots,\phi_x^{\ast\prime}$. In particular, the process runs to completion whp. Now \cref{eq:final-rainbow-1,eq:final-rainbow-2} together give
\begin{equation}\label{eq:final-rainbow-3}
\mc{R}=\sum_{i\in\mc{P}^+}\partial_{s,r}e_{S_t^+}+\sum_{i=1}^{y''}\partial_{s,r}(e_{S_{t,1}}-e_{S_{t,2}})+\partial_{s,r}\Phi_2.
\end{equation}
Note that for $i\in[y'']$,
\begin{equation}\label{eq:final-rainbow-4}
\partial_{s,r}(e_{S_{t,1}}-e_{S_{t,2}}) = \partial_{s,r}\bigg(\sum_{S\in\phi_i^{\ast\prime}(\Ups_1\setminus\{S_0\})}e_S-\sum_{S\in\phi_i^{\ast\prime}(\Ups_2\setminus\{S_0\})}e_S-\sum_{S\in\phi_i^{\ast\prime}(\Ups_1'\setminus\{S_1\})}e_S+\sum_{S\in\phi_i^{\ast\prime}(\Ups_2'\setminus\{S_2\})}e_S\bigg)
\end{equation}
since $\phi_i^{\ast\prime}(S_j)=S_{i,j}$ for $j\in\{1,2\}$ and by \cref{prop:subspace-exchange}. We also cancelled the common $e_{S_0}$ in $\Ups_1,\Ups_2$.

Plugging \cref{eq:final-rainbow-4} into \cref{eq:final-rainbow-3} for all $t\in[y'']$ yields a linear combination for $\mc{R}$ in terms of $\partial_{s,r}e_S$ for various $s$-spaces $S$, call it $\mc{R}=\partial_{s,r}\Phi_3$. Furthermore, inspection of the definition of the second random process shows that $\snorm{\partial_{s,r}\Phi_3^+}_\infty\le 1$: on $\mc{R}_\mr{init}^{(1)}$ (i.e., for the $e_{S_i^+}$ terms) it is true due to the pairing of cancelling $r$-spaces and for the rest we use disjointness. Here we are using in a key way that we cancelled the $s$-space $S_0$, for which having two copies of $\phi_i^{\ast\prime}(S_0)$ would introduce a violation along the $r$-space $\phi_t^{\ast\prime}(R_0)=\phi_t^{(2)}(S_1\cap S_2)=S_{t,1}\cap S_{t,2}\in\mc{R}_\mr{init}^{(1)}$. Also, we are implicitly using that all of the $s$-spaces in $\Ups_j\setminus\{S_0\},\Ups_j'\setminus\{S_j\}$ for $j\in\{1,2\}$ (those appearing on the right side of \cref{eq:final-rainbow-4}) have constituent $r$-spaces which are contained in $F^{(2)}$ only if they are in the set $\mr{Gr}(S_1,r)\cup\mr{Gr}(S_2,r)$. This can be seen since we are gluing two copies of the construction in \cref{prop:subspace-exchange} along an $s$-space and otherwise linearly disjointly.

Additionally, for every $S\in\mr{supp}(\Phi_3)$ we have distinct $i_R\in[z]$ such that $R\in G_{\mr{tem},i_R}$ for all $R\in\mr{Gr}(S,r)$ due to the definition of the two random processes used. In particular, one uses that the values of $H,F$ considered are such that every $s$-space used to define $H$ which is not contained in $F$ intersects $F$ in dimension at most $r$. Thus every $s$-space is assigned ``new distinct colors'' except one fixed color which is explicitly avoided by the new colors. Furthermore, every $s$-space is processed by at least one of the two processes so this applies to all $s$-spaces.

Next, we show $(\theta^{1/4},L)$-field boundedness of $\partial_{s,r}\Phi_3^\pm$. With this in hand we will see that taking $\Phi=\Phi_3$ finishes the proof. For this, note that all $r$-spaces of $\partial_{s,r}\Phi_3^+$ outside of $\mc{R}$ are introduced by one of the two disjoint random processes, and so we can apply the last property of \cref{lem:master-disjointness} to various $\mc{F}$ similar to the end of the proof of \cref{lem:final-covering} and take a union bound (note $\theta$ is replaced by $\theta^{1/2}$). Combining with the given $(\theta,L)$-field boundedness of $\mc{R}$, we obtain the desired field boundedness.

Finally, we show field disjointness. Again, every constituent $r$-space of $\partial_{s,r}\Phi_3^\pm$ outside of $\mc{R}$ is introduced by one of the two disjoint random processes. Since the processes create rainbow embeddings, we know that the $r$-spaces introduced at a specific time $t$ are field disjoint, and the definition of each random process shows they do not interfere across times. Furthermore, since the second random process excludes use of $\chi(\mc{R}_y^\ast)$, we see that the new $r$-spaces of the two processes are field disjoint as well. The fact $\chi(\mc{R})\subseteq\mc{R}_\mr{init}^{(1)}\subseteq\mc{R}_\mr{init}^{(2)}$ means nothing interferes with the $r$-spaces in $\mc{R}$ either. We are done.
\end{proof}

We will want to turn this into a monochromatic decomposition, among other things, but to do so we will need to first slightly massage the output of \cref{lem:final-rainbow}. Specifically, we need every $s$-space $S$ used to be rainbow in the template and additionally have the property that it is contained within $\bigcap_{i\in\mc{I}}K_i$ where $\mc{I}=\{\mr{ind}(R)\colon R\in\mr{Gr}(S,r)\}$.
\begin{lemma}\label{lem:final-color-consistent}

Given the setup of \cref{sub:setup,def:template}, we have the following as long as $c=c_{\ref{lem:final-color-consistent}}(q,s)>0$, $q^{-cn}\le\tau\le c$, $d\ge d_{\ref{lem:final-color-consistent}}(r)$, $\ell\ge\ell_{\ref{lem:final-color-consistent}}(d,s)$, and $n$ is large. Whp over the randomness of the template, for any (set) $\mc{R}\subseteq G_\mr{tem}$ which is field disjoint and $(\theta,L)$-field bounded for some $q^{-cn}\le\theta\le(\tau/z)^{1/c}$ such that $\mc{R}\in\mc{L}=\partial_{s,r}\mb{Z}^{\mr{Gr}_q(n,s)}$ (treated as a vector), there is $\Phi\in\mb{Z}^{\mr{Gr}_q(n,s)}$ such that
\begin{itemize}
    \item $\mc{R}=\partial_{s,r}\Phi$ and $\snorm{\partial_{s,r}\Phi^+}_\infty\le 1$;
    \item $\partial_{s,r}\Phi^\pm$ are $(\theta^{1/8},L)$-field bounded;
    \item For every $S\in\mr{supp}(\Phi)$ there are distinct $i_R\in[z]$ such that $R\in G_{\mr{tem},i_R}$ for all $R\in\mr{Gr}(S,r)$, and furthermore $S\leqslant\bigcap_{R\in\mr{Gr}(S,r)}K_{i_R}$.
    \item $\partial_{s,r}\Phi^\pm$ are field disjoint.
\end{itemize}
\end{lemma}
\begin{proof}
First apply \cref{lem:final-rainbow} to obtain some $\Phi_0$ with
\[\mc{R}=\partial_{s,r}\Phi_0\]
so that $\snorm{\partial_{s,r}\Phi^+}_\infty\le 1$, $\partial_{s,r}\Phi^\pm$ are $(\theta^{1/4},L)$-field bounded, and for every $S\in\mr{supp}(\Phi)$ there are distinct $i_R\in[z]$ such that $R\in G_{\mr{tem},i_R}$ for all $R\in\mr{Gr}(S,r)$. Also, $\partial_{s,r}\Phi_0^+$ is field disjoint. Our goal is to massage this to further guarantee the one added condition.

To this end, we consider another disjoint process governed by \cref{lem:master-disjointness}. Let $\{S_1,\ldots,S_{y'}\}$ be the positive $s$-spaces in $\Phi_0$ and $\{S_{y'+1},\ldots,S_y\}$ be the negative $s$-spaces.
\begin{itemize}
    \item $H$ is obtained by the following process. Take one copy of the construction in \cref{prop:subspace-exchange}, call its $s$-spaces $\Ups_0,\Ups_0'$. Then let $F$ be an $s$-space in $\Ups_0$, and $F_j$ for $j\le\qbinom{s}{r}_q$ be the $s$-spaces in $\Ups_0'$ intersecting $F$ in $r$ dimensions. Then glue $\qbinom{s}{r}_q$ additional copies of the construction from \cref{prop:subspace-exchange}, call the $s$-spaces $\Ups_j,\Ups_j'$, linearly disjointly along each $F_j$ so that $F_j\in\Ups_j$. Let $F_j'$ be the unique $s$-space in $\Ups_j'$ which contains $F\cap F_j$. (This is similar to the construction used in the proof of \cref{lem:rainbow-decomposition}.)
    \item $\mc{R}_\mr{init}=\chi(\mr{supp}(\partial_{s,r}(\Phi_0^+))\cup\mr{supp}(\partial_{s,r}(\Phi_0^-)))$;
    \item $x=y$ and $\phi_t$ is an arbitrary linear injection such that $\phi_t(F)=S_t$.
    \item $\mc{C}_t=\{\mr{ind}(R)\colon R\in\mr{Gr}(S_t,r)\}$;
    \item The basis $v^\ast$ is constructed in the following manner: start with a basis for $F$, then in order of $1\le j\le\qbinom{s}{r}_q$ extend to a basis for the span of $F_j'$. This is equivalent to extending from $F\cap F_j$ of dimension $r$ to $F_j'$ of dimension $s$, and the linear disjointness shows that these do not interfere with each other. Then extend this arbitrarily to a basis of $\mr{Vec}(H)$.
    \item $\mc{C}_{t,j}$ is defined as follows: if $v_j^\ast$ was constructed in the process of spanning some $F_i'$, then $\mc{C}_{t,j}=\{\mr{ind}(\phi_t(F\cap F_i))\}$. Otherwise, $\mc{C}_{t,j}=\emptyset$.
    \item $\theta$ is replaced by $2z\theta^{1/4}$.
\end{itemize}
By \cref{lem:chi-bounded} we know that $\chi(\partial_{s,r}\Phi_0^\pm)$ are $z\theta^{1/4}$-bounded hence whp the process runs to completion. Similar to prior analyses, we see that the ``new spaces'' produced by the output are whp $(\theta^{1/8}/2,L)$-field bounded, say, so the new decomposition we will obtain will ultimately be $(\theta^{1/8},L)$-field bounded.

We briefly describe, but do not fully write out, what the new decomposition $\Phi$ will be. Every $s$-space showing up in some $\partial_{s,r}e_{S_t}$ will be replaced by the following process: first replace this with the sum of $s$-spaces coming from $\Ups_0'$ minus those coming from $\Ups_0$ other than $S_t$; then for each space of $\Ups_0'$ which intersects $S_t$ in $r$ dimensions, which corresponds to some $\Ups_j,\Ups_j'$, we subsequently similarly replace it using $\Ups_j'$ and $\Ups_j$. This yields a sum and difference of $s$-spaces, and we check that the necessary conditions on the $\infty$-norm, field disjointness, etc., are preserved (using disjointness of the process, similar to the analyses in the proof of \cref{lem:final-rainbow}).

Finally, we consider the additional property that needs to be guaranteed. Similar to prior analyses, each new space will be rainbow by definition, but we need the additional property that if $S\in\mr{supp}(\Phi)$, we have $S\leqslant\bigcap_{R\in\mr{Gr}(S,r)}K_{i_R}$ if we write $i_R=\mr{ind}(R)$ for each $R\leqslant S$.

The key point is that the $\mc{C}_{t,j}$ will guarantee this for us. First note that the definition of $\mc{H}_t$ in \cref{lem:master-disjointness} shows that $\phi^\ast(H)\subseteq K_{\mr{ind}(R)}$ for each $R\in H\setminus H[F]$, which guarantees this property for almost all the $s$-spaces generated. However, for the $s$-spaces $S$ which are introduced and intersect $S_t$ in $r$ dimensions, we are precisely missing the property that $S\leqslant K_{\mr{ind}(S\cap S_t)}$. For such $S$ we can write $S=\phi_t^\ast(F_i')$ for some $1\le i\le\qbinom{s}{r}_q$. This is where the property from the $\mc{C}_{t,j}$ comes in: it provides the missing color which is $\mr{ind}(\phi_t(F\cap F_i))$. Note that $\phi_t(F\cap F_i)=\phi_t^\ast(F\cap F_i)=\phi_t^\ast(F\cap F_i')=S_t\cap S$. Since by definition $F_i'$ is spanned by $F\cap F_i$ and the $v_j^\ast$ which were assigned this color in $\mc{C}_{t,j}$, we see that the image of $F_i'$ under $\phi_t^\ast$, which is $S$, is fully contained in $K_{\mr{ind}(S\cap S_t)}$ as required.
\end{proof}

The final result of this section shows that an output of the previous lemma can be transformed so that every $s$-space is monochromatic (the constituent $r$-spaces are all in the same part of the template $G_{\mr{tem},i}$) and in fact configuration compatible (\cref{def:configuration-compatible}). Furthermore, we ensure that the sets of $s$-spaces are field bounded now (as opposed to just the underlying $r$-spaces). This will put us in position to apply \cref{prop:final-absorber}.
\begin{proposition}\label{prop:final-disjoint-monochromatic}
Given the setup of \cref{sub:setup,def:template}, we have the following as long as $c=c_{\ref{prop:final-disjoint-monochromatic}}(q,s)>0$, $q^{-cn}\le\tau\le c$, $d\ge d_{\ref{prop:final-disjoint-monochromatic}}(r)$, $\ell\ge\ell_{\ref{prop:final-disjoint-monochromatic}}(d,s)$, and $n$ is large. Whp over the randomness of the template, for any $\Phi\in\mb{Z}^{\mr{Gr}_q(n,s)}$ such that
\begin{itemize}
    \item $\partial_{s,r}\Phi,\partial_{s,r}\Phi^+\in\{0,1\}^{\mr{Gr}_q(n,r)}$ (and $\mr{supp}(\partial_{s,r}\Phi)\subseteq G_\mr{tem}$);
    \item $\partial_{s,r}\Phi^\pm$ are $(\theta,L)$-field bounded for some $q^{-cn}\le\theta\le(\tau/z)^{1/c}$;
    \item For every $S\in\mr{supp}(\Phi)$ there are distinct $i_R\in[z]$ such that $R\in G_{\mr{tem},i_R}$ for all $R\in\mr{Gr}(S,r)$, and furthermore $S\leqslant\bigcap_{R\in\mr{Gr}(S,r)}K_{i_R}$;
    \item $\partial_{s,r}\Phi^\pm$ are field disjoint;
\end{itemize}
there are $\mc{S}_1,\mc{S}_2\subseteq\mr{Gr}_q(n,s)$ with $\partial_{s,r}\Phi=\partial_{s,r}(\mc{S}_1-\mc{S}_2)$ such that $\snorm{\partial_{s,r}\mc{S}_1}_\infty\le 1$, $\mc{S}_1$ is $(\theta^{1/2},L)$-field bounded, and for each $S\in\mc{S}_1$ there is $i\in[z]$ such that (a) $\dim_L\mr{span}_L(\iota_i^{-1}(S))=s$ and (b) there is an $\mb{F}_q$-basis $b\in K_i^s$ of $S$ so that for every $\Pi\in\mr{Red}_q^{r\times s}$, we have for $R=\mr{span}_{\mb{F}_q}(\Pi b)$ that $\Pi_R=\Pi$, $b_R=\Pi b$, and $R\in G_{\mr{tem},i}$. Additionally, we guarantee that for $i\in[z]$ and distinct $S_1,S_2\in\mc{S}_1[G_{\mr{tem},i}]$, we have $\dim_L(\mr{span}_L(\iota_i^{-1}(S_1))\cap\mr{span}_L(\iota_i^{-1}(S_2)))<r$.
\end{proposition}
\begin{remark}
Note that our additional guarantee is slightly stronger than field boundedness, involving $s$-spaces, which is what is needed to apply \cref{prop:final-absorber}.
\end{remark}
\begin{proof}
We run a disjoint random process similar to \cref{lem:master-disjointness}, but instead of making the new $s$-spaces rainbow, we make them monochromatic. Additionally, we enforce configuration compatibility (\cref{def:configuration-compatible}). Let $k=k_{\ref{prop:subspace-exchange}}(s)$. Whp the template satisfies \cref{prop:template-extendable} for $h=k$. Further, given an $s$-space $F\leqslant\mb{F}_q^n$, Chernoff shows that with probability at least $1-\exp(-\Omega_{q,\ell}(z))$ over the random injections $\iota_1,\ldots,\iota_z$, there are at least $2z^{1/2}$ many $i\in[z]$ so that $F\leqslant K_i$ and $\dim_L\mr{span}_L(\iota_i^{-1}(F))=\dim F$. Taking a union bound, whp this property holds for all $s$-spaces. We thus condition on such an outcome and treat it as non-random.

Let $\mc{R}_\mr{init}=\mr{supp}(\partial_{s,r}\Phi^+)$. Let $\Phi^+\in\{0,1\}^{\mr{Gr}_q(n,s)}$ be supported on the set $\{S_1,\ldots,S_y\}$ with bases $\beta^{(1)},\ldots,\beta^{(y)}\in V^s$. By the condition in the third bullet, for each $t\in[y]$ and $\Pi\in\mr{Red}_q^{r\times s}$ the $r$-space $R_{t,\Pi}:=\mr{span}_{\mb{F}_q}(\Pi\beta^{(t)})$ is in some $G_{\mr{tem},i_{t,\Pi}}$, the $i_{t,\Pi}$ are distinct for fixed $t$, and $S_t\leqslant\bigcap_{\Pi\in\mr{Red}_q^{r\times s}}K_{i_{t,\Pi}}$. Let $\mc{I}_t=\{i_{t,\Pi}\colon\Pi\in\mr{Red}_q^{r\times s}\}$. Additionally, by the condition we assume on the random injections, there is a set $\mc{C}_t\subseteq[z]$ of size at least $z^{1/2}$ which is disjoint from $\mc{I}_t$ such that $S_t\leqslant K_i$ and $\dim_L\mr{span}_L(\iota_i^{-1}(S_t))=s$ for all $i\in\mc{C}_t$.

Let $H$ be the set of all $r$-spaces contained in $s$-spaces of $\Ups$ (with $\mr{Vec}(H)=\mb{F}_q^k$) as in \cref{prop:subspace-exchange}. Let $F$ be an arbitrary $s$-space in $\Ups$ with basis $\beta\in F^s$. Then for $\Pi\in\mr{Red}_q^{r\times s}$ let $F_\Pi$ be the unique $s$-space in $\Ups'$ intersecting $F$ in $R_\Pi:=\mr{span}_{\mb{F}_q}(\Pi\beta)$. Let $\mc{F}$ be the set of $s$-spaces in $\Ups'$ other than the $F_\Pi$. For $S\in\mc{F}$ let $x_S\in\mr{Vec}(H)^s$ be a basis of $S$. Finally let $\{v_1^\ast,\ldots,v_k^\ast\}$ be an $\mb{F}_q$-basis of $\mr{Vec}(H)$ such that $v_j^\ast=\beta_j$ for $j\in[s]$ form our basis for $F$.

For each $t\in[y]$ we choose an arbitrary linear injection $\phi_t\colon F\to V$ satisfying $\phi_t(\beta)=\beta^{(t)}$. Let $E_t:=(\phi_t,F,H)$ be an extension. Additionally, construct $\psi_t\colon H\setminus H[F]\to\mc{C}_t\cup\mc{I}_t\subseteq[z]$ with the following property: for $\Pi\in\mr{Red}_q^{r\times s}$ and $R\in\mr{Gr}(F_\Pi,r)\setminus\{R_\Pi\}$ we have $\psi_t(R)=i_{t,\Pi}$ and for $R\in\mr{Gr}(S,r)$ with $S\in\mc{F}$ we have $\psi_t(R)=i_{t,S}$ where the $i_{t,S}\in\mc{C}_t$ take on distinct values as $S\in\mc{F}$ varies. Thus all values of $\psi_t$ are distinct modulo the information of which $S\in\Ups'$ an $r$-space $R\in H\setminus H[F]$ is contained in. Also, every color class has size at most $\qbinom{s}{r}_q$. Furthermore, this is constructed in a way such that if we hypothetically perform a flip of this type transforming $\partial_{s,r}e_{S_t}$ into an analogous sum, we will obtain monochromatic $s$-spaces for the terms with positive coefficients (coming from $\Ups'$). For convenience write $i_{t,S}=i_{t,\Pi}$ in the case $S=F_\Pi$.

Finally, write $b_{t,\Pi}=b_{R_{t,\Pi}}$ and $\Pi_{t,\Pi}=\Pi_{R_{t,\Pi}}$ and consider a basis $x_{t,\Pi}\in\mr{Vec}(H)^s$ of $F_\Pi$ such that $\Pi_{t,\Pi}x_{t,\Pi}$ spans $R_\Pi$ and furthermore $\phi_t(\Pi_{t,\Pi}x_{t,\Pi})=b_{t,\Pi}$ (which trivially can be seen to exist since $\Pi_{t,\Pi}\in\mb{F}_q^{r\times s}$ has rank $r$).

We now run the following process after setting $\mc{R}_0:=\emptyset$ and $\mc{S}_0:=\emptyset$. The constant $C$ will be chosen suitably later.
\begin{itemize}
    \item Let $E_t:=(\phi_t,F,H)$ be an extension and $\mc{H}_t'$ be the set of $\phi^\ast\in\mc{X}_{E_t}(G_\mr{tem})$ such that (a) $\phi^\ast(R)\in G_{\mr{tem},\psi_t(R)}$ for all $R\in H$ (equivalently, $\mr{ind}(\phi^\ast(R))=\psi_t(R)$), (b) for every $S\in\phi_t(\Ups')$ we have $\dim_L\mr{span}_L(\iota_{i_{t,S}}^{-1}(S))=s$, (c) for every $\Pi\in\mr{Red}_q^{r\times s}$ and $S=\phi^\ast(F_\Pi)$ and any $\Pi'\in\mr{Red}_q^{r\times s}$ we have for $R=\mr{span}_{\mb{F}_q}(\Pi'x_{t,\Pi})$ that $b_{\phi^\ast(R)}=\phi^\ast(\Pi'x_{t,\Pi})$ and $\Pi_{\phi^\ast(R)}=\Pi'$, (d) for every $S\in\phi^\ast(\mc{F})$ and $\Pi\in\mr{Red}_q^{r\times s}$ we have for $R=\mr{span}_{\mb{F}_q}(\Pi x_S)$ that $b_{\phi^\ast(R)}=\phi^\ast(\Pi x_S)$ and $\Pi_{\phi^\ast(R)}=\Pi$, and (e) for all $S\in\Ups'$ and $i=i_{t,S}$, we have $\phi^\ast(\mr{Vec}(H))\leqslant K_i$ and $\dim_L\mr{span}_L(\iota_i^{-1}(\phi^\ast(\mr{Vec}(H))))-\dim_L\mr{span}_L(\iota_i^{-1}(\phi_t(F)))=k-s$.
    \item Let $\mc{H}_t$ be the set of $\phi^\ast\in\mc{H}_t'$ such that (i) for all $R\in H\setminus H[F]$, we have $\phi^\ast(R)\notin\mc{R}_\mr{init}\cup\mc{R}_{t-1}$, (ii) for all $S\in\Ups'$ and $S'\in\mc{S}_{t-1}$, if $\phi^\ast(S),S'\in\mr{Gr}_q(n,s)[G_{\mr{tem},i}]$ for some $i\in[z]$ then $\dim_L(\mr{span}_L(\iota_i^{-1}(\phi^\ast(S)))\cap\mr{span}_L(\iota_i^{-1}(S')))<r$, and (iii) for all $S\in\Ups'$ and $R\in\mc{R}_\mr{init}$, if $\mr{Gr}(\phi^\ast(S),r)\cup\{R\}\subseteq G_{\mr{tem},i}$ for some $i\in[z]$ then $\mr{span}_L(\iota_i^{-1}(R))$ is not contained in $\mr{span}_L(\iota_i^{-1}(\phi^\ast(S)))$ unless $R=\phi_t(F\cap S)$. (We remark that by (a) of the previous bullet, it suffices to consider $i=i_{t,S}$ in (i,iii); additionally, (iii) functions similarly to the role of applying $\chi$ to $\mc{R}_\mr{init}$ in \cref{lem:master-disjointness}.)
    \item If $|\mc{H}_t|<(\tau/z)^Cq^{vn}$, where $v=\dim\mr{Vec}(H)-\dim F=v_{E_t}$, we call the process \emph{failed}, let $\phi_t^\ast=\phi_{t+1}^\ast=\cdots=\phi_y^\ast:=\ast$, a special wildcard value, let $\mc{R}_t=\cdots=\mc{R}_y:=\mc{R}_{t-1}$, let $\mc{S}_t=\cdots=\mc{S}_y:=\mc{S}_{t-1}$, and stop the iteration.
    \item Otherwise, we sample $\phi_t^\ast\sim\mr{Unif}(\mc{H}_t)$, let $\mc{R}_t:=\mc{R}_{t-1}\cup\phi_t^\ast(H\setminus H[F])$ as a set, let $\mc{S}_t:=\mc{S}_{t-1}\cup\phi_t^\ast(\Ups')$, and continue.
\end{itemize}
As an important point, note that in condition (c) above for the definition of $\mc{H}_t'$, the case $\Pi'=\Pi_{t,\Pi}$ implies $R=R_\Pi\leqslant F$ hence $\phi^\ast(R)=\phi_t(R)\leqslant S_t$, so it really is a condition about the base space. But $b_{\phi_t(R)}=b_{R_{t,\Pi}}=\phi_t(\Pi_{t,\Pi}x_{t,\Pi})$ by definition, and similarly $\Pi_{\phi^\ast(R)}=\Pi_{t,\Pi}$, so the condition is actually redundant. Thus, furthermore note that by \cref{prop:template-extendable}, for all $t\in[y]$ we have
\begin{equation}\label{eq:many-embeddings-final}
|\mc{H}_t'|\ge(\tau/z)^{C/2}q^{vn}
\end{equation}
if $C$ is chosen appropriately. To verify \cref{prop:template-extendable} is even applicable, we must make sure that $S_t$ is in the intersection of various fields $K_i$ coming from the defined coloring function $\psi_t$ (condition \ref{prop:template-extendable}(a)). For the range elements $i_{t,S}$ coming from $r$-spaces in some $S\in\mc{F}$ this is by definition of $\mc{C}_t$, but for those of the form $i_{t,\Pi}$ for $\Pi\in\mr{Red}_q^{r\times s}$ we are using the third bullet point of the given conditions, which implies $S_t\leqslant\bigcap_{\Pi\in\mr{Red}_q^{r\times s}}K_{i_{t,\Pi}}$ (this is the reason for the extra \cref{lem:final-color-consistent}). Additionally notice that conditions (a,c,d,e) in the definition of $\mc{H}_t'$ fit within the framework of \cref{prop:template-extendable}, other than the condition on difference of dimensions in (e). Furthermore, that part of (e) as well as condition (b) have $O_{\ell,q,s}(q^{(v-1)n})$ violations: for (e), choose the $L$-dependence between $v_{s+1}^\ast,\ldots,v_k^\ast$ and $\{\beta_1,\ldots,\beta_s\}$ (after appropriate $\iota^{-1}\phi^\ast$; this nontrivially involves the former list) which causes the difference in dimensions to collapse and then count choices of $v-1$ remaining free parameters (also, (b) is implied by (e)). A similar argument where we use that $L$-degeneracies occur infrequently is performed in the deduction of \cref{prop:template-extendable} from \cref{lem:iterative-color-embedding}.

Finally, we verify that \ref{prop:template-extendable}(b) holds. Recall that the sets $\mc{C}_t$ are also such that $\dim_L\mr{span}_L(\iota_i^{-1}(S_t))=s$ for all $i\in\mc{C}_t$, which implies that every $\mb{F}_q$-subspace of $S_t$ of dimension $u$ will extend to $u$ dimensions over $L$ with respect to such $\iota_i$. Thus the condition is easily verified for $R\in H\setminus H[F]$ so that $\psi_t(R)\in\mc{C}_t$. The other possibility is $\psi_t(R)=i_{t,\Pi}$ for some $\Pi\in\mr{Red}_q^{r\times s}$, in which case we must have $R\leqslant F_\Pi$ hence $R\cap F\subseteq R\cap F_\Pi\cap F\leqslant R\cap R_\Pi$ (the last inclusion by the second bullet of \cref{prop:subspace-exchange}). But we have $R_{t,\Pi}\in G_{\mr{tem},i_{t,\Pi}}$, which means $\dim_L\mr{span}_L(\iota_{i_{t,\Pi}}^{-1}(R_{t,\Pi})=r$. Along with $\phi_t(R_\Pi)=R_{t,\Pi}$, and a similar argument as above, this completes verification of \ref{prop:template-extendable}(b) and hence justifies \cref{eq:many-embeddings-final}.

The rest of the argument is now a similar Bernoulli comparison and Chernoff analysis as in the proof of \cref{lem:master-disjointness}, with modifications due to certain conditions and desired outputs of the process that are inherently $s$-dimensional in nature.

Similarly to the proof of \cref{lem:master-disjointness}, for any $\mc{F}\subseteq\mr{Gr}_q(n,r)$ if $X_\mc{F}^{(r)}=\#\{(R,R')\in\mc{F}\times\mc{R}_y\colon R\in\chi(R')\}$ then
\begin{equation}\label{eq:2-disjoint-family-bounded}
\mb{P}[X_\mc{F}^{(r)}\ge\theta^{2/3}|\mc{F}|]\le\exp(-\theta|\mc{F}|).
\end{equation}
(Specifically, we similarly bound the expected probability that each step includes certain $r$-spaces and use \cref{lem:bernoulli-domination,lem:chernoff}.) Now we prove an analogous bound but in an $s$-dimensional sense. We will then apply this to show that the process runs to completion whp, and that the necessary field boundedness holds for the output.

Let $\mc{F}\subseteq\mr{Gr}_q(n,s)$ be a collection of $s$-spaces. We say that $\mc{F}$ is \emph{unforced} if for any $r$-space $R\in\mr{supp}(\partial_{s,r}\Phi^+)$, we have
\begin{equation}\label{eq:unforced}
\mc{F}\cap\{S\in\mr{Gr}_q(n,s)\colon S\geqslant R\}=\emptyset,
\end{equation}
that is, no $s$-space of $\mc{F}$ contains any $R\in\mr{supp}(\partial_{s,r}\Phi^+)$.

Let $X_\mc{F}=\sum_{t=1}^yX_t$ where $X_t$ is the number of choices of  $S\in\phi_t^\ast(\Ups')$ so that $S\in\mc{F}$. If $Y_t=\mbm{1}_{X_t\neq 0}$ we have $X_t\le q^{sk}Y_t$ so $Y=\sum_{t=1}^yY_t\ge q^{-sk}X$. For $S\in\mr{Gr}_q(n,s)$ and $t\in[y]$ let $\mc{H}_{S,t}=\{\phi\in\mc{H}_t'\colon S\in\phi(\Ups')\}$. Defining
\[\mu_t=\sum_{S\in\mc{F}}|\mc{H}_{S,t}|(z/\tau)^Cq^{-vn}\]
and $\mu=\sum_{t=1}^y\mu_t$, we can (similarly to the proof of \cref{lem:master-disjointness}) use \cref{lem:bernoulli-domination} and \cref{lem:chernoff} to show
\begin{equation}\label{eq:disjoint-s-family-bounded}
\mb{P}[X_\mc{F}\ge\theta^{2/3}q^{(r-s)n}|\mc{F}|]\le\exp(-\theta q^{(r-s)n}|\mc{F}|)\text{ for unforced }\mc{F},
\end{equation}
as long as we can prove $2\mu\le\theta^{3/4}q^{(r-s)n}|\mc{F}|$. To establish this, note that for fixed $S\in\mc{F}$ we have
\begin{align}
\sum_{t=1}^y|\mc{H}_{S,t}|&=\sum_{u=0}^{r-1}\sum_{\substack{t\in[y]\\\dim(S\cap\phi_t(F))=u}}|\mc{H}_{S,t}|\le\sum_{u=0}^{r-1}\bigg(\sum_{\substack{S_0\in\Ups'\\\dim(S_0\cap F)=u}}\sum_{\substack{t\in[y]\\\phi_t(S_0\cap F)\leqslant S}}\#\{\phi\in\mc{H}_t'\colon S=\phi(S_0)\}\bigg)\notag\\
&\le\sum_{u=0}^{r-1}O_{k,q,s}(z\theta q^{(r-u)n})\cdot O_{q,s}(q^{(\dim\mr{Vec}(H)-(\dim F+s-u))n})\le O_{q,s}(z\theta q^{(r-s+v)n}).\label{eq:master-disjointness-s-mean}
\end{align}
The equality and first inequality follow since for every $\phi\in\mc{H}_{S,t}$, we have $S=\phi(S_0)$ for some $S_0\in\Ups'$ which intersects $F$ in at most $r$ dimensions by \cref{prop:subspace-exchange}, call it $u=\dim(S_0\cap F)\in\{0,\ldots,r\}$, and $\phi_t(S_0\cap F)\leqslant S$ easily follows. Additionally, the case $u=r$ is ruled out by \cref{eq:unforced}: if we have such a situation then $\phi_t(S_0\cap F)$ has dimension $r$ and is inside $\mr{supp}(\partial_{s,r}\Phi^+)$, but is contained within some $S\in\mc{F}$, violating the unforcedness condition.

The inequality in the second line of \cref{eq:master-disjointness-s-mean} is proven by seeing there are at most $O_{k,q,s}(1)\cdot z\theta q^{(r-u)n}$ terms in the inner double sum by counting ways to choose $S_0$ and then using $(\theta,L)$-field boundedness (hence $z\theta$-boundedness from \cref{lem:strength-bounded}) of $\partial_{s,r}\{S_1,\ldots,S_y\}$ to count how many $t\in[y]$ satisfy $\phi_t(S_0\cap F)\leqslant S$. Then note that $\#\{\phi\in\mc{H}_t'\colon S=\phi(S_0)\}=O_{q,s}(q^{(\dim\mr{Vec}(H)-(\dim F+s-u))n})$ since this is bounded by the number of extensions $\phi\in\mc{X}_{E_t}(G)$ which map $S_0$ to $S$ (up to some permutation which has $O_{q,s}(1)$ choices), and then the remaining number of free dimensions to embed is $\dim\mr{Vec}(H)-(\dim F+s-u)$ since $\dim\mr{span}_{\mb{F}_q}(S_0\cup F)=\dim F+s-\dim(S_0\cap F)=\dim F+s-u$.

Now summing \cref{eq:master-disjointness-s-mean} over $S\in\mc{F}$ shows that $\mu\le O_{k,q,s}(1)\cdot z\theta(z/\tau)^Cq^{(r-s)n}|\mc{F}|$, which demonstrates the desired inequality $2\mu\le\theta^{3/4}|\mc{F}|$ as long as $1/c\ge 4C+2$ is chosen appropriately (using $\theta\le(\tau/z)^{1/c}$). Thus \cref{eq:disjoint-s-family-bounded} indeed holds. Similarly, for arbitrary $\mc{F}\subseteq\mr{Gr}_q(n,s)$ (not necessarily unforced) we can derive
\begin{equation}\label{eq:general-family-bounded}
\mb{P}[X_\mc{F}\ge z^2q^{(r-s)n}|\mc{F}|]\le\exp(-q^{(r-s)n}|\mc{F}|).
\end{equation}
Indeed, note that we sacrifice the ability to save a factor of $\theta$ by ruling out the case $u=r$ in the derivation of \cref{eq:master-disjointness-s-mean}, but otherwise the proof is analogous. (We use just $\snorm{\partial_{s,r}\Phi^+}_\infty\le 1$.)

Now we extract the field boundedness of $\mc{S}_y=\bigcup_{t=1}^y\phi_t^\ast(\Ups')$, taking a union bound over appropriate choices of $\mc{F}$ using \cref{eq:disjoint-s-family-bounded} and \cref{eq:general-family-bounded}. Specifically, given an $L$-space $Q^\ast\in\mr{Gr}_L(K,r-1)$ and an index $i\in[z]$, we consider all $s$-spaces $S\leqslant K_i\leqslant V$ such that (a) $\mr{span}_L(\iota_i^{-1}(S))$ contains $Q^\ast$, and (b) $S$ does not contain any $r$-space in $\mr{supp}(\partial_{s,r}\Phi^+)$. This family has say at most $\theta^{2/3}q^n$ many $s$-spaces from the union of the $\phi_t^\ast(\Ups')$ whp by \cref{eq:disjoint-s-family-bounded} if $c$ is small enough.

Then we additionally consider $s$-spaces $S\leqslant K_i$ satisfying (a) and such that (b) fails. By $(\theta,L)$-field boundedness of $\partial_{s,r}\Phi^+$, we see that there are at most say $\theta q^{(r-u)n}$ choices of $R\in\mr{supp}(\partial_{s,r}\Phi^+)$ so that $\dim_L(\mr{span}_L(\iota_i^{-1}(R))\cap Q^\ast)=u\in\{0,\ldots,r-1\}$. The number of $s$-spaces satisfying (a) and failing (b) with some fixed value $u$ is then at most say $z^2\cdot z\theta q^{(r-u)n}\cdot q^{(r-s)n}q^{(s-(2r-1-u))n}$ whp: we use \cref{eq:general-family-bounded} applied to the collection $\mc{F}$ of $s$-spaces $S\leqslant K_i$ such that $\iota_i^{-1}(S)$ contains $Q^\ast$ as well as $\iota_i^{-1}(R)$ for some such $R$, of which there are at most $\theta q^{(r-u)n}$. Summing over $u$, we obtain a contribution of at most $r\theta z^3q^n\le\theta^{2/3}q^n$ as well if $c$ is small enough.

Summing, we obtain at most $2\theta^{2/3}q^n\le\theta^{1/2}q^n$ total $s$-spaces which contain $Q^\ast$ when extended to $L$ after pulling back with respect to $\iota_i$. We only need to take a union bound of size $q^{O_s(n)}$ to run this argument, so indeed whp we have the desired $(\theta^{1/2},L)$-field boundedness of $\mc{S}_y$, as desired.

Thus, in particular when we apply $\partial_{s,r}e_F=\sum_{S\in\Ups'}\partial_{s,r}e_S-\sum_{S\in\Ups\setminus\{F\}}\partial_{s,r}e_S$ (using the map $\phi_t^\ast$ for all $t\in[y]$) to the right side of $\partial_{s,r}\Phi=\partial_{s,r}\Phi^++\partial_{s,r}\Phi^-$, we see that we obtain a signed decomposition (there are no repetitions due to disjointness of the process and \cref{prop:subspace-exchange}) where the positive spaces $\mc{S}_1\subseteq\mr{Gr}_q(n,s)$, which corresponds to everything coming from each $\phi_t^\ast(\Ups')$, satisfy all the desired properties (as long as the process runs to completion whp, say, which will be shown below). For instance, the conditions corresponding to configuration compatibility (\cref{def:configuration-compatible}) come from conditions (b), (c), and (d) on $\mc{H}_t'$. One nontrivial verification is the final property regarding the dimension of $L$ intersections: the definition of the process, specifically the second bullet point which defines $\mc{H}_t$, ensures that any violating $i\in[z]$ and $S_1,S_2\in\mc{S}_1[G_{\mr{tem},i}]$ must be introduced at the same time step $t$, and the $s$-spaces introduced at the same time $t$ must satisfy this dimension intersection property due to condition (e) in the first bullet point of the definition of the process (and since each $s$-space in $\Ups'$ is a different color).

Finally, to complete the argument we show that the process runs to completion whp. In the event it fails at time $t\in[y]$, call this $\mc{E}_t$, we have $|\mc{H}_t'\setminus\mc{H}_t|\ge((\tau/z)^{C/2}-(\tau/z)^C)q^{vn}\ge(\tau/z)^Cq^{vn}$ by \cref{eq:many-embeddings-final}. Furthermore,
\[\mc{H}_t'\setminus\mc{H}_t\subseteq\mc{B}_2\cup\mc{B}_3\cup\bigcup_{R\in H\setminus H[F]}(\{\phi^\ast\in\mc{X}_{E_t}(G)\colon\phi^\ast(R)\in\mc{R}_\mr{init}\}\cup\{\phi^\ast\in\mc{X}_{E_t}(G)\colon\phi^\ast(R)\in\chi(\mc{R}_y)\})\]
where $\mc{B}_2$ is the set of $\phi^\ast\in\mc{H}_t'$ failing the condition (ii) from the second bullet point in the definition of the process above and $\mc{B}_3$ is the set failing (iii).

Similarly to the proof of \cref{lem:master-disjointness}, using \cref{eq:2-disjoint-family-bounded} we can bound the size of the third set (the part other than $\mc{B}_2\cup\mc{B}_3$) by $O_{q,s}(\theta^{2/3}q^{vn})\le(\tau/z)^Cq^{vn}/3$, say under an event that holds whp. Thus it remains to understand embeddings $\phi^\ast\in\mc{B}_2\cup\mc{B}_3$. For $\phi^\ast\in\mc{B}_2\setminus\mc{B}_3$, there exist $S\in\Ups'$, $S'\in\mc{S}_{t-1}\subseteq\mc{S}_y$, and $i=i_{t,S}\in[z]$, witnessing the failure of (ii). This implies the existence of an $L$-space $R^\ast$ of dimension $r$ with $R^\ast\leqslant\mr{span}_L(\iota_i^{-1}(\phi^\ast(S)))\cap\mr{span}_L(\iota_i^{-1}(S'))$. Let $Q^\ast:=R^\ast\cap\mr{span}_L(\iota_i^{-1}(\phi_t(F)\cap K_i))$ in this situation and $\dim_L(Q^\ast)=u\in\{0,\ldots,r\}$. We have
\[Q^\ast\leqslant\mr{span}_L(\iota_i^{-1}(\phi_t(F)\cap K_i)),\mr{span}_L(\iota_i^{-1}(\phi^\ast(S))),\mr{span}_L(\iota_i^{-1}(S')).\]
By (e) of the definition of $\mc{H}_t'$, we see that the first two of these spaces have intersection $W^\ast:=\mr{span}_L(\iota_i^{-1}(\phi_t(F)\cap\phi^\ast(S)\cap K_i))=\mr{span}_L(\iota_i^{-1}(\phi_t(F\cap S)))$, whose dimension we call $d$. So $u\le d$ as a consequence of $Q^\ast\leqslant W^\ast$.

Now suppose $u<r$. There are at most $r(r+1)z$ ways to choose $u,d,i$, and given $u,d,i$ there are $O_{\ell,q,s}(1)$ ways to choose $Q^\ast$ since it is inside $\iota_i^{-1}(\phi_t(F)\cap K_i)$. Then given $Q^\ast$, the number of choices of $S'\in\mc{S}_y$ satisfying $S'\in\mr{Gr}_q(n,s)[G_{\mr{tem},i}]$ and $Q^\ast\leqslant\mr{span}_L(\iota_i^{-1}(S'))$ is at most $\theta^{1/2}q^{(r-u)n}$ by $(\theta^{1/2},L)$-field boundedness of $\mc{S}_y$ (which holds whp) and $u\le r-1$. Then after choosing $S'\in\mc{S}_y$ there are $O_{\ell,q,s}(1)$ choices for $R^\ast$. There are $O_{q,s}(1)$ choices for $S\in\Ups'$. Now, we know $\mr{span}_L(\iota_i^{-1}(\phi^\ast(S)))$ contains $R^\ast$ of dimension $r$ and contains $W^\ast=\mr{span}_L(\iota_i^{-1}(\phi_t(F\cap S)))$ of dimension $d$. These $L$-spaces $R^\ast,W^\ast$ both contain $Q^\ast$ and in fact intersect precisely in $Q^\ast$. Thus the $L$-span of their union has dimension $r+d-u$. Therefore there are at most $q^{(s-(r+d-u))n}$ choices for $\phi^\ast(S)$ now, and then at most $q^{(k-(2s-\dim(F\cap S)))n}=q^{(v-s+\dim(F\cap S))n}$ choices to complete the embedding. Overall, there are at most $O_{\ell,q,s}(z\theta^{1/2}q^{(r-u)n}q^{(s-(r+d-u))n}q^{(v-s+\dim(F\cap S))n})\le\theta^{1/3}q^{vn}/3$ total choices, using that $\dim(F\cap S)=d$. The fact $\dim(F\cap S)=d$ is nontrivial: note that $\phi_t(F\cap S)$ extends to $\dim(F\cap S)$ dimensions over $L$ with respect to any index $i\in\mc{C}_t$ by definition of $\mc{C}_t$, and for cases where $i=i_{t,\Pi}$ we have $S=F_\Pi$ and $\phi_t(F\cap S)=R$ with $R\in G_{\mr{tem},i_{t,\Pi}}$, so the $L$-extension with respect to $i_{t,\Pi}$ again has full dimension.

Now suppose $u=r$. This means $Q^\ast=R^\ast$ has $L$-dimension $r$. Recalling that $Q^\ast\leqslant W^\ast$, we see that $Q^\ast=W^\ast$ and $\dim(F\cap S)=r$ hence $\mr{span}_L(\iota_i^{-1}(\phi_t(F\cap S)))\leqslant\mr{span}_L(\iota_i^{-1}(S'))$. But then we deduce $S=F_\Pi$, $i=i_{t,\Pi}$, and $\phi_t(F\cap S)=R$ for $R\in G_{\mr{tem},i}\cap\mc{R}_\mr{init}$. This condition then violates (iii) regarding possible interactions between $S'\in\phi^\ast(\Ups')$ for $\phi^\ast\in\mc{H}_u'$ for all $u\in[y]$ and $\mc{R}_\mr{init}$. So there are no choices in this scenario, recalling we were considering $\phi^\ast\in\mc{B}_2\setminus\mc{B}_3$.

Finally, consider $\phi^\ast\in\mc{B}_3$. Let $S\in\Ups'$ and $R\in\mc{R}_\mr{init}\cap G_{\mr{tem},i}$ be the violators, and let $R^\ast=\mr{span}_L(\iota_i^{-1}(R))$ and $Q^\ast=R^\ast\cap\mr{span}_L(\iota_i^{-1}(\phi_t(F)\cap K_i))$ with $\dim_L(Q^\ast)=u$. By violation of (iii), we have $\mr{span}_L(\iota_i^{-1}(R))\leqslant\mr{span}_L(\iota_i^{-1}(\phi^\ast(S)))$. A similar argument to above, but using the deterministic fact that the $r$-dimensional $q$-system $\mc{R}_\mr{init}$ is $(\theta,L)$-field bounded (instead of the random event that $\mc{S}_y$ is $(\theta^{1/2},L)$-field bounded) shows that we have a contribution of at most $\theta^{1/3}q^{vn}/3$ possibilities, unless $u=r$. In the case $u=r$ we have $Q^\ast=R^\ast$ and then we deduce
\begin{align*}
\mr{span}_L(\iota_i^{-1}(R))&=R^\ast=Q^\ast=R^\ast\cap\mr{span}_L(\iota_i^{-1}(\phi_t(F)\cap K_i))\\
&\leqslant\mr{span}_L(\iota_i^{-1}(\phi_t(F)\cap K_i))\cap\mr{span}_L(\iota_i^{-1}(\phi^\ast(S))),
\end{align*}
the last containment from the above violation of (iii). Recall
\[\mr{span}_L(\iota_i^{-1}(\phi_t(F)\cap K_i))\cap\mr{span}_L(\iota_i^{-1}(\phi^\ast(S)))=\mr{span}_L(\iota_i^{-1}(\phi_t(F\cap S)))\]
due to (e) in the definition of the process (see e.g.~the discussion involving $W^\ast$ above). Thus $\mr{span}_L(\iota_i^{-1}(R))\leqslant\mr{span}_L(\iota_i^{-1}(\phi_t(F\cap S)))$. This means $\dim(F\cap S)=r$ so say $S=F_\Pi$ and then we have for $R'=\phi_t(F\cap S)\in\mc{R}_\mr{init}$ that $\mr{span}_L(\iota_i^{-1}(R))\leqslant\mr{span}_L(\iota_i^{-1}(R'))$. Furthermore, the violation of (iii) guarantees that $R'\neq R$ and $R,R'\in\mr{Gr}(\phi^\ast(S),r)\cup\{R\}\subseteq G_{\mr{tem},i}$. That is, $R,R'$ are distinct, in the same part of the template, and satisfy $\mr{span}_L(\iota_i^{-1}(R))\leqslant\mr{span}_L(\iota_i^{-1}(R'))$. Thus, this violates field disjointness of $\partial_{s,r}\Phi^+$.

Overall, this shows $|\mc{H}_t'\setminus\mc{H}_t|\le\theta^{1/3}q^{vn}$ given the event that $\mc{S}_y$ is $(\theta^{1/2},L)$-field bounded. For $c$ chosen small enough so that $1/c\ge 3C+1$, this precludes $\mc{E}_t$ which implies $|\mc{H}_t'\setminus\mc{H}_t|\ge(\tau/z)^Cq^{vn}$ as shown earlier. That is, under $(\theta^{1/2},L)$-field boundedness of $\mc{S}_y$ (which holds whp) we have that the process runs to completion. We are done.
\end{proof}

\section{Final proof and counting}\label{sec:final-counting}
Finally, we put together the pieces to demonstrate \cref{thm:main}.
\begin{proof}[Proof of \cref{thm:main}]
We are given $s>r\ge 1$ and $n$ satisfying various divisibility conditions from \cref{thm:main} (see \cref{eq:divisibility}). We will ultimately choose parameters in the following way:
\[q,s\ll d\ll\ell\ll 1/\eta\ll1/\zeta\ll n.\]
Consider the setup \cref{sub:setup} and let $J=\sum_{R\in\mr{Gr}(V,r)}e_R$ (by abuse we can think of $J$ as just being the $r$-system $G=\mr{Gr}(V,r)$). By \cref{thm:integral-lattice} we have $J\in\mc{L}=\partial_{s,r}\mb{Z}^{\mr{Gr}_q(n,s)}$ due to the given divisibility constraints on $n$. Choose $d$ large in terms of $s$, and then $\ell$ large in terms of $d,s$ so that various lemmas go through and so that there exist some $N=N_\mr{tem},x^\ast=x_\mr{abs}^\ast$ as defined in \cref{sub:setup} which are appropriately jointly generic (this follows similarly to the start of the proof of \cref{lem:generic}). Also, let $\tau = q^{-\zeta n}$ ($\tau$ is taken to be exponentially small in $n$ specifically in order to prove \cref{cor:counting}). In particular, by \cref{lem:template-basic} we obtain $\mc{S}_\mr{tem}$, an $s$-space covering of the template $r$-spaces by \cref{def:template}.

First apply \cref{prop:approximate-covering} to obtain $\eta$ depending only on $q,s$ (not $\zeta$) and to obtain $\mc{S}_\mr{approx}$ where $\Phi_0=\mc{S}_\mr{tem}\cup\mc{S}_\mr{approx}$ covers some collection of $r$-spaces exactly once and leaves a remainder $J_1=J-\partial_{s,r}\Phi_0\in\{0,1\}^{\mr{Gr}_q(n,r)}$ which is $q^{-\eta n}$-bounded with $\mr{supp}(J_1)\cap G_\mr{tem}=\emptyset$.

Next apply \cref{lem:final-covering} on $\mc{R}=J_1$ and $\theta=q^{-\eta n}$. We obtain $\Phi_1\in\{0,1\}^{\mr{Gr}_q(n,s)}$ which covers some collection of $r$-spaces exactly once, including all of $J_1$, and the remainder $J_2=\partial_{s,r}\Phi_1-J_1$ has support within $G_\mr{tem}$ and is $(q^{-\eta n/2},L)$-field bounded. This in particular shows that $J_2=\partial_{s,r}(\Phi_0+\Phi_1)-J$ is a set of $r$-spaces, which we think of as the ``spill'', which we additionally know is field disjoint.

Now apply \cref{lem:final-color-consistent} to $\mc{R}=J_2$ and $\theta=q^{-\eta n/2}$, to obtain some $\Phi_2$ with $J_2=\partial_{s,r}\Phi_2$ and satisfying various additional properties including that $\partial_{s,r}\Phi_2^\pm$ are field disjoint and $(q^{-\eta n/16},L)$-field bounded. Then apply \cref{prop:final-disjoint-monochromatic} to $\Phi=\Phi_2$ and $\theta=q^{-\eta n/16}$ (which is valid due to aforementioned additional properties) to find $\Phi_3,\Phi_4\in\{0,1\}^{\mr{Gr}_q(n,s)}$ (coming from $\mc{S}_1,\mc{S}_2$) with $\partial_{s,r}(\Phi_3-\Phi_4)=\partial_{s,r}\Phi_2=J_2$ and such that:
\begin{itemize}
    \item $\snorm{\partial_{s,r}\Phi_3}_\infty\le 1$;
    \item $\Phi_3$ is $(q^{-\eta n/32},L)$-field bounded;
    \item For each $S\in\mr{supp}(\Phi_3)$ there is $i\in[z]$ such that (a) $\dim_L\mr{span}_L(\iota_i^{-1}(S))=s$ and (b) there is an $\mb{F}_q$-basis $b\in K_i^s$ of $S$ so that for every $\Pi\in\mr{Red}_q^{r\times s}$, we have for $R=\mr{span}_{\mb{F}_q}(\Pi b)$ that $\Pi_R=\Pi$, $b_R=\Pi b$, and $R\in G_{\mr{tem},i}$;
    \item For $i\in[z]$ and distinct $S_1,S_2\in\mc{S}_1[G_{\mr{tem},i}]$, $\dim_L(\mr{span}_L(\iota_i^{-1}(S_1))\cap\mr{span}_L(\iota_i^{-1}(S_2)))<r$.
\end{itemize}
Finally by \cref{prop:final-absorber} applied to $\Phi=\Phi_3$ and $\theta=q^{-\eta n/32}$ (note $\theta$ is much smaller than $\tau$) we can find $\Phi_5,\Phi_6\in\{0,1\}^{\mr{Gr}_q(n,s)}$ with $\mr{supp}(\Phi_6)\subseteq\mc{S}_\mr{tem}$ and $\partial_{s,r}\Phi_3=\partial_{s,r}(\Phi_6-\Phi_5)$. Here condition \ref{prop:final-absorber}(c) comes directly from the fourth bullet point above. Putting everything together, we have
\begin{align*}
J&=\partial_{s,r}\Phi_0+J_1=\partial_{s,r}(\Phi_0+\Phi_1)-J_2\\
&=\partial_{s,r}(\Phi_0+\Phi_1)-\partial_{s,r}(\Phi_3-\Phi_4)=\partial_{s,r}(\Phi_0+\Phi_1+\Phi_4)-\partial_{s,r}\Phi_3\\
&=\partial_{s,r}(\Phi_0+\Phi_1+\Phi_4+\Phi_5-\Phi_6)=\partial_{s,r}((\Phi_0-\Phi_6)+\Phi_1+\Phi_4+\Phi_5).
\end{align*}
Finally, $\Phi_0$ contains $\mc{S}_\mr{tem}$ and $\Phi_6$ is contained in $\mc{S}_\mr{tem}$ so each of $\Phi_0-\Phi_6,\Phi_1,\Phi_4,\Phi_5\in\mb{Z}_{\ge 0}^{\mr{Gr}_q(n,s)}$. This implies that they indeed form a $(n,s,r)_q$-design and we are done.
\end{proof}

We briefly explain how to extract the counting result \cref{cor:counting} from the proof given.

\begin{proof}[Proof sketch of \cref{cor:counting}]
For the lower bound, consider the proof of \cref{thm:main}. Note that after planting the template and applying \cref{lem:template-boosted}, we applied results of Ehard, Glock, and Joos \cite{EGJ20} to extract a large approximate covering $\mc{S}_\mr{approx}$ and then proved that given such a suitably bounded hypergraph matching that the remainder can be completed whp. The completed system, furthermore, contained $\mc{S}_\mr{approx}$ as a subset (the only $s$-spaces that are potentially deleted in our scheme are template $s$-spaces).

In order to extract a suitable counting result, one can (e.g.) use a recent result of Glock, Joos, Kim, K\"uhn, and Lichev \cite[Theorem~3.5]{GJKKL22} which proves there are at least 
\[\bigg((1-q^{-\Omega(n)})\frac{\qbinom{n-r}{s-r}_q}{\exp(\qbinom{s}{r}_q-1)}\bigg)^{\qbinom{n}{r}_q/\qbinom{s}{r}_q}\]
matchings with properties as in \cref{prop:approximate-covering}. Our result then proves that almost all such matchings can be extended in an least $1$ valid manner, and it is not hard to see that any final $(n,s,r)_q$-design thus constructed cannot be overcounted significantly: it can arise from at most $\sum_{j\le y}\binom{x}{j}$ choices of $\mc{S}_\mr{approx}$, where $x = \qbinom{n}{r}_q/\qbinom{s}{r}_q$ and $y\le q^{-\Omega(n)}\qbinom{n}{r}_q$.

For the upper bound, recall that an $(n,s,r)_q$-design can also be thought of as a hypergraph matching with vertex set $\mr{Gr}(V,r)$, where each $S\in\mr{Gr}(V,s)$ is thought of as a $\qbinom{s}{r}_q$-uniform edge consisting of its constituent $r$-subspaces. This is a regular hypergraph of degree $d=\qbinom{n-r}{s-r}_q=\Theta_{q,s}(q^{(s-r)n})$ and has a codegree bound of $q^{(s-r+1)n}$ since any two distinct $r$-spaces have at most $q^{(s-r-1)n}$ extensions to an $s$-space. The desired upper bound then follows directly from \cite[Theorem~3.1]{Lur17}. (The $q^{-\Omega(n)}$ savings in the logarithm of the upper bound is an immediate consequence of making the proof in \cite{Lur17} effective, so we omit the routine modification.)
\end{proof}

Finally, we sketch the necessary modifications in order to prove the case with general $\lambda$.
\begin{proof}[Proof sketch of \cref{thm:lambda}]
Let $J=\lambda\sum_{R\in\mr{Gr}(V,r)}e_R$. We define the template $\mc{S}_\mr{tem}$ exactly as in the proof of \cref{thm:main}. Note that the set of $s$-spaces within the template cover every $r$-space in $\mr{supp}(\partial_{s,r}\mc{S}_\mr{tem})$ exactly once. 

The crucial difference when $\lambda>1$ is that we now cover each $r$-space within $\mr{supp}(\partial_{s,r}\mc{S}_\mr{tem})$ an additional $\lambda-1$ times. In particular, we order the $r$-spaces in $\mr{supp}(\partial_{s,r}\mc{S}_\mr{tem})$ and extend each $r$-space $R$ to an $s$-space $S$ such that the remaining $r$-spaces in $\mr{Gr}(S,r)\setminus\{R\}$ are not within $\mr{supp}(\partial_{s,r}\mc{S}_\mr{tem})$ and are disjoint from the $r$-spaces used so far. This process is easily seen to run to completion via an analysis completely analogous to \cref{lem:master-disjointness} noting that the template itself is appropriately bounded and each $r$-space has at least $\Omega(q^{(s-r)n})$ extensions. Let the set of $s$-spaces chosen at this stage be denoted as $\mc{S}_\mr{tem}'$ and notice that
\[J-\partial_{s,r}(\mc{S}_\mr{tem}\cup\mc{S}_\mr{tem}')\in \{\lambda-1,\lambda\}^{\mr{Gr}(V,r)\setminus G_{\mr{tem}}}.\]
Furthermore one can prove that for each $(r-1)$-space all but an $\eta$-fraction of extensions to an $r$-space have coefficient $\lambda$ in $J-\partial_{s,r}(\mc{S}_\mr{tem}\cup\mc{S}_\mr{tem}')$, where $\eta$ is some polynomial growth function of $\tau$. Let $J_1^\ast$ be the subset of $\mr{Gr}(V,r)\setminus G_{\mr{tem}}$ where $J-\partial_{s,r}(\mc{S}_\mr{tem}\cup \mc{S}_\mr{tem}')$ has coefficient $\lambda$ and $J_2^\ast=\mr{Gr}(V,r)\setminus G_\mr{tem}$. The idea is to approximately cover $J_2^\ast$ a total of $\lambda-1$ times and approximately cover $J_1^\ast$ one time. However, these covers must also be disjoint.

We now perform this approximate covering. In order to do so, one can regularize the set of $s$-spaces supported on $J_1^\ast$ and $J_2^\ast$, respectively, by an easy alteration of the proof of \cref{lem:template-boosted}. We then apply the proof of \cref{prop:approximate-covering} to $J_1^\ast$ such that the remainder is appropriately bounded. We then consider $\lambda-1$ copies of $J_2^\ast$. Given what we have done to approximately cover the first $j$ copies of $J_2^\ast$ for some $j\in\{0,\ldots,\lambda-2\}$, we remove all $s$-spaces which have been used in approximately covering $J_1^\ast$ as well as the previous copies of $J_2^\ast$. Notice that the set of remaining $s$-spaces are still appropriately regular. Again applying the proof of \cref{prop:approximate-covering}, we can ultimately find a set of $s$-spaces $\mc{S}_\mr{approx}$ such that $\Phi_0 = \mc{S}_\mr{tem}\cup \mc{S}_\mr{tem}' \cup\mc{S}_\mr{approx}$ has the property that $J_1=J-\partial_{s,r}(\Phi_0)$ is appropriately bounded and such that it is supported outside $G_\mr{tem}$.

At this stage we cover down the remaining $r$-spaces (with appropriately multiplicity) in $J_1$ into the template in a disjoint manner and such that the spillover in the template is field bounded and field disjoint. This follows exactly as in \cref{lem:final-covering}; this gives a set of $s$-spaces $\Phi_2$ such that $J_2 = \partial_{s,r}\Phi_2-J_1 \in \{0,1\}^{\mr{G}_\mr{tem}}$ and is appropriately field-bounded and field-disjoint. At this point, the proof is identical to that of \cref{thm:main} and in particular we can find a decomposition of $J_2$ into signed $s$-cliques $\Phi_5$ and $\Phi_6$ such that $J_2 = \partial_{s,r}(\Phi_6-\Phi_5)$ and $\Phi_5,\Phi_6\in\{0,1\}^{\mr{Gr}_q(n,s)}$ with $\mr{supp}(\Phi_6)\subseteq\mc{S}_\mr{tem}$. The desired decomposition is then $(\Phi_0\setminus\Phi_6)\cup \Phi_2 \cup \Phi_5$.

To see that this decomposition does not use a given $s$-space more than once, notice that $\Phi_5,\mc{S}_\mr{tem}$ have all constituent $r$-spaces in $\mr{G}_\mr{tem}$ and these $r$-spaces are disjoint (so they are distinct), each $s$-space in $\mc{S}_\mr{tem}'$ has exactly $1$ $r$-space in $\mr{G}_{\mr{tem}}$ (and they are distinct by construction), each $s$-space in $\mc{S}_\mr{approx}$ has all $r$-spaces outside $\mr{G}_\mr{tem}$ (and they are distinct by construction), and all $s$-spaces in $\Phi_2$ have all but $1$ $r$-space in $\mr{G}_\mr{tem}$. Thus we have disjointness within these groups, and counting how many $r$-spaces lie in $G_\mr{tem}$ implies the different pieces use disjoint groups of $s$-spaces since $\qbinom{s}{r}_q\ge \qbinom{2}{1}_2 = 3$ so that $\qbinom{s}{r}_q-1>1$.
\end{proof}

\bibliographystyle{amsplain0.bst}
\bibliography{main.bib}

\end{document}